\documentclass{article}
\usepackage{hyperref}
\usepackage{xcolor}
\usepackage{amsmath}
\usepackage{amsthm}
\usepackage{amssymb}
\usepackage{enumitem}
\usepackage{indentfirst}
\numberwithin{equation}{section}
\newtheorem{theorem}{Theorem}[section]
\newtheorem{definition}{Definition}[]
\newtheorem{corollary}[theorem]{Corollary}

\newtheorem{lemma}[theorem]{Lemma}
\newtheorem{proposition}[theorem]{Proposition}

\newtheorem{remark}{Remark}[section]

\newenvironment{proof-th1}[1][Proof of Theorem \ref{thm:main}:\ ]{\medskip\noindent\textbf{\textit{#1}}}{\hfill$\Box$\medskip}

\newenvironment{proof-th2}[1][Proof of Theorem \ref{thm:pm}:\ ]{\medskip\noindent\textbf{\textit{#1}}}{\hfill$\Box$\medskip}
\newenvironment{proof-th}[1][Proof]{\medskip\noindent\textbf{\textit{#1}}}{\hfill$\Box$\medskip}

\newcommand{\R}{\mathbb{R}}

\title{Inhomogeneous nonlinear Schr\"odinger equations with competing singular nonlinearities}
\author{\textbf{Elisandra Gloss}\\
\small Departamento de Matem\'atica\\
\small Universidade Federal da Para\'iba\\
\small Cidade Universit\'aria, 58051-900, João Pessoa, Brazil\\
\small \textit{elisandra.gloss@academico.ufpb.br}\medskip\\\textbf{Kanishka Perera}\\
\small Department of Mathematics\\\small Florida Institute of Technology\\\small
150 W University Blvd, Melbourne, FL 32901-6975, USA\\\small
\textit{kperera@fit.edu}\medskip\\\textbf{Bruno Ribeiro}\\
\small Departamento de Matem\'atica\\
\small Universidade Federal da Para\'iba\\
\small Cidade Universit\'aria, 58051-900, João Pessoa, Brazil\\
\small \textit{bhcr@academico.ufpb.br}
}

\begin{document}

\maketitle
\date

\begin{abstract}

We study nonlinear elliptic equations that arise as stationary states of inhomogeneous nonlinear Schrödinger equations with competing singular nonlinearities. The model involves the Laplacian combined with weighted power-type terms and naturally leads to a variational formulation in a weighted Sobolev space obtained from the intersection of the homogeneous Sobolev space with a weighted Lebesgue space. Using sharp weighted Sobolev and Caffarelli--Kohn--Nirenberg type inequalities, we establish continuous and compact embeddings of this space into suitable weighted Lebesgue spaces. These embedding results, together with a natural scaling structure of the model, allow us to apply the abstract critical point framework of Mercuri and Perera (2026), yielding a sequence of nonlinear eigenvalues for the associated problem via a min--max scheme based on the Fadell--Rabinowitz cohomological index. Within this framework we obtain a broad collection of existence and multiplicity results for equations driven by sums of weighted power nonlinearities, covering  interactions in both subcritical and critical cases. We also establish a nonexistence result derived from a Pohozaev-type identity. Finally, we analyze the radial setting, where improved radial Caffarelli--Kohn--Nirenberg inequalities allow us to enlarge some of the admissible embedding ranges. This leads to strengthened radial versions of our main results.

\bigskip

\noindent\textbf{Keywords:} inhomogeneous nonlinear Schrödinger equations; singular weighted elliptic problems; nonlinear eigenvalues; variational methods.

\medskip

\noindent\textbf{MSC 2020:} \textit{Primary:} 35J60, 35J20; \textit{Secondary:} 35Q55, 35B33, 35P30, 58E05.

\end{abstract}

\tableofcontents

\newpage

\section{Introduction}

Elliptic equations with singular or varying coefficients appear naturally in the study of several physical models such as plasma physics, nonlinear optics, and Bose–Einstein condensation.  In these contexts, the local strength of the nonlinearity depends on the spatial position, leading to so-called \emph{inhomogeneous nonlinear Schr\"odinger equations} (INLS).  The time-dependent INLS
\[
i\partial_t u + \Delta u = \lambda |x|^{-a}|u|^{p-2}u + \mu |x|^{-b}|u|^{q-2}u,
\qquad (t,x)\in\mathbb{R}\times\mathbb{R}^N,
\]
exhibits rich dynamics involving the competition between two nonlinearities with distinct homogeneities.  Depending on the signs of the parameters, one nonlinearity may be focusing and the other defocusing, leading to the coexistence of scattering and blow-up phenomena, as shown for instance in \cite{Bellazinietall, Campos1, Farah-Guzman JDE, Farah-Guzman BBMS, GoMaSa2023, GoMaSa2025, HaLuWa2024, Tao-Visan-Zhang-2007}.  The search for {stationary wave} solutions of such evolution problems gives rise to a large family of nonlinear elliptic equations of Schr\"odinger type with singular or decaying weights.

The study of weighted nonlinear elliptic equations has a long tradition, deeply connected with the development of weighted Sobolev embeddings and Caffarelli–Kohn–Nirenberg (CKN) inequalities \cite{CKN-1984, Mallick-Nguyen-2023}.  The presence of weights such as $|x|^{-a}$ or $|x|^{-b}$ typically breaks translation and scaling invariance, which requires the analysis to be carried out in suitable weighted Sobolev spaces.  Early works addressing these difficulties include Wang and Willem \cite{Wang-Willem-2000JDE}, who considered singular minimization problems for the Laplacian with critical weighted nonlinearities that restore scaling invariance and proved existence of solutions via minimization of the best embedding constant.  Ambrosetti, Felli and Malchiodi \cite{Amb-Fel-Mal-2005} studied Schr\"odinger equations with potentials vanishing at infinity and established embeddings in weighted Sobolev spaces, obtaining ground state solutions for nonlinearities in the range $\sigma<p<2^*$, where $\sigma$ depends on the decay of the potential.  Related approaches closer to the problems treated in the present paper were developed by Sintzoff \cite{Sintzoff-Thesis, Sintzoff} and by Su, Wang and Willem \cite{Su-Wang-Willem-2007CCM}, who considered unbounded and decaying radial potentials, establishing compact weighted embeddings and the existence of solutions.  These results were later partially extended to the \(p\)-Laplacian setting by  Su and Tian~\cite{Su-Tian} and by Assunção, Carrião and Miyagaki \cite{Assu-Carr-Miyag-2006AML}, who studied \(p\)-Laplacian equations with singular weights, employing Caffarelli–Kohn–Nirenberg inequalities and a symmetric Mountain Pass framework to obtain infinitely many solutions.
Further developments on eigenvalue-type problems with singular potentials were obtained by Chabrowski and Costa \cite{Chab-Cos-2010}, who studied multiplicity and bifurcation near the first eigenvalue. Other contributions along these lines, considering bounded and unbounded domains,  include \cite{Baraket-Ghorbal-Fig, Bas-Miy-Vie-2014, Chen-Chen-Xiu-2013, Fig-Kia, Filippuccietal, Liang-Zhang, Liu-Zhao, Miy-San-Vie-2019, Terracini}.

In parallel, considerable progress has been made in the dispersive framework associated with the time-dependent INLS equation.  For the three-dimensional cubic focusing case, Farah and Guzmán \cite{Farah-Guzman JDE} obtained scattering of radial solutions below the ground-state threshold, later extending their results to higher dimensions in \cite{Farah-Guzman BBMS}.  Campos \cite{Campos1} and Hajaiej, Luo and Wang \cite{HaLuWa2024} refined these techniques for more general exponents and in the presence of inverse-square potentials, while Bellazzini, Dinh and Forcella \cite{Bellazinietall} employed the interaction Morawetz method to treat the nonradial problem with combined nonlinearities.  The recent papers by Gou, Majdoub and Saanouni \cite{GoMaSa2023, GoMaSa2025} have also studied INLS equations with competing inhomogeneous terms, both in the radial and nonradial settings.  These works highlight the role of the associated elliptic stationary equations in determining the threshold dynamics of the corresponding evolution problem.

Motivated by these results, we turn our attention to the elliptic stationary problem,
\begin{equation}\label{NE problem}
    -\Delta u + \frac{|u|^{q-2}u}{|x|^b} = \lambda\,\frac{|u|^{p-2}u}{|x|^a}
    \qquad\text{in }\mathbb{R}^N,
\end{equation}
where $0\le b<a<2$, $2<p<q$, and $N\ge2$.  This problem can be viewed as the search for standing wave profiles of the time-dependent INLS with two competing nonlinearities.

Throughout this work, we define, for any $0\le\eta<N$,
\[
2^*_\eta=\left\{\begin{aligned}&+\infty&\text{ if }&\quad N=2\\&\frac{2(N-\eta)}{N-2}&\text{ if }&\quad N\geq 3,
\end{aligned}
\right.
\]
and we assume $2<q<2^*_b$. This is equivalent to requiring that $N<2(q-b)/(q-2)$.
Let
$$
\delta=\frac{2-b}{q-2}.
$$
The associated scaling given by 
\begin{equation}
\begin{aligned}\label{scalingdef}
    u_t(x) = t^{\delta} u(tx),\,\,t\in[0,\infty),\,\,x\in\R^N
\end{aligned}\end{equation}
leaves the equation invariant when the parameters satisfy \begin{equation}\label{a def}
a=a(b,q,p):=2-\delta(p-2) = b+\delta(q-p)\in(b,2).
\end{equation} 

If \( u \) is a solution of equation \eqref{NE problem}, then the entire 1-parameter family of functions \( \{u_t : t \geq 0\} \) is a solution. Therefore, \eqref{NE problem} can be regarded as a nonlinear eigenvalue problem. In this work we exploit the theory developed in \cite{MePe2025} to study \eqref{NE problem} and related problems of the form
\begin{equation}\label{general problem}
    -\Delta u + \frac{|u|^{q-2}u}{|x|^b} = f(x,u) \quad \text{in } \mathbb{R}^N.
\end{equation}
To emphasize the interaction between $f$ and the nonlinear eigenvalues for \eqref{NE problem}, we usually set 
\begin{equation*}
    f(x,u)=\lambda \frac{|u|^{p-2}u}{|x|^a} + g(x, u)
\end{equation*}
where \( g \) is a Carathéodory function on \( \mathbb R^N \times \mathbb{R} \) satisfying a suitable growth condition. Throughout most of this paper, we work under the following basic assumptions on the parameters:
\begin{equation}\tag{$\mathcal H$}\label{basicparametershypothesis}
N \ge 2, 
\qquad 
0 \le b < 2 < p < q < 2_b^*, 
\qquad 
\text{and } a \text{ is given by } \eqref{a def}.
\end{equation}

By using sharp weighted embeddings of Caffarelli–Kohn–Nirenberg type and their 
radial refinements \cite{CKN-1984, Mallick-Nguyen-2023}, we obtain continuous and 
compact embeddings of the relevant weighted Sobolev spaces into suitable weighted 
Lebesgue spaces.  
These embeddings allow for a variational construction of nonlinear eigenpairs 
\((\lambda_k,u_k)\) for problem \eqref{NE problem} via a min–max scheme based on the 
cohomological index of Fadell and Rabinowitz \cite{fadel-rabi}.  
Our approach treats one of the nonlinear terms as part of the operator, extending the 
classical frameworks and permitting a significantly broader range of parameters and 
growth conditions.  

Within this setting, we carry out a comprehensive analysis of nonlinearities given by 
sums of weighted power-type terms, under assumptions dictated by their asymptotic 
behavior relative to the scaled nonlinearity and the sequence of eigenvalues.  
We establish existence and multiplicity results for \eqref{general problem} across all 
the corresponding regimes.  
The precise statements of our main theorems are presented in 
Subsection~\ref{subsectionmaintheorems}. It is worth mentioning that some of these results are also new in the case $b=0$, where the singular weights appear only in the nonlinearities and not in the operator itself. This situation is particularly interesting in the radial setting, where our approach provides new insights into such problems. See Section~\ref{sectionradial}.

\subsection{Variational framework}

We define the natural energy space to look for solutions to problems \eqref{NE problem} and \eqref{general problem}. 
For $N\geq2$, $r\geq1$ and $\eta\in[0,N)$ we denote by $L^r_\eta(\R^N)$ the weighted Lebesgue space $L^r(\R^N;|x|^{-\eta}dx)$. Clearly,  $L^r_0(\R^N)= L^r(\R^N)$. Recall that $L^r_\eta(\R^N)$ is a Banach space with the norm
\[
\|u\|_{L^r_\eta}=\left(\int_{\mathbb{R}^N} \frac{|u|^r}{|x|^\eta} dx \right)^\frac{1}{r}
\]
and it is uniformly convex if $r>1$.
Considering $b\in[0,N)$ and $q\geq1$, we define 
\[
E^q_b(\mathbb R^N)=\left\{u\in L^q_b(\R^N):\nabla u\in (L^2(\R^N))^N\right\},
\]
where $\nabla u = (u_{x_1},\cdots,u_{x_N})$ represents the weak derivatives of $u$, endowed with the norm
\begin{align*}
\|u\| &= \left[ \int_{\mathbb{R}^N} |\nabla u|^2\, dx + \left( \int_{\mathbb{R}^N} \frac{|u|^q}{|x|^b} dx \right)^\frac{2}{q} \right]^\frac{1}{2}.
\end{align*}
We readily see that $E_b^q(\mathbb{R}^N)$ (briefly denoted by $E_b^q$) is a Banach space. 
For $N \ge 3$, the space ${C_c^\infty(\mathbb{R}^N)}$ is dense in $E_b^q(\mathbb{R}^N)$, 
while for $N = 2$ this density holds provided that $q \ge b$. Since we could not locate a direct reference for this rather standard fact,  we include a complete proof in the Appendix (see Proposition~\ref{Density}). Moreover, for $q>1$, recalling that the product space $L^q_b(\R^N)\times (L^2(\R^N))^N$  is uniformly convex with
$$
\|(u,v_1,\cdots,v_N)\|_1=\left(\|u\|_{L^q_b}^2+\sum_{i=1}^N\|v_i\|_{L^2}^2\right)^\frac{1}2,
$$
since the operator $T:E^q_b\to L^q_b(\R^N)\times (L^2(\R^N))^N $
given by $T(u)=(u,\nabla u)$ is an isometry, the Banach space $E^q_b$ is also uniformly convex. In particular, it is reflexive when $q>1$. 

As usual, a key step in formulating the energy functional associated with our problem is 
to determine into which Lebesgue spaces the underlying Sobolev-type space embeds.  For $N\geq2$, $0\le b< 2<q<2^*_b$  and $0\le\eta<N$ let us define
\begin{equation}\label{def2bqeta}
2^*_{b,q,\eta}:=
\begin{cases}
   \vspace{.2cm}
\frac{q(N-\eta)}{N-b}&if \quad N\geq2\quad\text{and}\quad b\leq\eta,\\
  \frac{\eta\,q\,(N-2)+2N(b-\eta)}{b\,(N-2)}&if\quad N\geq3\quad\text{and}\quad 0\le\eta<b .  
\end{cases}
\end{equation}
Note that \(2^*_{b,q,\eta} \le \min\{q,2^*_\eta\}<2^*\) when \(b \le \eta\), while for \(N \ge 3\) and \(\eta \in(0,b)\) 
one has $q< 2^*_{b,q,\eta} < 2^*_{\eta}<2^*$ and for \(\eta=0 < b\) it holds $q< 2^*_{b,q,0} = 2^*_{0}=2^*$.

Under suitable conditions on \(\eta\), we shall prove that 
\(E_b^q\) embeds continuously (and, in some regimes, compactly) into 
\(L_\eta^r(\mathbb{R}^N)\) for $r$ lying between $2^*_{b,q,\eta}$ and $2^*_{\eta}$ with the endpoints included depending on some additional conditions. The precise hypotheses and the corresponding embedding results will be established in 
Section~\ref{secpreliminaries}.

Given a Carath\'eodory function $f:\R^N\times\R\to\R$, we say that $u\in E^q_b$ is a weak solution for the problem
\[
    -\Delta u + \frac{|u|^{q-2}u}{|x|^b} = f(x,u)
    \qquad\text{in }\mathbb{R}^N,
\]
if $f(\cdot,u)\in L^1_{loc}(\R^N)$ and 
\begin{equation*} 
\int_{\mathbb{R}^N}\nabla u\nabla vdx+\int_{\mathbb{R}^N}\frac{|u|^{q-2}uv}{|x|^{b}}dx = \int_{\mathbb{R}^N}f(x,u)v dx,\quad\text{for all}\ v\in C^\infty_c(\R^N).
\end{equation*}


Thus, the natural energy functionals associated with problems \eqref{NE problem} and 
\eqref{general problem} are $\Phi_\lambda,\Phi:E^q_b\to\R$ defined by
\begin{equation}\label{functionals} \begin{array}{c}\displaystyle\Phi_\lambda(u)=\frac{1}{2}\int_{\mathbb{R}^N}|\nabla u|^2dx+\frac{1}{q}\int_{\mathbb{R}^N}\frac{|u|^q}{|x|^{b}}dx - \frac{\lambda}{p}\int_{\mathbb{R}^N}\frac{|u|^{p}}{|x|^{a}}dx,\medskip\\ 
\displaystyle\Phi(u)=\frac{1}{2}\int_{\R^N}|\nabla u|^2dx+\frac{1}{q}\int_{\mathbb{R}^N} \frac{|u|^q}{|x|^b} dx
-\int_{\mathbb{R}^N} F(x,u)dx,\end{array}
\end{equation}
where \(F(x,s) = \int_0^s f(x,t)\,dt\).
The growth assumptions on \(F\) will be stated in terms of the embeddings introduced above and will ensure that \(\Phi \in C^1\).  Under the corresponding conditions on the pair \((a,p)\), the same embeddings also 
guarantee that \(\Phi_\lambda \in C^1\). Therefore, a function \(u \in E_b^q\) is a solution of \eqref{NE problem} (resp. \eqref{general problem}) if and only if it is a critical point of \(\Phi_\lambda\) (resp. \(\Phi\)).

 \subsection{Main theorems}\label{subsectionmaintheorems}

Before stating the main results, we establish some notations and definitions that will be used throughout this paper. 

Given the definition of $u_t$ presented in \eqref{scalingdef},  for $\eta\in[0,N)$ and $r\geq1$ we have
 \[
\int_{\mathbb{R}^N} \frac{|u_t|^r}{|x|^\eta} dx =t^{r\delta+\eta-N}\int_{\mathbb{R}^N} \frac{|u|^r}{|x|^\eta} dx,\quad\forall t\geq0,\, u\in L^r_\eta(\R^N).
 \]
Notice that $r\delta+\eta-N>0$ for any $r\geq 2^*_{b,q,\eta}$ (see \eqref{def2bqeta}). Since this change of variables is a common argument in this work, from now on we denote
 \begin{equation}\label{l(eta,r)}
\ell(\eta,r):=r\delta+\eta-N\qquad\text{and}\qquad \ell:=\ell(b,q)=\ell(a,p).
 \end{equation}
 Problem \eqref{NE problem} yields a natural classification of possibly different growth regimes for the nonlinearities that we will consider in problem \eqref{general problem}. 

\begin{definition}\label{defscaled}
Assume \eqref{basicparametershypothesis}.
We say that  the pair \((\eta,r)\) is \emph{admissible} if
\[
    0 \le \eta < \frac{N+2}{2}\ \ \big(b\le\eta<2,\text{ if }N=2\big)
    \quad \text{and}\quad
    r \in (\,2^*_{b,q,\eta},\; 2^*_{\eta}\,),\quad\text{with}\quad r> 1.
\]
 In addition, when \(\eta \le b\) we allow \(r = 2^*_{b,q,\eta}\), and if 
\(\eta \le 2\) and \(N \ge 3\) we include the second endpoint \(r = 2^*_\eta\). We classify an admissible pair \((\eta,r)\) in three types:
\begin{enumerate}
    \item[(i)] \textit{Subscaled,} if \(\ell(\eta,r)< \ell\);
    \item[(ii)] \textit{Scaled,} if \(\ell(\eta,r)= \ell\);
    \item[(iii)] \textit{Superscaled,} if \(\ell(\eta,r)> \ell\).
\end{enumerate}
\end{definition}
The concept of admissible pairs, together with the conditional inclusion of the 
endpoints of the above intervals, is motivated by the embeddings established in 
Lemma~\ref{Lemma Embedding with weight} and it also relates with the nonexistence result given in Proposition \ref{prop nonexistence}. The restriction $r>1$ is necessary for the $C^1$ regularity of the functionals associated to the problems. 

Notice that $(\eta,r)$ is subscaled (scaled, superscaled) if and only if
\[
r< q+\frac{(b-\eta)}{\delta}\quad \left(r= q+\frac{(b-\eta)}{\delta},\quad r> q+\frac{(b-\eta)}{\delta}\right).
\]
Thus, the existence of a subscaled (scaled) pair \((\eta,r)\) requires \(\eta >b\) (\(\eta \ge b\)), 
since  one has
\[
    q + \frac{b - \eta}{\delta} <\frac{q(N-\eta)}{N-b}\ \left(< 2^*_{b,q,\eta}\text{ for }N\ge3\right),\quad\text{for all}\quad \eta\in[0,b).
\]
In other words, when \(0 \le \eta < b\), all admissible pairs are necessarily 
superscaled. On the other hand, the condition \(\eta < 2\) is necessary for the existence of a 
superscaled pair, since
\[
    q + \frac{b - \eta}{\delta} \;\ge\; 2^*_{\eta}
    \qquad \text{whenever } \eta \ge 2.
\]
Moreover, although $0\le\eta<N$ would be a more natural assumption for the singularities, the initial restriction $\eta<(N+2)/2$ in this definition is to ensure that there will always exist a number $r>1$ in the range considered.

\medskip

We are now ready to state the main results of this paper. 
We begin by considering subcritical superscaled problems and establishing the existence of nontrivial solutions. Our first theorem reads as follows.

\begin{theorem}\label{Th-superscaled-subcritical}
Assume \eqref{basicparametershypothesis}. 
Suppose $f(x,t)=\lambda|x|^{-a}|t|^{p-2}t+|x|^{-\eta}h(t)$, with $\lambda\in\R$, $0<\eta<N-q(N-2)/2$ ($\eta\in(b,2)$ if $N=2$) and $h=H'$ for $H:\R\to\R$ a $C^1$ function such that
 \begin{align}\label{AR h}
 |h(t)|\leq c_1|t|^{r_1-1}+c_2|t|^{r_2-1}
 \quad\text{and}\quad
 c |t|^r\leq r H(t)\leq h(t)t,\quad\forall t\in\R,
 \end{align}
for some constants $c,c_1,c_2>0$, and superscaled pairs $(\eta,r_1)$, $(\eta,r_2)$ and $(\eta,r)$, with $2^*_{b,q,\eta}<r_1,r_2<2^*_\eta$ and $q<r<2^*_\eta$.
Then,  problem \eqref{general problem} has a nontrivial solution. If $\lambda$ is not an eigenvalue of \eqref{NE problem}, then \eqref{general problem} has a solution with positive energy.
\end{theorem}

\begin{theorem}\label{Th-superscaled-subcritical-multi}
Besides the conditions of Theorem \ref{Th-superscaled-subcritical} we assume that $h$ is an odd function.
Then,  problem \eqref{general problem} has a sequence $(u_k)$ of nontrivial solutions satisfying $\Phi(u_k)\nearrow\infty$.
\end{theorem}

In the next results we consider problem \eqref{general problem} with critical nonlinearities as 
\begin{equation}\label{f critical r2}
f(x,u)=\lambda \frac{|u|^{p-2}u}{|x|^a} + \frac{|u|^{2^*_{\eta_1}-2}u}{|x|^{\eta_1}}+ \mu\frac{|u|^{r_2-2}u}{|x|^{\eta_2}}
 \end{equation}
 for $\lambda\in\R$, $\mu\geq0$ and $N\ge3$.  Suppose that $\eta=\eta_1$  satisfies
\begin{align}\label{eta1 small}
&0\le\eta<N-q\frac{N-2}2.
\end{align} 
For $(\eta_2,r_2)$, we assume that it is a  superscaled pair satisfying one of the following conditions:
\begin{align}
&\eta_2\in(0,\eta_1)\quad \text{and}\quad 2^*_{\eta_1}\leq r_2< 2^*_{\eta_2}.\label{eta2 subcri1}\\
&\eta_2\in[\eta_1,2)\, \,(\eta_2=\eta_1\,\, \text{only if} \,\,\eta_1>b)\quad  \text{and} \quad r_2<2^*_{\eta_2},
\quad\text{with}\,\,
r_2> q+(2^*_{\eta_1}-q)\frac{(b-\eta_2)}{b-\eta_1}
\quad\text{if}\quad\eta_2< b. \label{eta2 subcri2}\\
&\eta=\eta_2\,\,\text{satisfies \eqref{eta1 small}}\quad\text{and}\quad r_2=2^*_{\eta_2}.\label{eta2 cri}
\end{align}
These results depend on the position of \(\lambda\) relative to the sequence $(\lambda_k)$ of eigenvalues of problem~\eqref{NE problem}, constructed via the Fadell–Rabinowitz cohomological index.  
See Theorem~\ref{lambdak} for  definition and properties of this sequence.

\begin{theorem}\label{th cri-super-mult}
Assume $N\geq3$, \eqref{basicparametershypothesis}. Suppose that for some $k,m\in\mathbb{N}$ it holds
\[
\lambda_k = \lambda_{k+1} = \dots = \lambda_{k+m-1} < \lambda_{k+m}.
\]
Consider $f$ as in \eqref{f critical r2}, with $\mu \ge 0$, $\eta_1$ satisfying \eqref{eta1 small} and $(\eta_2,r_2)$ superscaled satisfying one of the conditions \eqref{eta2 subcri1}-\eqref{eta2 cri}. Then,
there is a number $\delta_k > 0$ such that
problem~\eqref{general problem} possesses at least $m$ distinct pairs of nontrivial
solutions with positive energy whenever
\(\lambda \in \bigl(\lambda_k - \delta_k,\; \lambda_k\bigr).\)
\end{theorem}

As an immediate consequence of this theorem we get the following 
\begin{corollary}\label{corollary-th}
If $\lambda_k<\lambda_{k+1}$ for some $k\geq1$, then, for $f$ as in Theorem \ref{th cri-super-mult}, there is $\delta_k>0$ such that problem \eqref{general problem} has a pair of nontrivial solutions at a positive energy level for all $\lambda\in(\lambda_k-\delta_k,\lambda_k)$.
\end{corollary}

Instead of assuming that \(\lambda\) is close to \(\lambda_k\) in Theorem~\ref{th cri-super-mult}, one may alternatively require \(\mu\) to be sufficiently large.  In this case, however, the exponent \(r_2\) is necessarily restricted to the 
subcritical range.  We have the following result.

\begin{theorem}\label{th cri any lambda}
Assume $N\ge 3$ and \eqref{basicparametershypothesis}.  Consider $f$ as in \eqref{f critical r2}. Then, for every $\lambda\in\R$, $\eta_1$ satisfying \eqref{eta1 small},  $(\eta_2,r_2)$ superscaled satisfying one of the conditions \eqref{eta2 subcri1}-\eqref{eta2 subcri2}, and  $m \in\mathbb{N}$, there exists  $\mu_m>0$ such that, for all $\mu>\mu_m$,  
problem~\eqref{general problem} admits at least $m$ distinct pairs of nontrivial solutions with positive energy.  
In particular, the number of these solutions tends to infinity as $\mu\to\infty$.
\end{theorem}

From now on, let us turn our attention to nonlinearities that display at least a subscaled pair $(\eta,r)$ in terms of  Definition \ref{defscaled}. The first result in this setting concerns a single subscaled term and gives existence of infinitely many solutions on negative energy levels.

\begin{theorem}\label{Th-single-subscaled}
Assume \eqref{basicparametershypothesis}. For a subscaled pair $(\eta,r)$, 
consider 
$$
f(x,t)=\frac{|t|^{r-2}t}{|x|^\eta},\quad a.e.\, x\in\R^N, \forall t\in\R.
$$
  Then problem \eqref{general problem} has a sequence of solutions $(u_k)$ at negative levels such that $\Phi(u_k)\nearrow 0$.
\end{theorem}

For the next theorem, recall the notation 
$\ell(\eta,r)=r\delta+\eta-N$ and $ \ell=\ell(b,q)=\ell(a,p)$, given in \eqref{l(eta,r)}. 
\begin{theorem}\label{Th-subscased+scaled+superscaled1} 
 Assume \eqref{basicparametershypothesis}. Consider 
$$
f(x,t)=\lambda\frac{|t|^{p-2}t}{|x|^a}+\frac{|t|^{r-2}t}{|x|^\eta}-\frac{|t|^{{r_1}-2}t}{|x|^{\eta_1}},\quad a.e.\, x\in\R^N, \forall t\in\R,
$$
 where $(\eta,r)$ is subscaled. 
 Suppose also that $(\eta_1,r_1)$ is an admissible pair satisfying 
 $$
       \ell(\eta,r)<\ell(\eta_1,r_1) \quad\mbox{for}\,\,\lambda\leq0\quad\mbox{or}\quad
        (\eta_1,r_1)\text{ superscaled }\quad{for}\,\,\lambda>0.
   $$
   Then problem \eqref{general problem} has a sequence of solutions $(u_k)$ at negative levels such that  $\Phi(u_k)\nearrow 0$.
    
\end{theorem}

An interesting feature of the result above is that it also covers the case $b=\eta_1=0$, despite the fact that the embeddings $E^q_0\hookrightarrow L^{r_1}(\mathbb{R}^N)$ are not compact for any $q\le r_1\le 2^*$ (or $q\le r_1<\infty$ if $N=2$).

\medskip

The next theorem concerns asymptotically scaled nonlinearities.  It requires that \(\lambda\) is not an eigenvalue of \eqref{NE problem}; that is, \(\lambda\) must not only differ from the sequence of eigenvalues obtained in Theorem~\ref{lambdak}, but also be such that \eqref{NE problem} admits no nontrivial solution.  From Theorem~\ref{lambdak}, it follows immediately that there are no eigenvalues in the  interval \((-\infty,\lambda_1)\), where \(\lambda_1>0\) is defined in that theorem.

\begin{theorem}\label{Th-subscased+scaled} 
 Assume \eqref{basicparametershypothesis}. Consider 
$$
f(x,t)=\lambda\frac{|t|^{p-2}t}{|x|^a}+\frac{|t|^{r-2}t}{|x|^\eta},\quad a.e.\, x\in\R^N, \forall t\in\R,
$$
 where $(\eta,r)$ is subscaled. Suppose that $\lambda$ is not an eigenvalue for \eqref{NE problem}. Then problem \eqref{general problem} has a sequence of solutions $(u_k)$ at negative levels such that  $\Phi(u_k)\nearrow 0$.
\end{theorem}

Notice that this theorem includes 
Theorem~\ref{Th-single-subscaled} as the particular case \(\lambda = 0\).  
Indeed, if \(\lambda \le 0\), the associated functional is coercive, and the proof 
proceeds exactly as in Theorem~\ref{Th-single-subscaled}.  
The genuinely new aspect of the present theorem concerns the case \(\lambda > 0\), 
where coercivity is lost and the verification of the \((PS)\) condition, as well as the 
geometric properties of the functional, requires a different argument.

The next result still gives infinitely many solutions in negative levels, but now with a critical nonlinearity with positive sign, perturbed with a subscaled nonlinearity with a small parameter. 

\begin{theorem}\label{Th-subscaled+critical} 
 Assume $N\ge 3$ and \eqref{basicparametershypothesis}. Consider 
$$
f(x,t)=\mu\frac{|t|^{r-2}t}{|x|^\eta}+\frac{|t|^{2^*_{\eta_1}-2}t}{|x|^{\eta_1}},\quad a.e.\, x\in\R^N, \forall t\in\R,
$$
 where $(\eta,r)$ is subscaled and $0\le\eta_1<2$. Then there exists  $\mu^*>0$ such that problem \eqref{general problem} has a sequence of solutions $(u_k)$ at negative levels with  $\Phi(u_k)\nearrow 0$ for each $\mu\in(0,\mu^*)$.
\end{theorem}

Our final theorem establishes a multiplicity result for a critical problem involving both scaled and subscaled nonlinearities.  As in Theorem~\ref{th cri-super-mult} and Corollary~\ref{corollary-th}, its statement requires referring to the sequence of eigenvalues associated with problem~\eqref{NE problem}, whose existence is ensured by Theorem~\ref{lambdak}.

\begin{theorem}\label{Th-subscaled+scaled+critical} 
 Assume $N\ge 3$ and \eqref{basicparametershypothesis}. Consider 
$$
f(x,t)=\lambda\frac{|t|^{p-2}t}{|x|^a}-\mu\frac{|t|^{r-2}t}{|x|^\eta}+\frac{|t|^{2^*_{\eta_1}-2}t}{|x|^{\eta_1}},\quad a.e.\, x\in\R^N, \forall t\in\R,
$$
 where \((\eta,r)\) is subscaled and $\eta_1$ satisfies \eqref{eta1 small}.  
Suppose that \(\lambda > \lambda_k\). Then there exists \(\mu^*>0\) such that, for every  
\(\mu \in (0,\mu^*)\), problem~\eqref{general problem} admits $k$ pairs of nontrivial solutions at positive energy levels.
\end{theorem}

\subsection*{Organization of the paper}

The structure of the remainder of the paper is as follows: In Section~\ref{secpreliminaries} we give some preliminary results such as the embedding theorems that justify Definition \ref{defscaled}, and a Pohozaev Identity. Section~\ref{section SOp} is devoted to the abstract framework of Mercuri and Perera \cite{MePe2025}. We verify that our framework satisfy the structural hypotheses of the theory, and we construct the nonlinear eigenvalues $(\lambda_k)$ of problem~\eqref{NE problem} using the Fadell--Rabinowitz cohomological index. Section~\ref{sectionproofs} contains all proofs of the main results. We first establish general Palais–Smale and boundedness criteria. These preliminary tools are then used to treat, in a unified manner, some classes of nonlinearities considered in the paper: superscaled perturbations in the subcritical and critical cases; combinations of scaled, subscaled, and superscaled terms; and problems involving the nonlinear eigenvalues $(\lambda_k)$. Section~\ref{sectionradial} is devoted to the radial setting. 
Using the improved radial Caffarelli--Kohn--Nirenberg inequalities, we obtain strictly larger embedding ranges for the radial space $E^{q}_{b,\mathrm{rad}}$ by extending the admissible regimes for $(\eta,r)$. When $N\ge3$, the range of admissible exponents can be enlarged in the case $0\le \eta<b$. In dimension $N=2$, where the nonradial theory does not allow $\eta<b$, the radial embeddings yield a new interval of admissible exponents for $0\le \eta<b$. 
Moreover, in any dimension we may treat the case $\eta=b=0$, obtaining  subcritical compact embeddings that are not available in the nonradial setting. We then show how all main existence and multiplicity results admit strengthened radial versions under these improved embedding intervals. Finally, Section~\ref{sectionappendix} contains an appendix where we prove the density 
of smooth functions in the underlying functional space used throughout this paper.

\section{Preliminaries}\label{secpreliminaries}

In this section we establish several basic ingredients needed for our analysis.  
We begin with the embedding results that enable the use of variational methods in the proof of our main theorems. We then derive a technical regularity lemma and a Pohozaev identity,  which also play 
a crucial role in obtaining our results.

\subsection{Embeddings in weighted Lebesgue spaces}

Consider $b\in[0,N)$. Due to the Caffarelli - Kohn - Nirenberg inequality (see \cite{CKN-1984}), for fixed numbers
\begin{equation}\label{parameters}
r>0,\quad\theta\in[0,1],\quad q\geq1,\quad \sigma\in\mathbb{R},\quad\gamma=\theta\sigma-(1-\theta)\frac{b}{q}
\end{equation}
there exists $C>0$ such that
\begin{equation}\label{CKN ineq}
\||x|^\gamma u\|_r\leq C \|\nabla u\|^\theta_2 \| |x|^\frac{-b}{q} u \|_q^{1-\theta},\quad\forall u\in C^\infty_c(\R^N)
\end{equation}
if, and only if, the following occur
\begin{align}\label{dimensional balance}
0<\frac{1}{r}+\frac{\gamma}{N}= \theta\frac{N-2}{2N}+(1-\theta)\frac{N-b}{qN},
\end{align}
\begin{equation}\label{sigma's condition}
\sigma\leq0\,\,\,\text{if}\,\,\,\theta>0,\quad\text{and}\quad \sigma\geq -1\,\,\,\text{if}\,\,\,\theta>0\,\,\, \text{and}\,\,\, \frac{1}{r}+\frac{\gamma}{N}= \frac{N-2}{2N}.
\end{equation}
Moreover, $C$ is bounded if $\sigma\in[-1,0]$ and the other parameters are in a compact set.

 We shall also make use of the following proposition, due to Mallick and Nguyen~\cite{Mallick-Nguyen-2023}, which is stated below for convenience with the particular choice of parameters relevant to the present work.

 \begin{proposition}\cite[Proposition 24]{Mallick-Nguyen-2023}
\label{propMallickNguyen} Let $N \ge 1$, $q \ge 1$, $r \ge 1$, $0 < \theta < 1$, $0\le b<N$ and $\gamma \in \mathbb{R}$.
Define $\sigma$ by \eqref{parameters} and assume that $\sigma<0$, \eqref{dimensional balance} holds and
\[
   \frac{N-2}{2N} \ne \frac{N-b}{qN}.
\]
Suppose that the embedding
$W^{1,2}(B_1) \cap L^{q}(B_1)$ into $L^{r}(B_1)$ is compact.
Let $(u_n)_n \subset L^1_{\mathrm{loc}}(\mathbb{R}^N)$
 be a sequence of compact support functions such that 
$\bigl( \|\nabla u_n \|_2 \bigr)$
and $\bigl( \|\,|x|^{-\frac bq} u_n \|_q\bigr)$
are bounded.
Then $\bigl(|x|^{\gamma} u_n\bigr)$ has a convergent subsequence in $L^r(\mathbb R^N)$.
\end{proposition}

The next result combines the classical CKN inequality with some new insights in \cite{Mallick-Nguyen-2023} and brings some embedding results for the space $E^q_b$. Recal the definition of $2^*_{b,q,\eta}$ given in \eqref{def2bqeta}.

\begin{lemma}\label{Lemma Embedding with weight}
Suppose $N\ge 2$, let $0\le b<2<q<2^*_b$ and $0\le\eta<N$. For $N=2$, assume, additionally, that $b\leq\eta$. 
Fix $r>0$ such that
   $$
r\in \left(2^*_{b,q,\eta},2^*_\eta\right)
    $$
   and $\theta\in (0,1)$ such that 
   \begin{equation}\label{balance -eta/r}
   \frac{1}{r}=\theta\frac{N-2}{2(N-\eta)}+(1-\theta)\frac{N-b}{q(N-\eta)}.
   \end{equation}
In addition, if we assume $\eta\le b$ we can consider  $r=2^*_{b,q,\eta}$ and, if $\eta\le 2$ and $N\geq3$, the other endpoint $r=2^*_\eta$. 
Then, it holds
    \begin{equation}\label{embedding nonradial}
\left(\int_{\R^N}\frac{|u|^r}{|x|^\eta} dx\right)^\frac{1}{r}\leq C \left(\int_{\R^N}|\nabla u|^2dx\right)^\frac{\theta}{2} 
\left( \int_{\R^N}\frac{|u|^q}{|x|^b}dx\right)^\frac{1-\theta}{q},\quad\forall u\in E^q_b(\R^N).
\end{equation}
In particular, we have \(E^q_b \hookrightarrow L_\eta^r(\mathbb{R}^N)\) for every 
\(r \ge 1\) in the above interval. 
Moreover, for  \(N \ge 3\) with \(\eta > 0\), or  \(N = 2\) with \(\eta > b\), these embeddings are 
compact except at the endpoints of the interval.
\end{lemma}
\begin{proof}
Fix $r\in[2^*_{b,q,\eta},2^*_\eta]$
and choose \(\theta \in [0,1]\) so that \eqref{balance -eta/r} holds for this value of \(r\).
Consider $\gamma=-\eta/r$  so that \eqref{dimensional balance} can be written as \eqref{balance -eta/r}. 
Let $\sigma\in \R$ be given such that $\gamma=\theta \sigma-(1-\theta)b/q$, as in \eqref{parameters}. 
Then, if $\theta=0$  we see that this expression of $\gamma$ and \eqref{balance -eta/r} imply $\eta=b$ and $r=q$.
If $\theta>0$, 
due to \eqref{balance -eta/r} we get
\[
\sigma=\frac{-\eta}{\theta r}+\frac{(1-\theta)b}{\theta q}=-\frac{\eta(N-2)}{2(N-\eta)}-\frac{(1-\theta)N(\eta-b)}{\theta q(N-\eta)}.
\]
We observe that for $N=2$, $\theta\in(0,1)$ and $0\le \eta< N$ we get $sgn(\sigma)=sgn(b-\eta)$. This means that \eqref{sigma's condition} only holds if $b\le \eta $. 
Notice that for $\theta\in(0,1)$, we have $\sigma< 0$  if $b<\eta$ and $\sigma=0$ if $b=\eta$.  
      Then, for $\eta\geq b$, the CKN inequality \eqref{CKN ineq} is valid  with
    $$
   \frac{1}{r}=(1-\theta)\frac{N-b}{q(N-\eta)},
    $$
where $\theta\in(0,1)$ if $\eta>b$ or $\theta\in[0,1)$ if $\eta=b$. 
On the other hand, for $N\geq3$ we get
$\sigma\le 0$ for any $\theta\in(0,1]$ if $b\leq\eta$, while for $0\le\eta<b$ it holds $\sigma\le 0$ if, and only if, $\theta$ satisfies
$$
\theta\ge \theta_1:=\frac{2N(b-\eta)}{\eta q(N-2)+2N(b-\eta)}\in(0,1].
$$
Notice that, in this case $0\le\eta<b$, $\theta\ge \theta_1$ is equivalent to  $r\ge 2^*_{b,q,\eta}$.
Moreover, when  $\theta=1$ and  $\eta\leq2$ we get $\sigma\geq-1$.
So,  $\sigma$ satisfies \eqref{sigma's condition} and then \eqref{CKN ineq} occurs for $r$  between
    $2^*_{b,q,\eta}$ and $2^*_{\eta}$, including $2^*_{\eta}$ only in case $\eta\leq 2$, and including ${2^*_{b,q,\eta}}$ only in case $\eta\leq b$.
      Under these conditions, we conclude that \eqref{embedding nonradial} follows by density. 

It remains to prove that these embeddings are compact (except at the endpoints).
Let \((u_n)\) be a bounded sequence of functions in \(E^q_b\) with \(u_n \in C_c^\infty(\mathbb{R}^N)\).
This means that the sequences \((\|\nabla u_n\|_2)\) and \(\bigl(\|\,|x|^{-b/q} u_n \|_q\bigr)\) are bounded.
Then, if  $N\ge 3$ and $\eta>0$, or $N=2$ with $\eta>b$, fixing $r\geq1$ in $(2^*_{b,q,\eta},2^*_\eta)$ and setting \(\gamma = -\eta/r\), all the chosen parameters satisfy the assumptions of Proposition~\ref{propMallickNguyen}. Notice that the space \(W^{1,2}(B_1) \cap L^q(B_1)\) is compactly embedded into \(L^k(B_1)\) for every \(1 \le k < 2^*\) (or $k\ge 1$ if $N=2$),
and \(r\) lies within this range.
Hence, \(\bigl(|x|^{-\eta/r} u_n\bigr)\) admits a convergent subsequence in \(L^r(\mathbb{R}^N)\), which completes the proof.
\end{proof}

It is worth noting that the upper endpoint $r = 2^*_{\eta}$ of the embedding ranges obtained in Lemma~\ref{Lemma Embedding with weight} cannot be
improved. Indeed, the standard concentration argument applies: fix a nontrivial radial function supported in an annulus and consider the rescalings $u_\lambda(x)=\lambda^{(N-2)/2}u(\lambda x)$. This family is bounded in $E_b^q$, but no subsequence can converge in
$L^{2^*_\eta}_\eta(\mathbb{R}^N)$. Thus compactness fails precisely at the critical exponent $2^*_\eta$, as in the classical Sobolev setting. In turn, this failure of compactness shows that $2^*_\eta$ is also the \emph{optimal} endpoint for the continuous embedding, as expected.

In contrast, the lower threshold exponent $2^*_{b,q,\eta}$, which marks the beginning of the admissible regime, is not known to be optimal.
The scaling construction used above does not give non-compactness at this value, and whether continuity or compactness extends up to this lower
endpoint appears to remain an open question.

The corollary below shows the compactness of the embeddings of the underlying functional spaces associated with the nonlinear eingenvalue problem \eqref{NE problem}. This property is crucial in the development of the eigenvalue theory based on the $\mathbb Z_2$-cohomological index of Fadell and Rabinowitz \cite{fadel-rabi}.

\begin{corollary}\label{corollary_emb}
Assume \eqref{basicparametershypothesis}. 
 Then we have $E^q_b(\R^N)\subset L^p_a(\R^N)$ with compact embedding, for $a$ as defined in \eqref{a def}.
\end{corollary}
\begin{proof} 
For $N\geq2$ we verify that $2^*_{b,q,a}<p$ if, and only if, 
\[
q[q(N-2)-2(N-b)]< p[q(N-2)-2(N-b)].
\]
Moreover, for $N\geq3$ we get $p< 2^*_{a}$ if, and only if, 
\[
p[q(N-2)-2(N-b)]< 2[q(N-2)-2(N-b)].
\]
Thus, since $p\in(2,q)$ and $q<2^*_b$, we obtain  \(2^*_{b,q,a}<p<2^*_a .\) The result then follows from Lemma \ref{Lemma Embedding with weight}.
\end{proof}

Another useful compact embedding, which will be used throughout the paper, is given in the next corollary. Its proof follows directly from the definition  of an admissible subscaled pair $(\eta,r)$ (Definition~\ref{defscaled}) together with Lemma~\ref{Lemma Embedding with weight}.

\begin{corollary}\label{subscaled compact}
Assume \eqref{basicparametershypothesis}. If $(\eta,r)$ is a subscaled pair, then the embedding 
\(
E^q_b(\mathbb{R}^N)\hookrightarrow L^r_\eta(\mathbb{R}^N)
\)
is compact.
\end{corollary}

\subsection{Pohozaev Identity and nonexistence results}

In this section we establish a Pohozaev identity that plays a crucial role in our results.  
The proof follows the ideas of \cite{Terracini}.  
This identity also yields conditions ensuring nonexistence of solutions to 
problem~\eqref{general problem}.

\begin{proposition}\label{Prop Pohozaev Identity}
Suppose $N\ge 2$ and let $0\le b<2<q<2^*_b$.
Let $u\in E^q_b$ be a weak solution for \eqref{general problem} where $f:\R^N\times\R\to\R$ is a Carathéodory function given by
\[
f(x,t)=\sum_{j=1}^mc_j\frac{|t|^{r_j-2}t}{|x|^{\eta_j}}
\]
for some $\eta_j\in[0,2)$ and $1< r_j\leq 2^*_{\eta_j}$ ($r_j\in(1,\infty)$ if $N=2$). 
Then $u\in C^2(\R^N\backslash\{0\})$ and, if $f(\cdot,u)u\in L^1(\mathbb{R}^N)$, it  satisfies the Pohozaev type identity
\begin{equation}\label{Pohozaev identity}
\frac{N-2}{2}\int_{\R^N}|\nabla u|^2dx+\frac{N-b}{q}\int_{\mathbb{R}^N} \frac{|u|^q}{|x|^b} dx
=\sum_{j=1}^m\frac{c_j(N-\eta_j)}{r_j}\int_{\mathbb{R}^N} \frac{|u|^{r_j}}{|x|^{\eta_j}} dx.
\end{equation}
\end{proposition} 

\begin{proof}
Let \(u \in E_b^q\) be a weak solution of \eqref{general problem}.  
We first show that \(u \in C^2(\mathbb{R}^N \setminus \{0\})\).  
Considering
\(h(x)=f(x,u(x))-|x|^{-b}{|u(x)|^{q-2}u(x)},\) we verify that  
\(h \in L^r_{\mathrm{loc}}(\mathbb{R}^N\setminus\{0\})\) for every  
\(r \ge 1\).  
When \(N=2\), this follows directly from the embeddings in 
Lemma~\ref{Lemma Embedding with weight};  
when \(N \ge 3\), it follows from Lemma~B.3 in \cite{struwe}.  
Standard Calderón--Zygmund estimates then give  
\(u \in W^{2,r}_{\mathrm{loc}}(\mathbb{R}^N\setminus\{0\})\) for all \(r\ge1\),  
and Schauder estimates yield
\(
u \in C^2(\mathbb{R}^N \setminus \{0\}).
\)

To prove the Pohozaev identity, we fix \(0<\rho<R\) and set \(A_{\rho,R}=\{x\in\mathbb{R}^N: \rho<|x|<R\}\).
Multiplying the equation \eqref{general problem} by \(x\cdot\nabla u\) and integrating over \(A_{\rho,R}\),
an application of the Green's identity gives
\[
\frac{N-2}{2}\int_{A_{\rho,R}} |\nabla u|^2 dx
+\frac{N-b}{q} \int_{A_{\rho,R}}\frac{|u|^q}{|x|^b}dx
=\sum_{j=1}^m c_j\,\frac{N-\eta_j}{r_j}
   \int_{A_{\rho,R}}\frac{|u|^{r_j}}{|x|^{\eta_j}}dx
   + \beta(R)-\beta(\rho),
\]
where
\[
\beta(t)=\frac{t}2\int_{\partial B_t}|\nabla u|^2dS-t\int_{\partial B_t}|\nabla u\cdot\nu|^2dS
+\frac{t}{q}\int_{\partial B_t}\frac{|u |^{q}}{|x|^{b}}dS-\sum_{j=1}^mc_j\frac{t}{r_j}\int_{\partial B_t}\frac{|u |^{r_j}}{|x|^{\eta_j}}dS.
\]
Since all integrands belong to \(L^1(\mathbb{R}^N)\), we may choose sequences
\(\rho_n\to 0^+\) and \(R_n\to\infty\) such that \(\beta(\rho_n)\to0\) and \(\beta(R_n)\to0\).  
Letting \(n\to\infty\) yields the identity \eqref{Pohozaev identity}.
\end{proof}

This Pohozaev identity leads to the following nonexistence result, which is 
similar to \cite[Theorem~9]{Sintzoff} (see also \cite[Theorem~4.11]{Pucci-Servadei}).
\begin{proposition}\label{prop nonexistence}
Assume $0\le b<2<q<2^*_b$ and consider $\eta\in [0,2)$. Let $u\in E^q_b\cap L^r_\eta(\mathbb R^N)$ be a weak solution of the problem
\begin{equation}\label{problem non}
    -\Delta u + \frac{|u|^{q-2}u}{|x|^b} = \frac{|u|^{r-2}u}{|x|^\eta}
    \qquad\text{in }\mathbb{R}^N.
\end{equation}
Assume that $u\in L^\infty_{loc}(\mathbb R^N\backslash\{0\})$ if $r>2^*$. Then $u\equiv 0$ if one of the following conditions occur: 
\begin{enumerate}
\item $r\ge 2^*_\eta$ or $1<r\le q(N-\eta)/(N-b)$, if $N\ge 3$;
\item $1<r\leq q(N-\eta)/(N-b)$, if $N=2$.
\end{enumerate}
\end{proposition}
\begin{proof}
As in the proof of Proposition \ref{Prop Pohozaev Identity} we see that $u\in C^2(\R^N\backslash\{0\})$ if $r\leq 2^*$. If  $r> 2^*$ and we assume $u\in L^\infty_{loc}(\mathbb R^N\backslash\{0\})$ then we reach the same regularity. So, the Pohozaev identity \eqref{Pohozaev identity} holds for  $u$, which means that 
  \begin{equation*}
\frac{N-2}{2}\int_{\R^N}|\nabla u|^2dx+\frac{N-b}{q}\int_{\mathbb{R}^N} \frac{|u|^q}{|x|^b} dx
=\frac{N-\eta}{r}\int_{\mathbb{R}^N} \frac{|u|^{r}}{|x|^{\eta}} dx.
\end{equation*}
On the other hand, from \eqref{problem non} we get
\begin{equation*}
\frac{N-\eta}{r}\int_{\R^N}|\nabla u|^2dx+\frac{N-\eta}{r}\int_{\mathbb{R}^N} \frac{|u|^q}{|x|^b} dx
=\frac{N-\eta}{r}\int_{\mathbb{R}^N} \frac{|u|^{r}}{|x|^{\eta}} dx.
\end{equation*}
Subtracting one equation from another we obtain
\begin{equation*}
\left(\frac{N-2}{2}-\frac{N-\eta}{r}\right)\int_{\R^N}|\nabla u|^2dx+\left(\frac{N-b}{q}-\frac{N-\eta}{r}\right)\int_{\mathbb{R}^N} \frac{|u|^q}{|x|^b} dx=0.
\end{equation*}
Therefore, from assumptions \textit{1.} or \textit{2.} it follows that $u=0$.
\end{proof}
Notice that Lemma \ref{Lemma Embedding with weight} and Proposition \ref{prop nonexistence} show  that $2_\eta^{*}$ is a critical threshold for the embeddings and also that it plays the same role for the existence of nontrivial solutions, which is the expected behavior. Even more interesting is the fact that the number $q (N-\eta)/(N-b)$ also emerges as a threshold for existence. Observe that when $\eta \ge b$ we have
that the critical lower endpoint for the embeddings in Lemma~\ref{Lemma Embedding with weight} coincides with the corresponding nonexistence threshold. In contrast, when $\eta < b$ there is a gap 
\[
\frac{q (N-\eta)}{N-b} < r < 2^{*}_{b,q,\eta}.
\]
This gap is shortened in the radial case (see Section \ref{sectionradial}).

\section{Scaled operators on $E^q_b(\mathbb R^N)$}\label{section SOp}

Consider the definition of a scaling on a Banach space, as given in \cite[Definition 1.1]{MePe2025}.

\begin{definition}
    Let $W$ be a reflexive Banach space. A scaling on $W$ is a continuous mapping
$$W\times [0,\infty) \to W,\quad (u,t) \mapsto u_t$$ 
satisfying
\begin{enumerate}[label={($H_\arabic*$)},  
                  ref=$H_\arabic*$]               
\item\label{H1} $(u_{t_1})_{t_2}=u_{t_1t_2}$ for all $u\in W$ and $t_1,t_2\geq0$,
\item
$(\tau u)_{t}=\tau u_{t}$ for all $u\in W$, $\tau\in\mathbb{R}$ and $t\geq0$,
\item\label{H3} $u_{0}=0$ and $u_{1}=u$ for all $u\in W$, 
    \item\label{H4} $u_{t}$ is bounded on bounded sets in $W\times [0,\infty)$, 
  \item\label{H5} $\exists \sigma >0$ such that $\|u_t\|=O(t^{\,\sigma})$ as $t\to\infty$,  uniformly in $u$ on bounded sets.
\end{enumerate}
\end{definition}
We have the following result.
\begin{lemma} Suppose $N\ge 2$ and let $0\le b<2<q<2^*_b$. Consider $u_t$ as defined in \eqref{scalingdef}. The map $E^q_b\times [0,\infty)\to E^q_b$ given by $(u,t)\mapsto u_t$
is a scaling on $E^q_b$.
\end{lemma}
 \begin{proof}  
 It is straightforward to verify that \eqref{H1}-\eqref{H3} hold. To prove \eqref{H4} and \eqref{H5}, recall the definition of $\ell=(q\delta + b)-N$, given in \eqref{l(eta,r)}. We have
$$
\|\nabla u_t\|_{2}^2=t^{\,\ell}\|\nabla u\|_{2}^2\quad\text{and}\quad  \int_{\mathbb{R}^N} \frac{|u_t|^q}{|x|^b} dx =t^{\,\ell} \int_{\mathbb{R}^N} \frac{|u|^q}{|x|^b} dx
$$
for all $u\in E^q_b$ and $t\geq0$. Therefore
\begin{equation}\label{normscaledineq}
\|u_t\|\leq\max\{t^{\,\ell/2},t^{\,\ell/q}\}\|u\|,\quad\forall (u,t)\in E^q_b\times [0,\infty),
\end{equation}
which implies that \eqref{H4} and \eqref{H5} are satisfied with $\sigma=\ell$. Continuity follows then as in \cite[Lemma 3.1]{MePe2025} as a result of the density of $C^\infty_c(\mathbb{R}^N)$ in $E^q_b$. 
\end{proof}

Let $\tilde{h}\in C(E^q_b,(E^q_b)^*)$. We say that $\tilde{h}$ is a \emph{potential operator} if there exists a functional $\tilde{H}\in C^1(E^q_b,\mathbb{R})$, called a \emph{potential} of $\tilde{h}$, such that $\tilde{h}=\tilde{H}'$. Without loss of generality, we may take $\tilde{H}(0)=0$ (if necessary, by considering $\tilde{H}-\tilde{H}(0)$).

An odd potential operator $\tilde{h}$ is said to be a \emph{scaled operator} provided it maps bounded subsets of $E^q_b$ into bounded subsets of $(E^q_b)^*$ and, for $\sigma$ as in \eqref{H5}, it satisfies
\begin{equation}\label{scaledoperator}
    \tilde{h}(u_t)v_t = t^{\,\sigma}\tilde{h}(u)v \quad\text{for all }u,v\in E^q_b\text{ and }t\ge 0.
\end{equation}

According to \cite[Proposition~2.2]{MePe2025}, whenever $\tilde{h}$ is a potential operator, the corresponding potential $\tilde{H}$ with $\tilde{H}(0)=0$ can be expressed as
\begin{equation*}
    \tilde{H}(u) = \int_{0}^{1}\tilde{h}(\tau u)\,u\,d\tau, \qquad u\in E^q_b.
\end{equation*}
It follows from this representation that $\tilde{H}$ is even whenever $\tilde{h}$ is odd, and that $\tilde{H}$ maps bounded sets to bounded sets if the same holds for $\tilde{h}$. In addition, when $\tilde{h}$ satisfies the scaling condition \eqref{scaledoperator}, the potential $\tilde{H}$ inherits the same homogeneity property:
\[
\tilde{H}(u_t)=t^{\,\sigma}\tilde{H}(u), \qquad u\in E^q_b,\, t\ge0.
\]

With the above notions at hand, we now introduce the odd potential operators
\[
\mathcal{A},\,\mathcal{B}:E^q_b\to (E^q_b)^*,
\]
defined respectively by
\begin{equation}\label{defPotAB}
\mathcal{A}(u):=-\Delta u+\frac{|u|^{q-2}u}{|x|^b}
\quad\text{and}\quad
\mathcal{B}(u):=\frac{|u|^{p-2}u}{|x|^a}.
\end{equation}
That means
\begin{equation*}
    \mathcal{A}(u)v
    =\int_{\mathbb{R}^N}\nabla u\nabla v\,dx
     +\int_{\mathbb{R}^N}\frac{|u|^{q-2}u\,v}{|x|^b}\,dx,
\end{equation*}
and
\begin{equation*}
    \mathcal{B}(u)v
    =\int_{\mathbb{R}^N}\frac{|u|^{p-2}u\,v}{|x|^a}\,dx.
\end{equation*}
Here $2<p<q$ and the parameters satisfy $0\le b<a<2$, as specified in \eqref{a def}.  Both $\mathcal{A}$ and $\mathcal{B}$ fulfill the scaling condition \eqref{scaledoperator} with  $\sigma=\ell$,
which is consistent with assumption \eqref{H5}. Furthermore, the continuous embedding 
\[
E^q_b \hookrightarrow L^p_a(\mathbb{R}^N),
\]
provided by Corollary~\ref{corollary_emb}, ensures that $\mathcal{B}$ maps bounded sets in $E^q_b$ to bounded sets in $(E^q_b)^*$. The same property is immediate for $\mathcal{A}$. Consequently, both $\mathcal{A}$ and $\mathcal{B}$ are \emph{scaled potential operators}.

 Notice that if $u\in E^q_b$ satisfies $\mathcal{A}(u)=\lambda\,\mathcal{B}(u)$, that is, if $u$ is a solution of problem \eqref{NE problem}, then, by the scaling properties of $\mathcal{A}$ and $\mathcal{B}$, its rescaled function $u_t$ is also a solution for every $t\ge0$. Associated with the operators $\mathcal{A}$ and $\mathcal{B}$, we introduce the corresponding potentials 
\(I,\,J:E^q_b\rightarrow\mathbb{R},\)
defined by
\begin{equation}\label{defPotIJ}
I(u)
=\frac{1}{2}\int_{\mathbb{R}^N}|\nabla u|^{2}\,dx
+\frac{1}{q}\int_{\mathbb{R}^N}\frac{|u|^{q}}{|x|^{b}}\,dx,
\qquad
J(u)
=\frac{1}{p}\int_{\mathbb{R}^N}\frac{|u|^{p}}{|x|^{a}}\,dx.
\end{equation}

In order to apply the abstract critical point framework developed in \cite{MePe2025}, 
we must verify that the operators $\mathcal{A}$ and $\mathcal{B}$, as well as their corresponding potentials $I$ and $J$, satisfy the structural properties required in that setting. 
They are the following:
\begin{enumerate}[label={($H_{\arabic*}$)},  
                  ref=$H_{\arabic*}$,start=6]               
    \item\label{H6} $\mathcal{A}(u)u>0$ for all $u\in E^q_b\backslash\{0\}.$
    \item\label{H7}  If $u_n\rightharpoonup u$ in $E$ and $\mathcal{A}(u_n)(u_n-u)\to0$ then $(u_n)$ has a subsequence that converges strongly to $u$.
    \item\label{H8} $\mathcal{B}(u)u>0$ for all $u\in E^q_b\backslash\{0\}.$
    \item\label{H9} If $u_n\rightharpoonup u$ in $E^q_b$ then $\mathcal{B}(u_n)\to \mathcal{B}(u)$ in $(E^q_b)^*$.
    \item\label{H10} $I$ is coercive, i.e., $I(u)\to\infty$ as $\|u\|\to\infty.$ 
    \item\label{H11} For each $u\in E^q_b\backslash\{0\}$ there is a unique $t>0$ such that $I(tu) = 1.$
    \item\label{H12} Every solution $u$ of equation \eqref{NE problem} satisfies $I(u)=\lambda J(u).$
\end{enumerate}

We check these conditions in the following lemma.
\begin{lemma}
Assume \eqref{basicparametershypothesis}. Let $\mathcal A,\mathcal B$ and $I$, $J$ as defined in \eqref{defPotAB} and \eqref{defPotIJ}. Then, the conditions \eqref{H6}-\eqref{H12} hold true.
\end{lemma}
\begin{proof}
Conditions \eqref{H6}, \eqref{H8}, and \eqref{H10} are immediate.  Moreover, \eqref{H9} holds because of the compact embedding
\(
E^q_b \hookrightarrow L^p_a(\mathbb{R}^N)
\quad\text{(Corollary~\ref{corollary_emb})}
\). Then, since $I(0)=0$ and, for every fixed $u\in E^q_b\setminus\{0\}$, the function $t\mapsto I(tu)$ is strictly increasing for $t\ge0$, \eqref{H11} also holds. It remains to verify \eqref{H7} and \eqref{H12}. To prove that \eqref{H7} occurs, assume that $u_n\rightharpoonup u$ in $E^q_b$ and $\mathcal{A}(u_n)(u_n-u)\to0$. 
  Note that
   $$
  \mathcal{A}(u_n)(u_n-u)= \int_{\mathbb{R}^N}\nabla u_n\nabla(u_n-u)dx +\int_{\mathbb{R}^N}\frac{|u_n|^{q-2}u_n(u_n-u)}{|x|^b}\,dx.
   $$
   Since $u_n\rightharpoonup u$ in $E^q_b$, we also have $\nabla u_n\rightharpoonup \nabla u$ weakly in $L^2(\mathbb{R}^N)$. This yields
    \begin{equation*}
 \int_{\mathbb{R}^N}\nabla u_n\nabla(u_n-u)dx= \|\nabla(u_n-u)\|_{2}^2+o_n(1).
   \end{equation*}
Moreover, since the embedding $E^q_b\hookrightarrow L^q_b(\mathbb R^N)$ is continuous (Lemma \ref{Lemma Embedding with weight}), we must have $u_n\rightharpoonup u$ in $L^q_b(\mathbb R^N)$ as well. {So, a standard duality argument gives that 
$$
\int_{\mathbb{R}^N}\frac{|u_n|^{q-2}u_nu}{|x|^b}\,dx\rightarrow \int_{\mathbb{R}^N}\frac{|u|^{q}}{|x|^b}\,dx,
$$
 up to a subsequence.}
Then, the Brezis-Lieb lemma (with weights) gives
$$
\int_{\mathbb{R}^N}\frac{|u_n|^{q-2}u_n(u_n-u)}{|x|^b}\,dx=\int_{\mathbb{R}^N}\frac{|u_n-u|^q}{|x|^{b}}dx+o_n(1).
$$
Thus, up to a subsequence, we have
$$
\mathcal{A}(u_n)(u_n-u)= \|\nabla(u_n-u)\|_{2}^2+ \int_{\mathbb{R}^N}\frac{|u_n-u|^q}{|x|^{b}}dx+o_n(1),
$$
which concludes the proof of \eqref{H7}.

Finally, suppose that $u$ is a nontrivial solution of $\eqref{NE problem}$. Then,  $\mathcal{A}(u)=\lambda \mathcal{B}(u)$ (notice that, due to \eqref{H6} and \eqref{H8}, we have $\lambda>0$). Then, using $u$ as a test function we have
\begin{equation}\label{uastestfunc}
\int_{\R^N}|\nabla u|^2dx+\int_{\mathbb{R}^N} \frac{|u|^q}{|x|^b} dx
=\lambda\int_{\mathbb{R}^N} \frac{|u|^p}{|x|^a} dx.
\end{equation}
Moreover, the Pohozaev Identity \eqref{Pohozaev identity} gives
\begin{equation}\label{Pohozaev_1.1}
\frac{N-2}{2}\int_{\R^N}|\nabla u|^2dx+\frac{N-b}{q}\int_{\mathbb{R}^N} \frac{|u|^q}{|x|^b} dx
=\frac{\lambda(N-a)}{p}\int_{\mathbb{R}^N} \frac{|u|^p}{|x|^a} dx.
\end{equation}
Multiply \eqref{Pohozaev_1.1} by $q-2$ and \eqref{uastestfunc} by $2-b$ and subtract one from another. Then, divide the resulting equation by $2(N-b)-q(N-2)$, which is nonzero, and we reach \eqref{H12}.
\end{proof}


\subsection{Scaled eigenvalue problem}

We continue to explore the scaled eigenvalue problem $\mathcal{A}=\lambda \mathcal{B}$, i.e. problem \eqref{NE problem} with $a>0$ as defined in \eqref{a def}. Whenever \eqref{NE problem} admits a nontrivial solution \(u\in E^q_b\), 
we refer to the corresponding parameter \(\lambda\) as an \emph{eigenvalue}, 
and to \(u\) as an \emph{eigenfunction} associated with \(\lambda\). 
It was shown above that the hypotheses \eqref{H6}--\eqref{H12} 
are fulfilled by the scaled operators \(\mathcal{A}\) and \(\mathcal{B}\), 
together with their respective potentials \(I\) and \(J\), as given in \eqref{defPotIJ}. 
We denote by \(\sigma(\mathcal{A},\mathcal{B})\) the set of all eigenvalues 
of problem \eqref{NE problem}, and call it the \emph{spectrum} 
of the pair of scaled operators \((\mathcal{A},\mathcal{B})\). 
From \eqref{H6} and \eqref{H8}, it follows that 
\(
\sigma(\mathcal{A},\mathcal{B}) \subset (0,\infty).
\) We set
$$
\mathcal{M}=\{u\in E^q_b: I(u)=1\},
$$
and define the projection
$
\pi:E^q_b\backslash\{0\}\to\mathcal{M}
$
as
\begin{equation}\label{def pi}
    \pi(u)=u_{t_u},\quad t_u=\left({I(u)}\right)^{-1/\ell}
\end{equation}
where $u_t$ is the scaling defined in \eqref{scalingdef} and $\ell$ is given in \eqref{l(eta,r)}.  We also define
\begin{equation}\label{def Psi}
\Psi(u)=\frac{1}{J(u)},\quad u\in E^ q_b\backslash\{0\},\quad\mbox{and}\quad \tilde{\Psi}=\Psi_{|_\mathcal{M}}.
\end{equation}
We observe that \(\mathcal{M}\) is a complete, symmetric, and bounded 
\(C^{1}\)-Finsler manifold, and that the eigenvalues of problem 
\eqref{NE problem} coincide with the critical values of 
\(\Tilde{\Psi}\), as established in \cite[Proposition~2.5]{MePe2025}. 
Let \(\mathcal{F}\) denote the family of symmetric subsets of 
\(\mathcal{M}\), and let \(i(Y)\) represent the cohomological index of 
\(Y\in\mathcal{F}\) in the sense of Fadell and Rabinowitz 
(see \cite{fadel-rabi}). For each \(k\ge1\), define
\[
\mathcal{F}_k = \{\,Y\in\mathcal{F} : i(Y)\ge k\,\},
\]
and set
\[
\lambda_k := \inf_{Y\in\mathcal{F}_k}\sup_{u\in Y}\Psi(u).
\]
The following theorem follows directly from an abstract variational 
result proved in \cite{MePe2025}. It guarantees that the sequence 
\((\lambda_k)\) consists of \emph{nonlinear eigenvalues} of 
\eqref{NE problem} and provides additional properties, 
which are stated below.
\begin{theorem}[\cite{MePe2025}, Theorem 1.3]\label{lambdak}
Assume \eqref{basicparametershypothesis}. Then $\lambda_k\nearrow \infty$ is a sequence of eigenvalues of \eqref{NE problem}. Moreover,
\begin{enumerate}[label={(\roman*)},  
                  ref=(\roman*)]               
\item\label{lambdak(i)} The first eigenvalue is given by
$$
\lambda_1=\min_{u\in\mathcal{M}}\tilde{\Psi}(u)>0.
$$
\item
If $\lambda_k = \cdots = \lambda_{k+m-1} = \lambda$, then $i(E_\lambda)\geq m$, where $E_\lambda$ is the set of eigenfunctions associated with $\lambda$ that lie on $\mathcal{M}$.
\item\label{lambdak(iii)} If $\lambda_k<\lambda<\lambda_{k+1}$, then
$$i(\tilde{\Psi}^{\lambda_k} ) = i(\mathcal{M}\backslash\tilde{\Psi}_{\lambda}) = i(\tilde{\Psi}^{\lambda} ) 
= i(\mathcal{M}\backslash \tilde{\Psi}_{\lambda_{k+1}}) = k,$$
where 
$\tilde{\Psi}^\alpha=\{ u\in \mathcal{M} : \tilde{\Psi}(u) \leq \alpha\}$ and 
$\tilde{\Psi}_\alpha = \{ u\in \mathcal{M} : \tilde{\Psi}(u) \geq \alpha\}$ for $\alpha\in\mathbb{R}$.
\end{enumerate}
\end{theorem}

\section{Proof of the main theorems}\label{sectionproofs}

Our existence and multiplicity results for problem~\eqref{general problem} will be obtained by means of critical point theorems, primarily those in \cite{MePe2025} together with an abstract result from \cite{Perera-book}. For the reader’s convenience, several of these tools are recalled below  (see Propositions~\ref{Theorem 2.33 MP}, \ref{critical point result} and Corollary~\ref{Corollary 2.34 MP}).  

As is standard in critical point theory, compactness assumptions such as the Palais--Smale condition, either in its global or localized form, are required for the application of these results.  We begin by establishing that such compactness properties indeed hold for the functionals considered in this work.  We then prove the main theorems, organizing the presentation according to the growth regimes involved: subscaled, superscaled, and subcritical or critical, as well as combinations of these settings.

\subsection{On the Palais-Smale condition}
 This subsection is devoted to verifying the Palais-Smale condition for the functionals defined in \eqref{functionals}.
 We start with a subcritical superscaled case. See Definition \ref{defscaled} for the definition of an admissible and superscaled pair $(\eta,r)$.
 
\begin{lemma}\label{PS-superscaled-subcritical}
Under the assumptions of Theorem \ref{Th-superscaled-subcritical}, the functional $\Phi$ defined in~\eqref{functionals} satisfies the (PS) condition.
\end{lemma}
\begin{proof} 
Here we have $F(x,t)=\lambda|x|^{-a}|t|^{p}/p+|x|^{-\eta}H(t)$ in the definition of $\Phi$. 
Let $c\in\mathbb{R}$ and $(u_n)$ in $E^q_b$ a $(PS)_c$ sequence for $\Phi,$ namely 
$$
\Phi(u_n)=c+o(1)\quad\mbox{and}\quad \|\Phi'(u_n)\|_{(E^q_b)^*}=o(1),\quad\mbox{as}\,\,n\to\infty.
$$
To show that $(u_n)$ contains a strongly convergent subsequence, we first show that it is bounded in $E^q_b$. Arguing by contradiction, suppose that $\|u_n\|\to\infty$ for a subsequence. We use the notation \eqref{def pi} and set
$$
t_n=t_{u_n}=[I(u_n)]^{-1/\ell},\quad\tilde{u}_n=\pi(u_n)=(u_n)_{t_n}\quad\mbox{and}\quad \tilde{t}_n=t_n^{-1}=[I(u_n)]^{1/\ell}.
$$
Since  $I$ is coercive by \eqref{H10}, we have $\tilde{t}_n\to\infty$.
Once $ \mathcal{M}$ is bounded,  so it is $(\tilde{u}_n)$. Then, up to a subsequence, $\tilde{u}_n\rightharpoonup v$ for some $v\in E^q_b$.
Since $(\tilde{u}_n)_{\tilde{t}_n}=u_n$ by \eqref{H1}, and using \eqref{normscaledineq} we may estimate
\begin{equation}\label{unOt0}
    \|u_n\|=\|(\tilde{u}_n)_{\tilde{t}_n}\|=O((\tilde{t}_n)^{\ell/2}) \quad\text{as}\quad n\to\infty.
\end{equation}
  Since $\Phi(u_n)=c+o(1)$, we observe that
$$
\Phi_\lambda(u_n)=\int_{\R^N}\frac{H(u_n)}{|x|^{\eta}}\,dx+c+o(1),
$$ 
for $\Phi_\lambda$ as defined in~\eqref{functionals}. Therefore, recalling that
\begin{equation*}
 \Phi_\lambda(u_t)
 =I(u_t)-\lambda J(u_t)= t^{\,\ell}\Phi_\lambda(u),\quad \forall u\in E, \forall t\geq0,
\end{equation*}
we  have
\(
\Phi_\lambda(u_n)=(\tilde{t}_n)^{\ell}\Phi_\lambda(\tilde u_n). 
\)
So, we can write
\begin{equation}\label{eqphilambda0}
\Phi_\lambda(\tilde u_n)
= (\tilde{t}_n)^{-\ell}\int_{\R^N}\frac{H(u_n)}{|x|^{\eta}}\,dx+o(1),
\end{equation}
and using \eqref{l(eta,r)} and \eqref{AR h} we obtain
\begin{equation*}
\Phi_\lambda(\tilde u_n)
\geq \frac{c(\tilde{t}_n)^{-\ell}}{r}\int_{\R^N}\frac{|u_n|^{r}}{|x|^\eta}\,dx+o(1)= \frac{c(\tilde{t}_n)^{\ell(\eta,r)-\ell}}{r}\int_{\R^N}\frac{|\tilde{u}_n|^{r}}{|x|^\eta}\,dx+o(1).
\end{equation*}
 Since $r>q$, it is readily seen that $(\eta,r)$ is superscaled, that is, $\ell(\eta,r)>\ell$. 
Then, the boundedness of $(\Phi(\tilde{u}_n))$ and the unboundedness of $(\tilde{t}_n)$ imply that $\tilde{u}_n\to0$ in $L^r_\eta(\R^N)$.
Consequently, $\tilde{u}_n\rightharpoonup0$ in $E^q_b$. From the compactness of the embedding $E^q_b\hookrightarrow L^p_a(\R^N)$, see Corollary \ref{corollary_emb}, we have $\tilde{u}_n\to0$ in $L^p_a(\R^N)$.
On the other hand, using
$$
\Phi_\lambda'(u_t)u_t=\mathcal{A}(u_t)u_t-\lambda \mathcal{B}(u_t)u_t=t^{\,\ell}\Phi_\lambda'(u)u,\quad \forall u\in E, \forall t\geq0,
$$ 
$\Phi'(u_n)u_n=o(\|u_n\|)$ and \eqref{unOt0}, we get
\begin{eqnarray}\label{eqPhi'Ln0}
\Phi'_\lambda(\tilde u_n)\tilde{u}_n=
{(\tilde{t}_n)^{-\ell}}\int_{\mathbb{R}^N}\frac{h(u_n)u_n}{|x|^\eta}dx+o((\tilde{t}_n)^{\ell/2}).
\end{eqnarray}
Finally, multiplying \eqref{eqphilambda0} by $r$ and subtracting \eqref{eqPhi'Ln0}, due to \eqref{AR h} we conclude that
\begin{align*}
  \left(\frac{r}{2}-1\right)&\int_{\mathbb{R}^N}|\nabla\tilde u_n|^2dx+\left(\frac{r}{q}-1\right) \int_{\mathbb{R}^N}\frac{ |\tilde u_n|^q}{|x|^{b}}dx\nonumber \\
 & =\lambda\left(\frac{r}{p} -1\right)\int_{\mathbb{R}^N}\frac{|\tilde u_n|^{p}}{|x|^a}dx 
  +(\tilde{t}_n)^{-\ell}\int_{\mathbb{R}^N}\frac{\left[rH(u_n)-h(u_n)u_n\right]}{|x|^\eta}dx + o(1)\leq  o(1).
\end{align*}
This implies that $\tilde{u}_n\to0$ in $E^q_b$, which is an absurd for $\tilde{u}_n\in\mathcal{M}$. Therefore, $({u}_n)$ must be bounded in $E^q_b$. Now, we consider the potential operator $\tilde{f}:E^b_q\to (E^b_q)^*$, induced by $f$, defined by
\[
\tilde f(u)v=\lambda\int_{\mathbb{R}^N}\frac{|u|^{p-2}uv}{|x|^a}dx+\int_{\mathbb{R}^N}\frac{h(u)v}{|x|^\eta}dx. 
\] 
Since $E^q_b$ is compactly embedded in $L^p_a(\mathbb R^N)$, $L^{r_1}_{\eta}(\mathbb R^N)$ and $L^{r_2}_{\eta}(\mathbb R^N)$ for $r_1,r_2$ as in \eqref{AR h}, see Lemma \ref{embedding nonradial}, $\tilde f$ is a compact operator. If $(u_n)$ is a bounded $(PS)_c$ sequence for $\Phi$, we may use \cite[Proposition 2.3]{MePe2025} by which $(u_n)$ has a convergent subsequence, and this concludes the proof.
\end{proof}
 
Now we focus our attention on critical nonlinearities given by sums of powers of $u$ with singular weights, as in \eqref{f critical r2}. For convenience,   we write it here again:
\begin{equation}\label{nonlinearityPohozaev1}
f(x,u)=\lambda \frac{|u|^{p-2}u}{|x|^a} + \frac{|u|^{2^*_{\eta_1}-2}u}{|x|^{\eta_1}}+ \mu\frac{|u|^{r_2-2}u}{|x|^{\eta_2}},\quad \lambda,\mu\in\R.
 \end{equation}
 This means that $F(x,u)=\frac{\lambda}{p} \frac{|u|^{p}}{|x|^a}+\frac1{2^*_{\eta_1}}\frac{|u|^{2^*_{\eta_1}}}{|x|^{\eta_1}}+\frac1{r_2}\frac{|u|^{r_2}}{|x|^{\eta_2}}$ on the definition of functional $\Phi$ in \eqref{functionals}. Here, the pair $(\eta_1,2^*_{\eta_1})$ brings a critical superscaled configuration to problem~\eqref{general problem}, 
whereas $(\eta_2,r_2)$  may be treated in either the subscaled or superscaled regime, even critical. See Definition \ref{defscaled}.

\begin{lemma}\label{PS-superscaled}
Assume $N\geq3$ and \eqref{basicparametershypothesis}. For
 $f$ given by~\eqref{nonlinearityPohozaev1}, 
suppose that $\lambda\in \R$,  $\eta_1$ satisfies \eqref{eta1 small} and $\mu\in\R$. Moreover suppose that the admissible pair $(\eta_2,r_2)$ and $\mu$ satisfy one of the following assumptions:
\begin{description}
    \item[i)] $\mu\in\R$ and $(\eta_2,r_2)$ is scaled or subscaled, with $b<\eta_2$ ($r_2<2$ if $\eta_2=2$);
\item[ii)] $\mu>0$, $0\le\eta_2<\eta_1$ and $2^*_{\eta_1}\leq r_2\leq 2^*_{\eta_2}$; 
\item[iii)] $\mu>0$, $(\eta_2,r_2)$ is superscaled,  $\eta_2\in[\eta_1,2)$ (equality only if $\eta_1=\eta_2>b$) and $r_2<2^*_{\eta_2}$.
 Additionally, assume  that
\[
r_2>q+\frac{(2^*_{\eta_1}-q)(b-\eta_2)}{b-\eta_1}\quad\text{if}\quad\eta_2< b.
\]
\end{description}
Then every $(PS)$ sequence of the functional $\Phi$ defined in~\eqref{functionals}
is bounded in $E_b^q$.
\end{lemma}
\begin{proof} 
Let $c\in\mathbb{R}$ and $(u_n)$ in $E^q_b$ a $(PS)_c$ sequence for $\Phi.$ 
To show that $(u_n)$ is bounded in $E^q_b$ we start arguing as in Lemma \ref{PS-superscaled-subcritical}. We assume, by contradiction, that $\|u_n\|\to\infty$ for a subsequence, and consider
$$
t_n=t_{u_n}=[I(u_n)]^{-1/\ell},\quad\tilde{u}_n=\pi(u_n)=(u_n)_{t_n}\quad\mbox{and}\quad \tilde{t}_n=t_n^{-1}=[I(u_n)]^{1/\ell}.
$$
We have $\tilde{t}_n\to\infty$, $\tilde{u}_n\in\mathcal{M}$ and, up to a subsequence, $\tilde{u}_n\rightharpoonup v$ for some $v\in E^q_b$.
Since $(\tilde{u}_n)_{\tilde{t}_n}=u_n$, as in \eqref{eqphilambda0} we obtain
\begin{equation}\label{eqphilambda}
\Phi_\lambda(\tilde u_n)
= \frac{(\tilde{t}_n)^{\ell_1-\ell}}{2^*_{\eta_1}}\int_{\mathbb{R}^N}\frac{|\tilde{u}_n|^{2^*_{\eta_1}}}{|x|^{\eta_1}}dx
+\frac{\mu(\tilde{t}_n)^{\ell_2-\ell}}{r_2}\int_{\mathbb{R}^N}\frac{|\tilde{u}_n|^{r_2}}{|x|^{\eta_2}}dx+o(1),
\end{equation}
where $\ell_1=\ell(\eta_1,2^*_{\eta_1})$ and $\ell_2=\ell(\eta_2,r_2)$ (see \eqref{l(eta,r)}). Also, similar  to \eqref{eqPhi'Ln0}, 
we get
\begin{eqnarray}\label{eqPhi'Ln}
\Phi'_\lambda(\tilde u_n)\tilde{u}_n=
{(\tilde{t}_n)^{\ell_1-\ell}}\int_{\mathbb{R}^N}\frac{|\tilde{u}_n|^{2^*_{\eta_1}}}{|x|^{\eta_1}}dx
+{\mu(\tilde{t}_n)^{\ell_2-\ell}}\int_{\mathbb{R}^N}\frac{|\tilde{u}_n|^{r_2}}{|x|^{\eta_2}}dx+o((\tilde{t}_n)^{\ell/2}).
\end{eqnarray}
Finally, multiplying \eqref{eqphilambda} by $2^*_{\eta_1}$ and subtracting \eqref{eqPhi'Ln} we conclude that
\begin{align}\label{principal}
  \left(\frac{2^*_{\eta_1}}{2}-1\right)&\int_{\mathbb{R}^N}|\nabla\tilde u_n|^2dx+\left(\frac{2^*_{\eta_1}}{q}-1\right) \int_{\mathbb{R}^N}\frac{ |\tilde u_n|^q}{|x|^{b}}dx\nonumber \\
 & =\lambda\left(\frac{2^*_{\eta_1}}{p} -1\right)\int_{\mathbb{R}^N}\frac{|\tilde u_n|^{p}}{|x|^a}dx 
  +(\tilde{t}_n)^{\ell_2-\ell}\mu\left(\frac{2^*_{\eta_1}}{r_2} -1\right)\int_{\mathbb{R}^N}\frac{|\tilde u_n|^{r_2}}{|x|^{\eta_2}}dx + o(1).
\end{align}
By the assumption on $\eta_1$ we have $2^*_{\eta_1}>q$ and
$\ell_1>\ell$. Now, let us consider the  conditions on $\mu$ and  $(\eta_2,r_2)$.

\medskip

\noindent \textbf{[(i)]} For $(\eta_2,r_2)$ satisfying i) we have $\ell_2\le\ell$. Since the sequences $(\Phi_\lambda(\tilde{u}_n))$ and $(\|\tilde{u}_n\|_{L^{r_2}_{\eta_2}})$ are bounded and $\tilde{t}_n\to\infty$, it follows from \eqref{eqphilambda} that 
\(
\tilde{u}_n\rightarrow0
\)
in $L^{2^*_{\eta_1}}_{\eta_1}(\R^N)$. This implies that the weak limit of $\tilde{u}_n$ in $E^q_b$ is $v\equiv0$. Since  $b<\eta_2$ and $1< r_2\in(2^*_{b,q,\eta_2},2^*_{\eta_2})$, the compactness of the embeddings of  $E^q_b$ in $ L^{r_2}_{\eta_2}(\R^N)$  
and $L^p_a(\R^N)$ 
(see Lemma \ref{Lemma Embedding with weight} and Corollary \ref{corollary_emb}) give also 
\[
\|\tilde{u}_n\|_{L^{r_2}_{\eta_2}}\rightarrow0\quad\text{and}\quad \|\tilde{u}_n\|_{L^p_a}\rightarrow0.
\]
Then, as  $\ell_2\le\ell$ and $2^*_{\eta_1}>q>2$, it follows from \eqref{principal} that $\|\tilde u_n\|=o(1)$. Since $\tilde u_n\in\mathcal{M},$ this is a contradiction, and therefore the (PS) sequence $(u_n)$ is bounded in $E^q_b$. 

\medskip

\noindent \textbf{[(ii)]} In this case, since $\mu>0$ \eqref{eqphilambda} also yields $\tilde u_n\to0$ in $L^{2^*_{\eta_1}}_{\eta_1}(\R^N)$.
Consequently $\tilde u_n\rightharpoonup0$ in $E^q_b$ and so $\tilde u_n\to0$ in $L^p_a(\R^N)$. Once $2^*_{\eta_1}\leq r_2$ and $\mu>0$, by \eqref{principal} we reach $\|\tilde u_n\|\leq o(1)$, which is an absurd. 

\medskip

\noindent \textbf{[(iii)]} Here we assume $\mu>0$, $\eta_1\le\eta_2$ (equality only if $\eta_1=\eta_2>b$), and the conditions in $(\eta_2,r_2)$ ensures that $r_2<2^*_{\eta_1}$ and $\ell<\ell_2<\ell_1$. Then, from \eqref{eqphilambda} we get
\begin{equation}\label{cri O}
    \int_{\mathbb{R}^N}\frac{|\tilde{u}_n|^{2^*_{\eta_1}}}{|x|^{\eta_1}}dx=O((\tilde{t}_n)^{\ell-\ell_1}),\quad\text{as}\quad n\to\infty,
\end{equation}
which implies that  $\tilde u_n\rightharpoonup0$ in $E^q_b$. Let us denote
\begin{equation*}
\bar\eta:=\frac{(\eta_1-\eta_2)(\delta q+b)-\delta\eta_1 r_2+\delta\eta_2 2^*_{\eta_1}}{\delta(2^*_{\eta_1}-r_2)+\eta_1-\eta_2}\quad\text{and}\quad \bar s:=q+\frac{b-\bar\eta}{\delta}.
\end{equation*}
Under the conditions in this item, we get $\bar\eta\in(b,2)$. Moreover, $\ell_2>\ell$ implies that $\bar{s}<r_2$.  
Let us show that $\bar\eta\in(b,2)$. Since $\ell_1>\ell_2$ we verify that
\begin{align*}
    \bar\eta\ge\eta_2 \quad\Leftrightarrow \quad (\eta_2-\eta_1)(\ell_2-\ell)\ge0.
\end{align*}
Notice that, if $\eta_1=\eta_2>b$ we have $\bar\eta=\eta_2>b$. For $\eta_2\ge b$ and $\eta_2>\eta_1$, it holds $\bar\eta>\eta_2\geq b$. 
If $\eta_1<\eta_2<b$ we compare $\bar\eta$ directly with $b$, getting 
\begin{align*}
   \bar\eta > b \quad\Leftrightarrow \quad (\eta_2-\eta_1)(2^*_{\eta_1}-q)-(2^*_{\eta_1}-r_2)(b-\eta_1)>0
    \quad\Leftrightarrow \quad r_2>q+(2^*_{\eta_1}-q)\frac{b-\eta_2}{b-\eta_1},
\end{align*}
as we required. 
On the other hand, 
\(\bar\eta<2\) iff $r_2<2^*_{\eta_2}$.
Thus, for these  $\bar\eta\in(b,2)$ and $\bar{s}=q+(b-\bar\eta)/\delta$ we have $2<\bar{s}\in(2^*_{b,q,\bar\eta},2^*_{\bar\eta})$. Due to the compactness of the embeddings $E^q_b\hookrightarrow L^{\bar{s}}_{\bar\eta}(\R^N)$ and $E^q_b\hookrightarrow L^p_a(\R^N)$  we obtain 
\begin{equation}\label{both convergence}
\|\tilde{u}_n\|_{L^{\bar{s}}_{\bar\eta}}\rightarrow0\quad\text{and}\quad \|\tilde{u}_n\|_{L^p_a}\rightarrow0.
\end{equation}
Besides, by the choice of $\bar\eta$ and $\bar{s}$, we see that there exists $\theta\in(0,1)$ such that
\[
\frac{1}{r_2}=\frac{\theta}{2^*_{\eta_1}}+\frac{1-\theta}{\bar{s}}\quad\text{and}\quad 
\frac{\eta_2}{r_2}=\frac{\theta\eta_1}{2^*_{\eta_1}}+\frac{(1-\theta)\bar\eta}{\bar{s}}.
\]
Using the interpolation inequality in weighted Lebesgue spaces (see e.g. \cite[Inequality (2.15)]{CKN-1984}) we get
\[
\left(\int_{\mathbb{R}^N}\frac{|\tilde{u}_n|^{r_2}}{|x|^{\eta_2}}dx\right)^\frac{1}{r_2}\leq \left(\int_{\mathbb{R}^N}\frac{|\tilde{u}_n|^{2^*_{\eta_1}}}{|x|^{\eta_1}}dx\right)^\frac{\theta}{2^*_{\eta_1}}\left(\int_{\mathbb{R}^N}\frac{|\tilde{u}_n|^{\bar{s}}}{|x|^{\bar\eta}}dx\right)^\frac{1-\theta}{\bar{s}}.
\]
So, using \eqref{cri O} we obtain 
\[
(\tilde{t}_n)^{\ell_2-\ell}\int_{\mathbb{R}^N}\frac{|\tilde{u}_n|^{r_2}}{|x|^{\eta_2}}dx
\leq C(\tilde{t}_n)^{\ell_2-\ell-(\ell_1-\ell)\frac{\theta r_2}{2^*_{\eta_1}}}\|\tilde{u}_n\|_{L^{\bar{s}}_{\bar\eta}}^{(1-\theta)r_2}=o(1),
\]
since $\ell_2-\ell=(\ell_1-\ell)(\theta r_2/2^*_{\eta_1})$. Adding this information to \eqref{both convergence}, from \eqref{principal} we reach $\tilde{u}_n\to0$ in $E^q_b$, which is an absurd because $\tilde{u}_n\in\mathcal{M}$.

Therefore, in each of the cases [(i)]-[(iii)] we have $(u_n)$ a bounded sequence.
\end{proof}

\begin{remark}
When $N\ge3$, observe that \(\max\{q, 2^*_{b,q,\eta}\} = 2^*_{b,q,\eta}\) whenever \(\eta < b\), 
whereas it equals \(q\) if \(b \le \eta < N\).  
Thus, conditions of the form \(r > q\) for an admissible pair \((\eta,r)\) are 
automatically satisfied whenever \(0 \le \eta < b\). Moreover, the condition \(\eta < N - {q(N-2)}/{2}\) is equivalent to 
\(q < 2^*_\eta\).  
Since we assume \(q < 2^*_b\), this inequality also holds automatically for every 
\(\eta < b\).
\end{remark}

The next lemma provides the thresholds where one can recover the $(PS)_c$ conditions for the critical cases. As usual, this threshold is related to the constants of the Embeddings of $D^{1,2}$ into the Lebesgue spaces. Then, let us define, for $N\ge 3$ and each $0\le\eta<2$,
\begin{equation}\label{defSeta}
S_{\eta}=\inf_{ u\in D^{1,2}(\mathbb R^N)\setminus\{0\}}\frac{\int_{\mathbb R^N}|\nabla u|^2dx}{\left(\int_{\mathbb R^N}\frac{|u|^{2^*_{\eta}}}{|x|^{\eta}}\right)^{\frac{2}{2^*_\eta}}}.
\end{equation}

\begin{lemma}\label{lemmaPS-criticalpower+superscaledpower}
Assume \eqref{basicparametershypothesis} and $N\ge 3$. Consider $f(x,u)$ as given in \eqref{f critical r2},
with $\lambda\in \R$, $\mu\ge 0$, $\eta_1$  satisfying \eqref{eta1 small} and assume that
 $(\eta_2,r_2)$ is a superscaled pair. 
\begin{description}
    \item[(i)] If $(\eta_2,r_2)$ satisfies \eqref{eta2 subcri1} or \eqref{eta2 subcri2}, then  the functional $\Phi$ defined in~\eqref{functionals}
satisfies $(PS)_c$  for all 
\[c<\frac{2-\eta_1}{2(N-\eta_1)}S^{\frac{N-\eta_1}{2-\eta_1}}_{\eta_1}.\]
  \item[(ii)] If $(\eta_2,r_2)$ satisfies \eqref{eta2 cri}, define $\tilde S=\tilde S(\mu,\eta_1,\eta_2)>0$ such that
\(
\mu S_{\eta_2}^{-\frac{2^*_{\eta_2}}{2}}\tilde S^{\frac{2-{\eta_2}}{N-2}}+S_{\eta_1}^{-\frac{2^*_{\eta_1}}{2}}\tilde S^{\frac{2-\eta_1}{N-2}}=1.
\)
Then, the functional $\Phi$ satisfies the $(PS)_c$ condition for all 
\[c<\frac{2-\max\{\eta_1,\eta_2\}}{2(N-\max\{\eta_1,\eta_2\})}\tilde S.\]
\end{description}
\end{lemma}
\begin{proof}
Let \(c \in \mathbb{R}\) and let \((u_n)\) be a \((PS)_c\) sequence for \(\Phi\).
We are under the hypotheses of Lemma~\ref{PS-superscaled}. Therefore, \((u_n)\) is bounded.
Then, let $u\in E^q_b$ be such that $u_n\rightharpoonup u$ in $E^q_b$.  By this weak convergence, we have 
$$
\int_{\mathbb R^N}|\nabla(u_n-u)|^2dx=\int_{\mathbb R^N}|\nabla u_n|^2dx-\int_{\mathbb R^N}|\nabla u|^2dx+o(1)
$$
and by Brezis-Lieb Lemma
$$
\int_{\mathbb R^N}\frac{|u_n-u|^q}{|x|^b}dx=\int_{\mathbb R^N}\frac{|u_n|^q}{|x|^b}dx-\int_{\mathbb R^N}\frac{|u|^q}{|x|^b}dx+o(1)
$$
$$
\int_{\mathbb R^N}\frac{|u_n-u|^{2^*_{\eta_1}}}{|x|^{\eta_1}}dx=\int_{\mathbb R^N}\frac{|u_n|^{2^*_{\eta_1}}}{|x|^{\eta_1}}dx-\int_{\mathbb R^N}\frac{|u|^{2^*_{\eta_1}}}{|x|^{\eta_1}}dx+o(1)
$$
with analogous identity for the pair $(\eta_2,2^*_{\eta_2})$. For later use, notice that these relations imply that
\begin{align}
\label{brezisliebqb}\int_{\mathbb R^N}\frac{|u_n|^q}{|x|^b}dx\ge\int_{\mathbb R^N}\frac{|u|^q}{|x|^b}dx+o(1)\\\label{brezisliebeta1}
\int_{\mathbb R^N}\frac{|u_n|^{2^*_{\eta_2}}}{|x|^{\eta_2}}dx\ge\int_{\mathbb R^N}\frac{|u|^{2^*_{\eta_2}}}{|x|^{\eta_2}}dx+o(1)
\end{align}
Moreover, by the compactness of the embeddings \[E^q_b\hookrightarrow L^{p}_a(\mathbb R^N),\ E^q_b\hookrightarrow L^{r_2}_{\eta_2}(\mathbb R^N)\quad (\text{if }r_2<2^*_{\eta_2}),\] we have $u_n\rightarrow u$ strongly in both spaces. Notice also that  standard and well-known arguments show that $\Phi^\prime(u)v=0$ for all $v\in E^q_b$. Then, using all the convergences and Brezis-Lieb relations in the equation
$$
\Phi^\prime(u_n)u_n-\Phi^\prime(u)u=o(1)
$$
gives, in the case $r_2<2^*_{\eta_2}$, the estimate
\begin{equation}\label{PSeq1-1critical}
\int_{\mathbb R^N}|\nabla(u_n-u)|^2dx+\int_{\mathbb R^N}\frac{|u_n-u|^q}{|x|^b}dx=\int_{\mathbb R^N}\frac{|u_n-u|^{2^*_{\eta_1}}}{|x|^{\eta_1}}dx+o(1)
\end{equation}
 or, in case $r_2=2^*_{\eta_2}$, the following
\begin{equation}\label{PSeq1-2critical}
\int_{\mathbb R^N}|\nabla(u_n-u)|^2dx+\int_{\mathbb R^N}\frac{|u_n-u|^q}{|x|^b}dx=\mu\int_{\mathbb R^N}\frac{|u_n-u|^{2^*_{\eta_2}}}{|x|^{\eta_2}}dx+\int_{\mathbb R^N}\frac{|u_n-u|^{2^*_{\eta_1}}}{|x|^{\eta_1}}dx+o(1).
\end{equation}

\noindent
\textbf{Proof of (i):}  Here we have $r_2<2^*_{\eta_2}$. We start from \eqref{PSeq1-1critical}. Due to \eqref{defSeta}, we get
\[
\int_{\mathbb R^N}|\nabla(u_n-u)|^2dx\le\int_{\mathbb R^N}|\nabla(u_n-u)|^2dx+\int_{\mathbb R^N}\frac{|u_n-u|^q}{|x|^b}dx\le S_{\eta_1}^{-\frac{2^*_{\eta_1}}{2}}\left(\int_{\mathbb R^N}|\nabla(u_n-u)|^2dx\right)^{\frac{2^*_{\eta_1}}{2}}\!\!\!+o(1).
\]
Assume, by contradiction, that \((u_n)\) does not converge to \(u\) in \(E_b^q\), not even along a subsequence.  
Then, by \eqref{PSeq1-1critical} and the last inequality above, the same lack of convergence holds in \(D^{1,2}(\mathbb{R}^N)\). So we can divide the last inequality by  $\|\nabla(u_n-u)\|_{L^2}^2$ and  get
\[
\int_{\mathbb R^N}|\nabla(u_n-u)|^2dx\geq S_{\eta_1}^{\frac{N-\eta_1}{2-\eta_1}}+o(1).
\]
Then,
\begin{equation}\label{PSineq3}
\int_{\mathbb R^N}|\nabla u_n|^2dx-\int_{\mathbb R^N}|\nabla u|^2dx=\int_{\mathbb R^N}|\nabla(u_n-u)|^2dx+o(1)\ge S_{\eta_1}^{\frac{N-\eta_1}{2-\eta_1}}+o(1).
\end{equation}
Now, divide the equation $\Phi^\prime(u_n)u_n=o(1)$ by $2^*_{\eta_1}$ and subtract if from $\Phi(u_n)=c+o(1)$. This gives
\begin{align*}
c+o(1)\,\,= \,\,& \frac{2-\eta_1}{2(N-\eta_1)}\int_{\mathbb R^N}|\nabla u_n|^2dx+\left(\frac{1}{q}-\frac{1}{2^*_{\eta_1}}\right)\int_{\mathbb R^N}\frac{|u_n|^{q}}{|x|^{b}}dx\\
&-\lambda\left(\frac{1}{p}-\frac{1}{2^*_{\eta_1}}\right)\int_{\mathbb R^N}\frac{|u_n|^{p}}{|x|^{a}}dx-\mu\left(\frac{1}{r_2}-\frac{1}{2^*_{\eta_1}}\right)\int_{\mathbb R^N}\frac{|u_n|^{r_2}}{|x|^{\eta_2}}dx.
\end{align*}
Now, recall \eqref{brezisliebqb} and \eqref{PSineq3} and let $n\rightarrow\infty$ in this last equality to reach
\begin{align}\nonumber
c\,\,\ge\,\,& \frac{2-\eta_1}{2(N-\eta_1)}S_{\eta_1}^{\frac{N-\eta_1}{2-\eta_1}}+\frac{2-\eta_1}{2(N-\eta_1)}\int_{\mathbb R^N}|\nabla u|^2dx+\left(\frac{1}{q}-\frac{1}{2^*_{\eta_1}}\right)\int_{\mathbb R^N}\frac{|u|^{q}}{|x|^{b}}dx\\ \label{PSineq5}
&-\lambda\left(\frac{1}{p}-\frac{1}{2^*_{\eta_1}}\right)\int_{\mathbb R^N}\frac{|u|^{p}}{|x|^{a}}dx-\mu\left(\frac{1}{r_2}-\frac{1}{2^*_{\eta_1}}\right)\int_{\mathbb R^N}\frac{|u|^{r_2}}{|x|^{\eta_2}}dx,
\end{align}
where we have used that $u_n\rightarrow u$ both in $L^{p}_a(\mathbb R^N)$ and $L^{r_2}_{\eta_2}(\mathbb R^N)$. We need now to invoke the Pohozaev Identity given in \eqref{Pohozaev identity}. Multiply the Pohozaev Identity by $\alpha$ and $\Phi^\prime(u)u=0$ by $\beta$, where $\alpha$ and $\beta$ will be chosen later. Then, add the results in \eqref{PSineq5}. We get
\begin{align}\nonumber
c\,\,\ge\,\,&  \frac{2-\eta_1}{2(N-\eta_1)}S_{\eta_1}^{\frac{N-\eta_1}{2-\eta_1}}\!+\left[\frac{2-\eta_1}{2(N-\eta_1)}+\frac{N-2}{2}\alpha+\beta\right]\!\int_{\mathbb R^N}\!\!|\nabla u|^2dx\\\nonumber&+\left[\left(\frac{1}{q}-\frac{1}{2^*_{\eta_1}}\right)\!+\frac{N-b}{q}\alpha+\beta\right]\!\int_{\mathbb R^N}\!\!\frac{|u|^{q}}{|x|^{b}}dx-\lambda\left[\left(\frac{1}{p}-\frac{1}{2^*_{\eta_1}}\right)+\frac{N-a}{p}\alpha+\beta\right]\int_{\mathbb R^N}\frac{|u|^{p}}{|x|^{a}}dx\\\label{PSineq7}
&-\mu\left[\left(\frac{1}{r_2}-\frac{1}{2^*_{\eta_1}}\right)+\frac{N-\eta_2}{r_2}\alpha+\beta\right]\int_{\mathbb R^N}\frac{|u|^{r_2}}{|x|^{\eta_2}}dx-\left[\frac{N-2}{2}\alpha+\beta\right]\int_{\mathbb R^N}\frac{|u|^{2^*_{\eta_1}}}{|x|^{\eta_1}}.
\end{align}
Notice that we used that $(N-\eta_1)/2^*_{\eta_1}=(N-2)/2$.  Now choose
\[
    \alpha=\frac{q-2}{2(N-b)-q(N-2)}
    \qquad\text{and}\qquad
    \beta
    =-\frac{2-\eta_{1}}{2(N-\eta_{1})}
      -\frac{N-2}{2}\,\alpha.
\]
With this choice, the first three expressions in \([\cdots]\) vanish, while the last one becomes strictly negative.
It remains to verify that
\[
\left( \frac{1}{r_{2}}
      -\frac{1}{2^*_{\eta_{1}}}
\right)
+\frac{N-\eta_{2}}{r_{2}}\,\alpha
+\beta
<0,
\]
which is equivalent to
\[
    r_{2}
    >
    2\,\frac{1+\alpha (N-\eta_{2})}
          {1+\alpha (N-2)}.
\]
However, this is precisely the requirement for the pair $(\eta_2,r_2)$ to be called superscaled, since
\[
2\frac{1+\alpha(N-\eta_2)}{1+\alpha(N-2)}=q+\frac{b-\eta_2}{\delta}.
\]
Finally, we go back to \eqref{PSineq7} and see that
\[c\ge\frac{2-\eta_1}{2(N-\eta_1)}S_{\eta_1}^{\frac{N-\eta_1}{2-\eta_1}}.\]
This completes the proof for the item (i).

\medskip 

\noindent
\textbf{Proof of (ii):} 
Now, let us consider $\eta_2$ satisfying \eqref{eta1 small} and $r_2=2^*_{\eta_2}$. From \eqref{PSeq1-2critical}, using \eqref{defSeta}, we get
\[
\int_{\mathbb R^N}|\nabla(u_n-u)|^2dx\le \mu S_{\eta_2}^{-\frac{2^*_{\eta_2}}{2}}\left(\int_{\mathbb R^N}|\nabla(u_n-u)|^2dx\right)^{\frac{2^*_{\eta_2}}{2}} + S_{\eta_1}^{-\frac{2^*_{\eta_1}}{2}}\left(\int_{\mathbb R^N}|\nabla(u_n-u)|^2dx\right)^{\frac{2^*_{\eta_1}}{2}}\!\!\!+o(1).
\]
Here, analogously to the previous case, we suppose that $(u_n)$ does not have any subsequence converging to $u$ in $E^q_b$, which, due to \eqref{PSeq1-2critical}, implies that $\|\nabla(u_n-u)\|^2_{L^2}$ stays away from 0 for large values of $n$. So, we can divide the last inequality by $\|\nabla(u_n-u)\|^2_{L^2}$ and get
\[
\mu S_{\eta_2}^{-\frac{2^*_{\eta_2}}{2}}\left(\int_{\mathbb R^N}|\nabla(u_n-u)|^2dx\right)^{\frac{2-\eta_2}{N-\eta_2}} + S_{\eta_1}^{-\frac{2^*_{\eta_1}}{2}}\left(\int_{\mathbb R^N}|\nabla(u_n-u)|^2dx\right)^{\frac{2-\eta_1}{N-\eta_1}}\ge 1+o(1).
\]
Thus, the definition of $\tilde S$ in the assumptions of this lemma gives us
\begin{equation*}
\int_{\mathbb R^N}|\nabla u_n|^2dx-\int_{\mathbb R^N}|\nabla u|^2dx=\int_{\mathbb R^N}|\nabla(u_n-u)|^2dx+o(1)\ge \tilde S+o(1).
\end{equation*}
Now, without loss of generality, suppose $\eta_1\ge\eta_2$. The other case can be handled by just interchanging their roles in the following arguments. We divide the identity 
\(\Phi'(u_n)u_n = o(1)\) by \(2^*_{\eta_1}\) and subtract it from 
\(\Phi(u_n) = c + o(1)\).  
This yields
\begin{align*}
c+o(1)\,\,=\,\,&\frac{2-\eta_1}{2(N-\eta_1)}\int_{\mathbb R^N}|\nabla u_n|^2dx+\left(\frac{1}{q}-\frac{1}{2^*_{\eta_1}}\right)\int_{\mathbb R^N}\frac{|u_n|^{q}}{|x|^{b}}dx\\&-\lambda\left(\frac{1}{p}-\frac{1}{2^*_{\eta_1}}\right)\int_{\mathbb R^N}\frac{|u_n|^{p}}{|x|^{a}}dx+\mu\left(\frac{1}{2^*_{\eta_1}}-\frac{1}{2^*_{\eta_2}}\right)\int_{\mathbb R^N}\frac{|u_n|^{2^*_{\eta_2}}}{|x|^{\eta_2}}dx.
\end{align*}
By assumption, we have $q<2^*_{\eta_1}\le 2^*_{\eta_2}$ . So, we use \eqref{brezisliebqb} and \eqref{brezisliebeta1}  and letting $n\rightarrow\infty$ in the last inequality, we have
\begin{align*}
c\,\,\ge\,\,&\frac{2-\eta_1}{2(N-\eta_1)}\tilde S+\frac{2-\eta_1}{2(N-\eta_1)}\int_{\mathbb R^N}|\nabla u|^2dx+\left(\frac{1}{q}-\frac{1}{2^*_{\eta_1}}\right)\int_{\mathbb R^N}\frac{|u|^{q}}{|x|^{b}}dx\\&-\lambda\left(\frac{1}{p}-\frac{1}{2^*_{\eta_1}}\right)\int_{\mathbb R^N}\frac{|u|^{p}}{|x|^{a}}dx+\mu\left(\frac{1}{2^*_{\eta_1}}-\frac{1}{2^*_{\eta_2}}\right)\int_{\mathbb R^N}\frac{|u|^{2^*_{\eta_2}}}{|x|^{\eta_2}}dx.
\end{align*}
We now proceed exactly as in the subcritical case $r_2<2^*_{\eta_2}$. This leads to the conclusion that
$$
c\,\,\ge\,\,\frac{2-\eta_1}{2(N-\eta_1)}\tilde S,$$
and, of course, if $\eta_2>\eta_1$ we would have $\eta_2$ instead of $\eta_1$ in this last estimate. This finishes this proof.
\end{proof}
\begin{remark} \label{RemarkPSsubscaled+critical}
The \((PS)_c\) condition also holds for levels below the threshold in 
Lemma~\ref{lemmaPS-criticalpower+superscaledpower} when \((\eta_2,r_2)\) is subscaled, $b<\eta_2<2$  and 
\(\mu \le 0\).
Indeed, notice that \((PS)_c\) sequences are bounded due to Lemma~\ref{PS-superscaled}. To prove that there exist convergent subsequences, we use Corollary \ref{subscaled compact}, follow the same steps of the proof above, and recall that, in the case with only one critical pair, the argument required
\[
    \mu\left[
        \left(\frac{1}{r_2} - \frac{1}{2^*_{\eta_1}}\right)
        + \frac{N-\eta_2}{r_2}\,\alpha
        + \beta
    \right] \le 0,
\]
for the specific choices of \(\alpha\) and \(\beta\) made there.
That inequality was valid because \(\mu \ge 0\) and \((\eta_2,r_2)\) was superscaled.  
The same inequality remains valid when \(\mu \le 0\) and \((\eta_2,r_2)\) is subscaled, and thus the \((PS)_c\) condition holds in this setting as well.
\end{remark}

\subsection{Existence of solutions for subcritical superscaled problems}

Here we prove that \eqref{general problem} has a nontrivial solution for subcritical superscaled nonlinearities.   
We present distinct proofs for Theorem \ref{Th-superscaled-subcritical}, concerning $\lambda\in \sigma(\mathcal{A},\mathcal{B})$ or not. If $\lambda$ is an eigenvalue for \eqref{NE problem}, we  apply a recent abstract result based on a notion of local linking (see \cite[Theorem 2.30 and Corollary 2.31]{MePe2025}).

\begin{proof-th} \textbf{\textit{of Theorem \ref{Th-superscaled-subcritical}.}}\addcontentsline{toc}{subsubsection}{Proof of Theorem~\ref{Th-superscaled-subcritical}}
Lemma \ref{PS-superscaled-subcritical}  ensures that $\Phi$ satisfies the $(PS)$ condition. 
Let us denote  $\tilde{F},\tilde{H}:E^q_b\to\mathbb{R}$ by
$$
\tilde{F}(u)=\int_{\mathbb{R}^N}F(x,u)dx=\frac{\lambda}{p}\int_{\mathbb{R}^N}\frac{|u|^{p}}{|x|^a}dx+\int_{\mathbb{R}^N}\frac{H(u)}{|x|^\eta}dx \quad\text{and}\quad \tilde{H}(u)=\int_{\mathbb{R}^N}\frac{H(u)}{|x|^\eta}dx, 
$$
and $\tilde{h}=\tilde{H}'$. For $\ell_i=\ell(\eta_i,r_i)= r_i\delta+\eta-N$, $i=1,2$, due to \eqref{AR h}, we verify that
\begin{align*}
|\tilde{H}(u_t)|&\leq\frac{c_1t^{\,\ell_1}}{r}\int_{\mathbb{R}^N}\frac{|u|^{r_1}}{|x|^\eta}dx
+\frac{c_2t^{\,\ell_2}}{r}\int_{\mathbb{R}^N}\frac{|u|^{r_2}}{|x|^\eta}dx,\\
|\tilde{h}(u_t)v_t|&\leq c_1t^{\,\ell_1}\int_{\mathbb{R}^N}\frac{|u|^{r_1-1}|v|}{|x|^\eta}dx
+{c_2t^{\,\ell_2}}\int_{\mathbb{R}^N}\frac{|u|^{r_2-1}|v|}{|x|^\eta}dx
\leq  c_1t^{\,\ell_1}\|u\|_{L^{r_1}_\eta}^{r_1-1}\|v\|_{L^{r_1}_\eta}+ c_2t^{\,\ell_2}\|u\|_{L^{r_2}_\eta}^{r_2-1}\|v\|_{L^{r_2}_\eta}
\end{align*}
and 
\begin{align*}
\tilde{F}(u_t)\geq\tilde{H}(u_t)\ge\frac{ct^{\,\ell(\eta,r)}}{r}\int_{\mathbb{R}^N}\frac{|u|^{r}}{|x|^\eta}dx.
\end{align*}
Using the embeddings of $E^q_b$ in $L^{r_i}_\eta(\R^N)$ and  that $\ell<{\ell(\eta,r)},\ell_i$ (since we considered superscaled pairs), $i=1,2$,  we obtain  $\tilde{h}(u_t)v_t=o(t^{\,\ell})\|v\|$ and $\tilde{H}(u_t)=o(t^{\,\ell})$ as ${t\to0^+}$, uniformly in $u$ on bounded subsets, and
\begin{equation*}
   \lim_{t\to\infty}\frac{\tilde{F}(u_t)}{t^{\,\ell}}=\infty,
\end{equation*}
uniformly in $u$ on compact subsets of $\mathcal{M}.$
Then, if $\lambda\notin\sigma(\mathcal{A},\mathcal{B})$, we can apply \cite[Theorem 2.27]{MePe2025} to guarantee the existence of a nontrivial solution for \eqref{general problem} at a positive level. In particular, since $\sigma(\mathcal{A},\mathcal{B})\subset[\lambda_1,\infty)$, it holds for any $\lambda<\lambda_1$.

In general, for $\lambda\geq\lambda_1$ (possibly an eigenvalue) we will apply \cite[Corollary 2.31]{MePe2025} to prove this result. Consider $k\in\mathbb{N}$ such that $\lambda_k\leq\lambda<\lambda_{k+1}$. Our first step is to show that  $\Phi$ has a scaled local linking near the origin in dimension $k$, which means that 
there exist two nonempty, closed, symmetric and disjoint subsets $A_0,B_0\subset\mathcal{M}$ such that
$$
i(A_0)=i(\mathcal{M}\backslash B_0)=k
$$
and $\rho>0$ satisfying
\begin{align}\label{local linking k}
\begin{cases}
    \Phi(u_t)\leq0,&\forall u\in A_0\,\,\mbox{and}\,\, 0\leq t\leq\rho,\\
      \Phi(u_t)>0,&\forall u\in B_0\,\,\mbox{and}\,\, 0< t\leq\rho.
\end{cases}
\end{align}
Consider $A_0=\tilde{\Psi}^{\lambda_k}$ and $B_0=\tilde{\Psi}_{\lambda_{k+1}}$.
Since $\lambda_k<\lambda_{k+1}$, by Theorem \ref{lambdak}\ref{lambdak(iii)} we have
  $$
i(\tilde{\Psi}^{\lambda_k})=i(\mathcal{M}\backslash\tilde{\Psi}_{\lambda_{k+1}})=k.
  $$
Moreover, for $u\in B_0$ we have $\tilde{\Psi}(u)\ge \lambda_{k+1}$ so that $J(u)\le1/\lambda_{k+1}$ and then
\[
\Phi(u_t)=I(u_t)-\lambda J(u_t)-\tilde{H}(u_t)\geq t^{\,\ell}\left(1-\frac{\lambda}{\lambda_{k+1}}+o(1)\right)\quad\text{as}\,\,t\to0^+,
\]
uniformly on $u\in\mathcal{M}$. Thus, since $\lambda<\lambda_{k+1}$, for small $\rho>0$ we see that $\Phi$  satisfies  \eqref{local linking k} on $B_0$.
On the other hand, for $u\in A_0$ we have $\tilde{\Psi}(u)\le \lambda_{k}\le\lambda$. As $\tilde{H}\ge0$, for all $u\in A_0$ and $t\in[0,\rho]$ we get 
\[
\Phi(u_t)=I(u_t)-\lambda J(u_t)-\tilde{H}(u_t)\le  t^{\,\ell}\left(1-\frac{\lambda}{\lambda_{k}}\right)\leq0.
\]
Thus \eqref{local linking k} holds and $\Phi$ has a scaled local linking near the origin in dimension $k$.

Next we need to show that $\Phi^a$ is contractible for some $\alpha<0$. Consider
$$
\alpha<\inf_{u\in\mathcal{M}, \,0\leq t\leq1}\Phi(u_t).
$$
Denoting  $\varphi_u(t)=\Phi(u_t)$ we have
$$
\Phi^\alpha=\{u\in E^q_b:\Phi(u)\leq \alpha\}=\{u_t:u\in\mathcal{M},\,t>1\,\,\mbox{and}\,\,\varphi_u(t)\leq \alpha\}.
$$
Hence, using again \eqref{AR h}, for each $u\in\mathcal{M}$ we  get
\begin{eqnarray*}
\varphi_u(t)&=&I(u_t)-\lambda J(u_t)-
\tilde{H}(u_t)=t^{\,\ell}\left(1-{\lambda}J(u)\right)-\int_{\mathbb{R}^N}\frac{H(u_t)}{|x|^\eta}dx\\
&\leq& t^{\,\ell}\left(1-\frac{\lambda}{\tilde{\Psi}(u)}\right)-\frac{ct^{{\ell(\eta,r)}}}{r}\int_{\mathbb{R}^N}\frac{|u|^{r}}{|x|^\eta}dx\,\,\to\,\,-\infty\quad\text{as}\,\,t\to\infty,
\end{eqnarray*}
since ${\ell(\eta,r)}>\ell>0$. Notice that
\[\int_{\mathbb{R}^N}\frac{H(u_t)}{|x|^\eta}dx=\int_{\mathbb{R}^N}\frac{H(t^\delta u(tx))}{|x|^\eta}dx
=t^{\eta-N}\int_{\mathbb{R}^N}\frac{H(t^\delta u(x))}{|x|^\eta}dx
\]
and so  $\varphi_u$ is a $C^1$ function satisfying 
\begin{align*}
t\varphi'_u(t)&=\ell t^{\,\ell }\left(1-\frac{\lambda}{\tilde{\Psi}(u)}\right)-(\eta-N)t^{\eta-N}\int_{\mathbb{R}^N}\frac{H(t^\delta u)}{|x|^\eta}dx-\delta t^{\eta-N}\int_{\mathbb{R}^N}\frac{h(t^\delta u)t^{\delta}u}{|x|^\eta}dx\\
&\leq \ell t^{\,\ell }\left(1-\frac{\lambda}{\tilde{\Psi}(u)}\right)-\ell(\eta,r)t^{\eta-N}\int_{\mathbb{R}^N}\frac{H(t^\delta u)}{|x|^\eta}dx
\,\,\le \,\,\ell \varphi_u(t),
\quad\mbox{for all}\,\,t>0.
\end{align*}
So, $\varphi_u(t)<0$ yields $\varphi'_u(t)<0$. Then, the implicit function theorem provides a $C^1-$map $\psi:\mathcal{M}\to(1,\infty)$ such that
$$
\varphi_u(t)>\alpha\,\,\mbox{for}\,\, 0\leq t<\psi(u),\quad\varphi_u(\psi(u))=\alpha\quad\mbox{and}\quad\varphi_u(t)<\alpha\,\,\forall t>\psi(u),
$$
so that
$$
\Phi^\alpha=\{u_t:u\in\mathcal{M},\,t\geq \psi(u)\}.
$$
Finally, for $t_u$ and $\tilde{u}=\pi(u)=u_{t_u}$ as in \eqref{def pi}, 
we consider $T:(E\backslash\{0\})\times [0,1] \to E\backslash\{0\}$ defined as
$$
T(u,\theta)=
\begin{cases}
\tilde{u}_{\theta\psi(\tilde{u})+(1-\theta)t^{-1}_u}& \mbox{if}\,\,u\notin\Phi^\alpha\\
u& \mbox{if}\,\,u\in\Phi^\alpha,
\end{cases}
$$
which is a deformation retraction onto $\Phi^\alpha$. Since $E^q_b\backslash\{0\}$ is a contractible set, we conclude that $\Phi^\alpha$ is also contractible. At this point, we can apply \cite[Corollary 2.31]{MePe2025} and conclude this proof.
\end{proof-th}

Next we are going to prove our multiplicity result for subcritical problems.

\begin{proof-th} \textbf{\textit{of Theorem \ref{Th-superscaled-subcritical-multi}.}}\addcontentsline{toc}{subsubsection}{Proof of Theorem~\ref{Th-superscaled-subcritical-multi}}
To prove this result apply \cite[Propositions 3.42 and 3.44]{Perera-book}.
From Lemma \ref{PS-superscaled-subcritical}  we know that $\Phi$ satisfies the $(PS)_c$ condition for all $c\in\R$, in particular for $c\in(0,\infty)$. Since we are assuming that $h$ is an odd function we have $\Phi$ an even functional. 
We set 
 $$
 B_0:=\tilde{\Psi}_{\lambda_{k_0}}\quad\mbox{for}\quad k_0:=\min\{k\in\mathbb{N}: \lambda<\lambda_k\}.
$$ 
 If $k_0=1$, by Theorem \ref{lambdak}\ref{lambdak(i)} we have
$$
B_0=\{u\in\mathcal{M}:\tilde{\Psi}(u)\geq\lambda_1\}=\mathcal{M}.
$$
 Hence $i(\mathcal{M}\backslash B_0)=i(\emptyset)=0= k_0-1$. If $k_0\ge2$, it holds that $\lambda_{k_0-1}\leq\lambda<\lambda_{k_0}$, hence Theorem \ref{lambdak}\ref{lambdak(iii)} yields 
$$
i(\mathcal{M}\backslash B_0)=
i(\mathcal{M}\backslash\tilde{\Psi}_{\lambda_{k_0}})=k_0-1.
$$
For $k\geq k_0$ in $\mathbb{N}$ we consider $W_k$ a subspace of $E^q_b$ with dimension $k$ and denote $S^{k-1}$ the unit sphere of $W_k$, which is a compact set in $E^q_b$ and satisfies $i(S^{k-1})=k$ (see \cite{fadel-rabi}).
For $R,\rho>0$ to be chosen later, we set
\begin{eqnarray*}
    B=\{ u_\rho:u\in B_0\},\quad A=\{Ru:u\in S^{k-1}\}\quad\mbox{and}\quad X=\{tu:u\in A,\,0\leq t\leq 1\}.
\end{eqnarray*}
We will verify that we may choose $R>\rho>0$ such that $A\subset E^q_b\setminus I^{-1}[0,\rho]$ and
\begin{equation}\label{claim}
\sup_{v\in A}\Phi(v)\leq 0<\inf_{v\in B}\Phi(v)\quad\mbox{and}
\quad\sup_{v\in X}\Phi(v)<\infty.
\end{equation}
From \eqref{AR h} we have $|H(s)|\leq c_1|s|^{r_1}+c_2|s|^{r_2}$ for all $s\in\R$, for superscaled pairs $(\eta,r_i)$. Then, 
\[
\tilde{H}(u_t)=\int_{\mathbb{R}^N}\frac{H(u_t)}{|x|^{\eta}}dx=o(t^{\,\ell}),\quad\text{as}\,\,\, t\to0^+,
\]
uniformly in $u\in B_0$. See the proof of Theorem \ref{Th-superscaled-subcritical} for details. Since $(J(u))^{-1}=\tilde{\Psi}(u)\geq\lambda_{k_0}$ in $B_0$ we get
\begin{align*}
    \Phi(u_t)&={t^{\,\ell}}I(u)-\lambda t^{\,\ell} J(u)-\tilde{H}(u_t)\ge t^{\,\ell}\left(1-\frac{\lambda^+}{\lambda_{k_0}}+o(1)\right),\quad\text{as}\,\,\, t\to0^+ .
\end{align*}
Once $\lambda^+=\max\{\lambda,0\}<\lambda_{k_0}$, for $t=\rho>0$ sufficiently small it holds
\begin{equation}\label{inf in B rho subcri}
    \inf_{v\in B}\Phi(v)= \inf_{u\in B_0}\Phi(u_\rho)>0.
\end{equation}
On the other hand, due to \eqref{AR h} we have $rH(s)\geq c|s|^r$ for $s\in\R$. So 
\begin{align}\label{Phi subcri}
    \Phi(tu)
    &\leq\frac{t^{2}}2\int_{\mathbb{R}^N}|\nabla u|^{2}dx+\frac{t^{q}}{q}\int_{\mathbb{R}^N}\frac{|u|^{q}}{|x|^{b}}dx-\frac{\lambda t^{p}}{p}\int_{\mathbb{R}^N}\frac{|u|^{p}}{|x|^{a}}dx-\frac{ct^{r}}{r}\int_{\mathbb{R}^N}\frac{|u|^r}{|x|^{\eta}}dx .
\end{align}
The compactness of  $S^{k-1}$, and its boundedness in $L^{p}_a(\R^N)$ and in  $L^{r}_{\eta}(\R^N)$, ensures that 
$$
\frac{c}{r}\int_{\mathbb{R}^N}\frac{|u|^{r}}{|x|^{\eta}}dx\geq c_0\quad\text{and}\quad\frac{|\lambda|}{p}\int_{\mathbb{R}^N}\frac{|u|^{p}}{|x|^a}dx\leq \tilde{c}_1,\quad\forall  u\in S^{k-1},
$$
for some constants $c_0,\tilde c_1>0$. Recalling that $r>q>p>2$, from \eqref{Phi subcri} we obtain
 $$
\Phi(tu)\leq 
\frac{t^{2}}2+\frac{t^{q}}q+\tilde{c}_1t^{p}-c_0t^{r}
\leq0, \quad\forall u\in S^{k-1}, 
$$
whether $t=R$ is large enough, yielding also $A=RS^{k-1}\subset E^q_b\setminus I^{-1}[0,\rho]$. Hence, we get
\begin{equation}\label{sup in A0 subcri}
    \sup_{v\in A}\Phi(v)=\sup_{u\in S^{k-1}}\Phi(Ru)\leq0.
\end{equation}
Moreover, 
$$
\sup_{v\in X}\Phi(v)=
\sup_{t\in[0,1],\,u\in A} \Phi(tu)\leq \max_{t\ge0}\left( 
\frac{t^{2}}2+\frac{t^{q}}q+\tilde{c}_1t^{p}-c_0t^{r}\right)<\infty.
$$
Joining this estimate with \eqref{inf in B rho subcri} and \eqref{sup in A0 subcri} we get  \eqref{claim}. 
Therefore, using \cite[Propositions 3.42 and 3.44]{Perera-book} we obtain a sequence of positive critical values $(c^*_k)$, ${k\geq k_0}$, such that $c^*_k\nearrow\infty$. 
\end{proof-th}

\subsection{Existence and multiplicity for critical problems}

In the proofs of our next results, we will need some technical  estimates for the norm of \(u\) in the weighted Lebesgue spaces \(L^r_\eta\). These are the subject of the next lemmas.

\begin{lemma}\label{lemma-general-interpolation}
Suppose that $(\eta,r)$ is an admissible pair. Then, there exists an admissible pair $(\tilde\eta,\tilde r)$, with $b<\tilde\eta<2$, and $0<\theta<1$ such that the following interpolation inequality holds 
\begin{equation}\label{general-interpolation}
\left(\int_{\mathbb{R}^N} \frac{|u|^p}{|x|^a}\,dx\right)^\frac{1}{p}
\le
\left(\int_{\mathbb{R}^N} \frac{|u|^{r}}{|x|^{\eta}}\,dx\right)^\frac{\theta}{r}
\left(\int_{\mathbb{R}^N} \frac{|u|^{\tilde r}}{|x|^{\tilde\eta}}\,dx\right)^\frac{1-\theta}{\tilde r}.
\end{equation}
Moreover, the pair $(\tilde\eta,\tilde r)$ can be chosen so that it is scaled (respectively subscaled, superscaled) whenever $(\eta, r)$ is scaled (respectively superscaled, subscaled).  
\end{lemma}

\begin{proof} 
We use interpolation inequality with weights.  If $\eta=a$ and $r\neq p$ the conclusion is immediate: choose
$\tilde\eta=a$, $\max\{2^*_{b,q,a},1\}<\tilde r<2^*_a$ and  $0<\theta<1$ such that
\[
\frac{1}{p}=\frac{\theta}{r}+\frac{1-\theta}{\tilde r}
\]
and use the interpolation inequality in the Lebesgue spaces with the same weight, $|x|^{-a}$, to get \eqref{general-interpolation}.  Thus, consider $\eta\neq a$.  Here we must find an admissible pair $(\tilde\eta,\tilde r)$ and $0<\theta<1$ satisfying
\begin{equation}\label{interpolationbalance:scaled=subscaled+scaled}
\frac{1}{p}
=\frac{\theta}{r}+\frac{1-\theta}{\tilde r}
\qquad\text{and}\qquad
\frac{a}{p}
=\frac{\theta\eta}{r}
+\frac{(1-\theta)\tilde\eta}{\tilde r},
\end{equation}
so that using the interpolation inequality with distinct weights (see e.g. \cite[Inequality~(2.15)]{CKN-1984}), we get \eqref{general-interpolation}.
The identities in \eqref{interpolationbalance:scaled=subscaled+scaled} are equivalent to
\begin{equation}\label{interpolation-identities-rewritten}
\theta(\tilde\eta)
=
\frac{a-\tilde\eta}{p}\,\frac{r}{\eta-\tilde\eta}
\qquad\text{and}\qquad
\frac{1}{\tilde r}
=\frac{1}{p}
+\frac{\theta(\tilde\eta)}{1-\theta(\tilde\eta)}
\biggl(\frac{1}{p}-\frac{1}{r}\biggr).
\end{equation}
We distinguish two situations:

\medskip

\noindent
{Case (i): $\eta>a$.}
Pick $\tilde\eta<a$, and so, $\eta-\tilde\eta> \eta-a>0$.  
Then letting $\tilde\eta\nearrow a$ we see from \eqref{interpolation-identities-rewritten} that
\[
\theta(\tilde\eta)\searrow 0,
\qquad
2^*_{b,q,\tilde\eta}\searrow 2^*_{b,q,a},\qquad 2^*_{\tilde\eta}\searrow 2^*_{a}
\qquad
\tilde r\to p.
\]
Since $2^*_{b,q,a}<p<2^*_a$, we may choose $\tilde\eta$ sufficiently close to $a\in(b,2)$ so that
\[
b<\tilde\eta<2,\qquad 0<\theta(\tilde\eta)<1,\qquad \max\{1,2^*_{b,q,\tilde\eta}\}<\tilde r<2^*_{\tilde\eta}.
\]

\noindent
{Case (ii): $\eta<a$.}
The argument is completely analogous: now select $\tilde\eta>a$, 
sufficiently close to $a$ so that $b<\tilde\eta<2$,
$0<\theta(\tilde\eta)<1$ and  $ \max\{1,2^*_{b,q,\tilde\eta}\}<\tilde r<2^*_{\tilde\eta}$.

\medskip

 We need to locate the pair $(\tilde\eta,\tilde r)$ with respect to $(a,p)$. From \eqref{l(eta,r)} and \eqref{interpolation-identities-rewritten}, we get that
\begin{equation*}
{\ell-\ell(\tilde\eta,\tilde r)}
=
\frac{{\tilde r}}{r}\,\frac{\theta(\tilde\eta)}
     {1-\theta(\tilde\eta)}\,\big(\ell(\eta,r)-\ell\,\big).
\end{equation*}
 This means that the sign of the left hand side is ruled by the sign of $\ell(\eta,r)-\ell$, which gives exactly the conclusion we need: $(\tilde\eta,\tilde r)$ is scaled (subscaled) (superscaled) whenever $(\eta,r)$ is scaled (superscaled) (subscaled).
\end{proof}

    Using the interpolation inequality from Lemma \ref{lemma-general-interpolation} we get some uniform estimate in $L^r_\eta(\R^N)$ for functions in the set levels $\tilde{\Psi}^{\lambda}$.
 \begin{lemma}\label{interpolation estimates}
Suppose $(\eta,r)$ is an admissible pair. Then, for each
$\lambda\geq\lambda_1$ there exists a positive constant $c$, which depends on $\lambda,\eta,r$ and other fixed parameters, such that
\begin{align}\label{II general}
\frac{1}{r}\int_{\mathbb{R}^N}\frac{|u|^{r}}{|x|^\eta}dx\geq c,\quad\forall u\in\tilde{\Psi}^{\lambda}.
\end{align}
 \end{lemma}
\begin{proof}
By the definition of $\tilde{\Psi}$, given in \eqref{def Psi}, we have  $\tilde{\Psi}(u)\leq \lambda$ for all $\tilde{\Psi}^{\lambda}$. This implies that \eqref{II general} holds for $(\eta,r)=(a,p)$ with $c=1/\lambda$. 
Giving another admissible pair $(\eta,r)$, from Lemma \ref{lemma-general-interpolation} we get $(\tilde\eta,\tilde r)$, also admissible, and $\theta\in(0,1)$ satisfying \eqref{general-interpolation}. Due to the the embedding $E^q_b\hookrightarrow L^{\tilde r}_{\tilde\eta}(\R^N)$ we obtain
\begin{align*}
\frac{1}{p}\int_{\mathbb{R}^N}\frac{|u|^p}{|x|^a}dx\leq \frac{1}{p}\left(\int_{\mathbb{R}^N}\frac{|u|^r}{|x|^\eta}dx\right)^\frac{p\theta}{r}\left(\int_{\mathbb{R}^N}\frac{|u|^{\tilde{r}}}{|x|^{\tilde\eta}}dx\right)^\frac{p(1-\theta)}{\tilde{r}}
\leq C\left(\int_{\mathbb{R}^N}\frac{|u|^{r}}{|x|^\eta}dx\right)^\frac{\theta}{r},
\end{align*}
for some positive constant $C=C(N,b,q,p,\tilde\eta,\tilde r,\theta)$. Since  \eqref{II general} holds for $(\eta,r)=(a,p)$, we conclude that it occurs for any admissible pair $(\eta,r)$.
\end{proof}

For reader's convenience we recall the definition of pseudo-index. For any $\rho > 0$, we define
\begin{equation*}
\mathcal{M}_\rho = \{ u \in E^q_b : I(u) = \rho^{\ell} \} = \{ u_\rho : u \in \mathcal{M} \}.
\end{equation*}
Let $\Gamma$ be the group of odd homeomorphisms of $E^q_b$ that act as the identity outside the set $\Phi^{-1}(0, c^*)$, for some $c^*>0$ such that $\Phi$ satisfies the $(\mathrm{PS})_c$ condition for all $c \in (0, c^*)$. Let $\mathcal{A}^*$ be the family of symmetric subsets of $E^q_b$.
For any $M \in \mathcal{A}^*$, the pseudo-index $i^*(M)$ with respect to $i$, $\mathcal{M}_\rho$, and $\Gamma$ (see Benci \cite{V. Benci}) is defined 
\[
i^*(M) = \min_{\gamma \in \Gamma} i(\gamma(M) \cap \mathcal{M}_\rho).
\]

The proofs of Theorem \ref{th cri-super-mult} and Theorem \ref{th cri any lambda} will be based on the following
\begin{proposition}[\cite{MePe2025}, Theorem 2.33]\label{Theorem 2.33 MP}
Let $A_0$ and $B_0$ be symmetric subsets of $\mathcal{M}$ such that $A_0$ is compact, $B_0$ is closed, and
\begin{equation*}
i(A_0) \geq k + m - 1, \qquad i(\mathcal{M} \setminus B_0) \leq k - 1
\end{equation*}
for some $k, m \geq 1$. Let $R > \rho > 0$ and let
\[
X = \{u_t : u \in A_0,\, 0 \leq t \leq R \}, \quad
A = \{u_R : u \in A_0 \}, \quad
B = \{u_\rho : u \in B_0 \}.
\]
Assume that
\begin{equation*}
\sup_{u \in A} \Phi(u) \leq 0 < \inf_{u \in B} \Phi(u), \qquad
\sup_{u \in X} \Phi(u) < c^*.
\end{equation*}
For $j = k, \ldots, k + m - 1$, let
\[
\mathcal{A}^*_j = \{M \in \mathcal{A}^* : M \text{ is compact and } i^*(M) \geq j \}
\]
and set
\[
c^*_j := \inf_{M \in \mathcal{A}^*_j} \max_{u \in M} \Phi(u).
\]
Then $0 < c^*_k \leq \cdots \leq c^*_{k+m-1} < c^*$, each $c^*_j$ is a critical value of $\Phi$, and $\Phi$ has $m$ distinct pairs of associated critical points.
\end{proposition}

The next abstract result is a direct corollary of the preceding one and will be used in the proof of Theorem~\ref{Th-subscaled+scaled+critical}.

\begin{corollary}[\cite{MePe2025}, Corollary 2.34]\label{Corollary 2.34 MP}
Let $A_0$ be a compact symmetric subset of $\mathcal{M}$ with $i(A_0) = m \geq 1$, let $R > \rho > 0$, and let
\[
A = \{u_R : u \in A_0\}, \qquad X = \{u_t : u \in A_0,\, 0 \leq t \leq R\}.
\]
Assume that
\[
\sup_{u \in A} \Phi(u) \leq 0 < \inf_{u \in \mathcal{M}_\rho} \Phi(u), \qquad \sup_{u \in X} \Phi(u) < c^*.
\]
Then $\Phi$ has $m$ distinct pairs of critical points at levels in $(0, c^*)$.
\end{corollary}

Now we prove the result of multiplicity of solution for $\lambda$ close to some eigenvalue $\lambda_k.$

\begin{proof-th} \textbf{\textit{of Theorem \ref{th cri-super-mult}.}}\addcontentsline{toc}{subsubsection}{Proof of Theorem~\ref{th cri-super-mult}}
To prove this theorem we will apply Proposition \ref{Theorem 2.33 MP}. Let us then consider $\lambda\in(0,\lambda_k)$. By Lemma \ref{lemmaPS-criticalpower+superscaledpower} we know that $\Phi$ satisfies the $(PS)_c$ for all $c<c^*$, for a    positive constant $c^*$ depending on $N$ and $\eta_1$ (and $\mu,\eta_2$ if $r_2=2^*_{\eta_2}$). Fix $\varepsilon\in(0,\lambda_{k+m}-\lambda_{k+m-1})$. Then, by Theorem \ref{lambdak}\ref{lambdak(iii)} we have
$$
i(\mathcal{M}\backslash\tilde{\Psi}_{\lambda_{k+m-1}+\varepsilon})=k+m-1.
$$
Since $\mathcal{M}\backslash\tilde{\Psi}_{\lambda_{k+m-1}+\varepsilon}$ is an open symmetric subset of $\mathcal{M}$,  there is a compact symmetric $C_\varepsilon\subset\mathcal{M}\backslash\tilde{\Psi}_{\lambda_{k+m-1}+\varepsilon}$ with $i(C_\varepsilon)=k+m-1$  (see the proof of Proposition 3.1 in Degiovanni and Lancelotti \cite{De-Lan-2007}).
We set $A_0=C_\varepsilon$ and $B_0=\tilde{\Psi}_{\lambda_k}$. The properties of $C_\varepsilon$ yield
$$
i(A_0)=k+m-1.
$$
As far as $B_0,$ we distinguish the cases $\lambda_k=\lambda_1$ and $\lambda_k>\lambda_1$. If $\lambda_1=\cdots=\lambda_k$, we have
$$
B_0=\{u\in\mathcal{M}:\tilde{\Psi}(u)\geq\lambda_1\}=\mathcal{M},
$$
by Theorem \ref{lambdak}\ref{lambdak(i)}. Then, $i(\mathcal{M}\backslash B_0)=i(\emptyset)=0\leq k-1$. In the other case, it holds $\lambda_{l-1}<\lambda_l=\cdots=\lambda_k$ for some $2\leq l\leq k$. Then, Theorem \ref{lambdak}\ref{lambdak(iii)} implies that
$$
i(\mathcal{M}\backslash B_0)=
i(\mathcal{M}\backslash\tilde{\Psi}_{\lambda_{l}})=l-1\leq k-1.
$$
Now, for some $R>\rho>0$, we set
\begin{eqnarray*}
    X=\{u_t:u\in A_0,\,0\leq t\leq R\},\quad A=\{u_R:u\in A_0\}\quad\mbox{and}\quad B=\{u_{\rho}:u\in B_0\},
\end{eqnarray*}
and show that it is possible to choose suitable $R,\rho$ and $\delta_k>0$ in such a way that, for all $\lambda\in(\lambda_k-\delta_k,\lambda_k)$, it holds that
$$
\sup_{v\in A}\Phi(v)\leq 0<\inf_{v\in B}\Phi(v)\quad\mbox{and}
\quad\sup_{v\in X}\Phi(v)<c^*.
$$
Let us recall that $\ell=\delta q+b-N<\ell_1,\ell_2$, where $\ell_i=\ell(\eta_i,r_i)=\delta r_i+\eta_i-N$.  For all $\lambda\in\mathbb{R}$, $u\in\mathcal{M}$ and $t\geq0$ we have
\begin{equation}\label{Phi cri in M}
    \Phi(u_t)=t^{\,\ell}\left(1-\frac{\lambda}{\tilde{\Psi}(u)}\right)-\frac{\mu t^{\,\ell_2}}{r_2}\int_{\mathbb{R}^N}\frac{|u|^{r_2}}{|x|^{\eta_2}}dx-\frac{t^{\,\ell_1}}{ 2^*_{\eta_1}}\int_{\mathbb{R}^N}\frac{|u|^{ 2^*_{\eta_1}}}{|x|^{\eta_1}}dx.
\end{equation}
Since $\mathcal{M}$ is bounded, there exist constants $c_1,c_2>0$ such that
$$
\Phi(u_t)\geq 
t^{\,\ell}\left(1-\frac{\lambda}{\tilde{\Psi}(u)}-c_1\mu t^{\,\ell_2-\ell}-c_2t^{\,\ell_1-\ell}\right)\geq 
\frac{t^{\,\ell}}{2}\left(1-\frac{\lambda}{\lambda_k}\right)>0,\quad \forall u\in B_0,
$$
since $\tilde{\Psi}(u)\geq\lambda_k>\lambda$, 
provided $t=\rho$ is small enough. 
This gives us that
\begin{equation}\label{inf in B}
    \inf_{v\in B}\Phi(v)>0.
\end{equation}
Now, recall that $\lambda_{k+m-1}+\varepsilon<\lambda_{k+m}$ and  $A_0\subset\mathcal{M}\backslash\tilde{\Psi}_{\lambda_{k+m-1}+\varepsilon}\subset\mathcal{M}\backslash\tilde{\Psi}_{\lambda_{k+m}}$.
Then, from Lemma \ref{interpolation estimates} we know that there exists $c_3>0$, independent on $\varepsilon$, such that
$$
\frac{1}{2^*_{\eta_1}}\int_{\mathbb{R}^N}\frac{|u|^{2^*_{\eta_1}}}{|x|^{\eta_1}}dx\geq c_3,\quad\forall u\in A_0.
$$
Since $\lambda,\mu\geq0$, by \eqref{Phi cri in M} we obtain
 $$
\Phi(u_t)\leq t^{\,\ell}-\frac{t^{\,\ell_1}}{2^*_{\eta_1}}\int_{\mathbb{R}^N}\frac{|u|^{2^*_{\eta_1}}}{|x|^{\eta_1}}dx\leq 
t^{\,\ell}\left(1-c_3t^{\,\ell_1-\ell}\right)\leq0, \quad\forall u\in A_0,
$$
for $t=R>\rho$, large enough. Hence,
\begin{equation}\label{sup in A 2} 
    \sup_{v\in A}\Phi(v)\leq0.
\end{equation}
It remains to estimate the supremum of $\Phi$ in $X$. Again by  \eqref{Phi cri in M}, since $\lambda,\mu\geq0$, we get
\begin{equation*}
    \Phi(u_t)\leq t^{\,\ell}\left(1-\frac{\lambda}{\lambda_{k}+\varepsilon}\right)-c_3t^{\,\ell_1}, \quad\forall u\in A_0, \,t\geq0.
\end{equation*}
Thus, 
\begin{eqnarray*}
   \sup_{u\in A_0, 0\leq t\leq R}\Phi(u_t)&\leq &\max_{t\geq0}\left[t^{\,\ell}\left(1-\frac{\lambda}{\lambda_{k}+\varepsilon}\right)-c_3t^{\,\ell_1}\right]
    =\frac{\ell_1-\ell}{\ell_1}\left(\frac{\ell}{c_3\ell_1}\right)^{\frac{\ell}{\ell_1-\ell}}\left(1-\frac{\lambda}{\lambda_{k}+\varepsilon}\right)^\frac{\ell_1}{\ell_1-\ell}\\
    &=&\frac{2-\eta_1}{N-\eta_1}\left(\frac{N-2}{c_3(N-\eta_1)}\right)^\frac{N-2}{2-\eta_1}\left(1-\frac{\lambda}{\lambda_{k}+\varepsilon}\right)^\frac{N-\eta_1}{2-\eta_1}
\end{eqnarray*}
Choosing
$$
\delta_k=\lambda_k\left(c^*\frac{N-\eta_1}{2-\eta_1}\right)^\frac{2-\eta_1}{N-\eta_1}\left(\frac{c_3(N-\eta_1)}{N-2}\right)^\frac{N-2}{N-\eta_1},
$$
 we can see that for each $\lambda_k-\delta_k<\lambda<\lambda_k$, there exists $\varepsilon>0$ small enough 
 such that
\begin{eqnarray*}
   \sup_{u\in A_0, 0\leq t\leq R}\Phi(u_t)<c^*.
\end{eqnarray*}
Finally, combining this inequality, \eqref{inf in B} and \eqref{sup in A 2},
the conclusion now follows by Proposition \ref{Theorem 2.33 MP}.
\end{proof-th}

In the next result $\lambda\in\R$ can be any number, even an eigenvalue for \eqref{NE problem}, and does not need to be close to an eigenvalue. However, we deal with subcritical perturbation $\mu|x|^{-\eta}|t|^{r_2-2}t$ and consider $\mu$ a large number.

\begin{proof-th}
    \textbf{\textit{of Theorem \ref{th cri any lambda}.}}\addcontentsline{toc}{subsubsection}{Proof of Theorem~\ref{th cri any lambda}}
This proof is similar to the proof of previous theorem.  Since $r_2<2^*_{\eta_2}$, by Lemma \ref{lemmaPS-criticalpower+superscaledpower} we know that $\Phi$ satisfies the $(PS)_c$ for all $c<c^*(N,\eta_1)$. 
Given  $m\in\mathbb{N}$ we set 
 $$
 k_0=\min\{k\in\mathbb{N}: \lambda<\lambda_k\}\quad\mbox{and}\quad \tilde{m}=k_0+m-1.
 $$
Let  $j\in\mathbb{N}$ and $\alpha\in\mathbb{R}$ be such that $\lambda_{\tilde m}=\cdots=\lambda_{\tilde m+j-1}<\alpha<\lambda_{\tilde m+j}$.   Then, by Theorem \ref{lambdak}\ref{lambdak(iii)} we have
$$
i(\mathcal{M}\backslash\tilde{\Psi}_{\alpha})=\tilde m+j-1.
$$
Since $\mathcal{M}\backslash\tilde{\Psi}_{\alpha}$ is an open symmetric subset of $\mathcal{M}$,  there exists a compact symmetric set $C\subset\mathcal{M}\backslash\tilde{\Psi}_{\alpha}$ with $i(C)=\tilde m+j-1$  (see the proof of Proposition 3.1 in Degiovanni and Lancelotti \cite{De-Lan-2007}).
Set
$$
A_0=C\quad \mbox{and}\quad B_0=\tilde{\Psi}_{\lambda_{k_0}}.
$$ 
Hence, analyzing the cases $k_0=1$ or $k_0\geq 2$ for $B_0$, as done in the proof of Theorem \ref{Th-superscaled-subcritical-multi}, we see that
$$
i(A_0)=\tilde m+j-1=(m+j-1)+k_0-1\quad\text{and}\quad 
i(\mathcal{M}\backslash B_0)=k_0-1.
$$
As before, for some $R>\rho>0$ to be chosen later, we set
\begin{eqnarray*}
    X=\{u_t:u\in A_0,\,0\leq t\leq R\},\quad A=\{u_R:u\in A_0\}\quad\mbox{and}\quad B=\{u_{\rho}:u\in B_0\}.
\end{eqnarray*}
We claim that we may choose $R,\,\mu_m>0$ such that, for each $\mu\geq \mu_m$ there is $\rho\in(0,R)$ such that
$$
\sup_{v\in A}\Phi(v)\leq 0<\inf_{v\in B}\Phi(v)\quad\mbox{and}
\quad\sup_{v\in X}\Phi(v)<c^*.
$$
Recalling \eqref{l(eta,r)}, we have $\ell<\ell_1,\ell_2$, where $\ell_1=\ell(\eta_1,2^*_{\eta_1})$ and $\ell_2=\ell(\eta_2,r_2)$.  For all $\lambda\in\mathbb{R}$, $u\in\mathcal{M}$ and $t\geq0$ we have
\begin{equation}\label{Phi cri MPT}
    \Phi(u_t)=t^{\,\ell}\left(1-\frac{\lambda}{\tilde{\Psi}(u)}\right)-\frac{\mu t^{\,\ell_2}}{r_2}\int_{\mathbb{R}^N}\frac{|u|^{r_2}}{|x|^{\eta_2}}dx-\frac{t^{\,\ell_1}}{ 2^*_{\eta_1}}\int_{\mathbb{R}^N}\frac{|u|^{ 2^*_{\eta_1}}}{|x|^{\eta_1}}dx.
\end{equation}
Due to the compactness of  $A_0\subset\mathcal{M}$, there exists $c_0>0$ such that 
$$
\frac{1}{p}\int_{\mathbb{R}^N}\frac{|u|^{p}}{|x|^{a}}dx,\quad
\frac{1}{r_2}\int_{\mathbb{R}^N}\frac{|u|^{r_2}}{|x|^{\eta_2}}dx,\quad\frac{1}{ 2^*_{\eta_1}}\int_{\mathbb{R}^N}\frac{|u|^{ 2^*_{\eta_1}}}{|x|^{\eta_1}}dx \geq c_0,\quad\forall u\in A_0.
$$
Denoting $\lambda^-=-\min\{\lambda,0\}$, by the boundedness of $\mathcal{M}$ in $L^{p}_a(\R^N)$ we get
$$
1-\frac{\lambda}{\tilde{\Psi}(u)}\leq 1+\frac{\lambda^-}{p}\int_{\mathbb{R}^N}\frac{|u|^{p}}{|x|^a}dx\leq \beta,\quad\forall u\in A_0,
$$
for some $\beta>0$.  Since $\mu>0$, from \eqref{Phi cri MPT} we obtain
 $$
\Phi(u_t)\leq 
\beta t^{\,\ell}-\frac{t^{\,\ell_1}}{ 2^*_{\eta_1}}\int_{\mathbb{R}^N}\frac{|u|^{ 2^*_{\eta_1}}}{|x|^{\eta_1}}dx\leq \beta t^{\,\ell}-c_0t^{\,\ell_1}
\leq0, \quad\forall u\in A_0,
$$
whether $t=R$, independent of $\mu$, is large enough. Hence
\begin{equation}\label{sup in A}
    \sup_{v\in A}\Phi(v)\leq0.
\end{equation}
To estimate the supremum of $\Phi$ in $X$, using again  \eqref{Phi cri MPT} we obtain
\begin{eqnarray*}
   \sup_{u\in A_0, 0\leq t\leq R}\Phi(u_t)&\leq &\max_{t\geq0}\left[\beta t^{\,\ell}-\mu c_0t^{\,\ell_2}\right]=\frac{\ell_2-\ell}{\ell_2}\left(\frac{\beta^\frac{\ell_2}{\ell}\ell}{c_0\mu\ell_2}\right)^{\frac{\ell}{\ell_2-\ell}}.
\end{eqnarray*}
Picking $\mu_m>0$ sufficiently large, for $\mu\geq\mu_m$ it holds that
\begin{eqnarray}\label{sup in X lambda}
   \sup_{v\in X}\Phi(v)= \sup_{u\in A_0, 0\leq t\leq R}\Phi(u_t)<c^*.
\end{eqnarray}
Now we estimate $\Phi$ on $B$. By \eqref{Phi cri MPT}, due to the boundedness of $\mathcal{M}$, there exist $c_1,c_2>0$ such that
$$
\Phi(u_t)\geq 
t^{\,\ell}\left(1-\frac{\lambda^+}{\tilde{\Psi}(u)}-c_1 \mu t^{\,\ell_2-\ell}-c_2t^{\,\ell_1-\ell}\right),
$$
where $\lambda^+=\max\{\lambda,0\}$. Hence, as $\tilde{\Psi}(u)\geq\lambda_{k_0}>\lambda^+$ for all $u\in B_0$ and $\ell_1,\ell_2>\ell$,  we can choose $t=\rho\in(0,R)$ small enough such that
\begin{equation*}
  \inf_{u\in B_0}\Phi(u_\rho)=  \inf_{v\in B}\Phi(v)>0.
\end{equation*}
Due to this last inequality, \eqref{sup in A} and \eqref{sup in X lambda}, the Proposition \ref{Theorem 2.33 MP} can be applied to finish this proof.
\end{proof-th}

\subsection{Problems with subscaled nonlinearities}

We now focus on nonlinearities $f$ that display at  least a subscaled pair $(\eta, r)$.  For the proofs of Theorems \ref{Th-single-subscaled}-\ref{Th-subscaled+critical}, we make use of an abstract result in critical point theory (see \cite[Proposition~3.36]{Perera-book}), which we state here for reader's convenience.
 
\begin{proposition}\label{critical point result}
    Let $\Phi$ be an even $C^1-$functional on a Banach space $W$ such that $\Phi(0) = 0$
and $\Phi$ satisfies the $(PS)_c$ condition for all $c<0$. Let $\mathcal{F}$ denote the class of symmetric subsets
of $W \backslash\{0\}$. For $k\geq1$, let
$$
\mathcal{F}_k = \{M \in \mathcal{F}: i(M) \geq k\}
$$
and set
\begin{equation}\label{ck sub near 0}
    c_k:=\inf_{M\in \mathcal{F}_k}\sup_{u\in M}\Phi(u).
\end{equation}
If there exists a $k_0 \geq1$ such that $-\infty<c_k <0$ for all $k \geq k_0$, then $c_{k_0}\leq c_{k_0+1}\leq\cdots  \to0$ is a sequence of critical values of $\Phi$.
\end{proposition}

\begin{proof-th}
    \textbf{\textit{of Theorem \ref{Th-single-subscaled}.}}\addcontentsline{toc}{subsubsection}{Proof of Theorem~\ref{Th-single-subscaled}}   
  To proceed, let us recall the notations given in Definition \ref{defscaled} and \eqref{l(eta,r)}. Since $(\eta,r)$ is a subscaled pair, we see that the potential operator $\tilde f$ induced by the nonlinearity $f(x,t)=|x|^{-\eta}|t|^{r-2}t$ satisfies
  $$
  \tilde f(u_t)v_t=\int_{\mathbb R^N}\frac{|u_t|^{r-2}u_tv_t}{{|x|^\eta}}dx=o(t^{\,\ell})\|v\|\quad\mbox{as}\,\,t\to\infty,
  $$
uniformly in $u$ on bounded sets of $E^q_b$, for all $v\in E^q_b$. Then, it is straightforward to verify that the associated functional, which in this case is given by 
\[\displaystyle\Phi(u)=\frac{1}{2}\int_{\R^N}|\nabla u|^2dx+\frac{1}{q}\int_{\mathbb{R}^N} \frac{|u|^q}{|x|^b} dx
-\frac{1}{r}\int_{\mathbb{R}^N} \frac{|u|^r}{|x|^\eta} dx,\]
is coercive (see \cite[Lemma 2.15]{MePe2025} for a proof in a more general setting). Therefore, for any $c\in\mathbb R$, $(PS)_c$ sequences are always bounded. The compactness of the embedding $E^q_b\hookrightarrow L^r_\eta(\R^N)$ (see Corollary \ref{subscaled compact}) implies that $\Phi$ satisfies the $(PS)_c$ conditions for all levels $c\in\mathbb R$, as a consequence of \cite[Proposition 2.3]{MePe2025}.

The boundedness of ${\Phi}$ on bounded sets implies that ${\Phi}$ is bounded from below in $E^q_b$. Hence, the levels $c_k$ in \eqref{ck sub near 0} satisfy
$$
c_k\geq -C, \quad\mbox{for all}\,\, k\geq1,
$$
for some positive constant $C$. Let us show that $c_k<0$ for all $k\geq 1$. For any $u\in \mathcal M$, we have 
\begin{eqnarray*}
   \Phi(u_t)&=&t^{\,\ell} - \frac{t^{\,\ell(\eta,r)}}{r}\int_{\mathbb{R}^N}\frac{|u|^{r}}{|x|^\eta}dx
\end{eqnarray*}
and since $\ell(\eta,r)<\ell$, this implies that
\begin{eqnarray}\label{Phi <0 near origin}
   \Phi(u_t)&=&-t^{\,\ell(\eta,r)}\left(\frac{1}{r}\int_{\mathbb{R}^N}\frac{|u|^{r}}{|x|^\eta}dx+o(1)\right)
\end{eqnarray}
with $o(1)\to 0$ as $t\to 0^+$ uniformly in  $u\in\mathcal{M}$. Now, for $k\geq1$, in order to estimate the minimax level $c_k$ from above, let us set
$$
M_t=\{u_t : u\in\tilde{\Psi}^{\lambda_k}\},\quad\mbox{for}\quad t>0.
$$
We have $i(M_t)=i(\tilde{\Psi}^{\lambda_k})$ as $u\mapsto u_t$ with $u\in \tilde{\Psi}^{\lambda_k}$, is an odd homeomorphism.  Let $m\geq1$ be such that $\lambda_k=\cdots=\lambda_{k+m-1}<\lambda_{k+m}$.
By Theorem \ref{lambdak}\ref{lambdak(iii)} we have $i(\tilde{\Psi}^{\lambda_k})=k+m-1$. This implies that $i(M_t)=k+m-1\geq k,$ and so $M_t\in\mathcal{F}_k$. Using \eqref{Phi <0 near origin} and Lemma \ref{interpolation estimates} we see that
\begin{eqnarray*}
   \Phi(u_t)\leq  - t^{\,\ell(\eta,r)}\left(c+o(1)\right) <0\quad\mbox{for all}\quad u\in \tilde{\Psi}^{\lambda_k}
\end{eqnarray*}
for $t>0$ sufficiently small. 
Then 
\begin{equation*}
    c_k=\inf_{M\in \mathcal{F}_k}\sup_{u\in M}\Phi(u)\leq \sup_{u\in M_t}\Phi(u)=\sup_{u\in \tilde{\Psi}^{\lambda_k}}\Phi(u_t)<0,
\end{equation*}
for all $k\geq1$. By Proposition \ref{critical point result}, this concludes the proof.
\end{proof-th}

\begin{proof-th}
    \textbf{\textit{of Theorem \ref{Th-subscased+scaled+superscaled1}.}}\addcontentsline{toc}{subsubsection}{Proof of Theorem~\ref{Th-subscased+scaled+superscaled1}} The associated functional in this case is given by 
\[\displaystyle\Phi(u)=\frac{1}{2}\int_{\R^N}|\nabla u|^2dx+\frac{1}{q}\int_{\mathbb{R}^N} \frac{|u|^q}{|x|^b} dx
-\frac{\lambda}{p}\int_{\mathbb{R}^N} \frac{|u|^p}{|x|^a} dx-\frac{1}{r}\int_{\mathbb{R}^N} \frac{|u|^r}{|x|^\eta} dx+\frac{1}{r_1}\int_{\mathbb{R}^N} \frac{|u|^{r_1}}{|x|^{\eta_1}}dx.\]
Then, if $\lambda\le 0$ we get
\[\displaystyle\Phi(u)\ge \frac{1}{2}\int_{\R^N}|\nabla u|^2dx+\frac{1}{q}\int_{\mathbb{R}^N} \frac{|u|^q}{|x|^b} dx
-\frac{1}{r}\int_{\mathbb{R}^N} \frac{|u|^r}{|x|^\eta} dx,\]
and the right hand side is exactly the expression of the functional in the proof of Theorem \ref{Th-single-subscaled}, which is coercive. So, $\Phi$ is also coercive. On the other hand, if $\lambda>0$ then we assume that the pair $(\eta_1,r_1)$ is superscaled. By the interpolation given in Lemma \ref{lemma-general-interpolation} and Young's inequality, we can choose  $\tilde C>0$ and a subscaled pair $(\tilde\eta,\tilde r)$ with $b<\tilde\eta<2$ such that
\begin{equation*}
\frac{\lambda}{p}\int_{\mathbb{R}^N} \frac{|u|^p}{|x|^a} \, dx
\;\le\;
\frac{1}{r_1}\int_{\mathbb{R}^N} \frac{|u|^{r_1}}{|x|^{\eta_1}}\,dx
\;+\;
\tilde C \int_{\mathbb{R}^N} \frac{|u|^{\tilde r}}{|x|^{\tilde\eta}}\,dx.
\end{equation*}
This gives us
\[\displaystyle\Phi(u)\ge\tilde\Phi(u):= \frac{1}{2}\int_{\R^N}|\nabla u|^2dx+\frac{1}{q}\int_{\mathbb{R}^N} \frac{|u|^q}{|x|^b} dx
-\frac{1}{r}\int_{\mathbb{R}^N} \frac{|u|^r}{|x|^\eta} dx-\tilde C \int_{\mathbb{R}^N} \frac{|u|^{\tilde r}}{|x|^{\tilde\eta}}\,dx.\]
Notice that, as in the previous case, $\tilde\Phi$ is a coercive functional since both pairs $(\eta,r)$ and $(\tilde\eta,\tilde r)$ are subscaled. This implies that $\Phi$ is also coercive. Thus, $(PS)_c$ sequences  for $\Phi$ are bounded.  Let $(u_n)$ be a $(PS)_c$ sequence, for some $c\in\mathbb R$.  Since $(u_n)$ is bounded, let $u\in E^q_b$ be its weak limit, up to a subsequence. Then, we reason as in  the proof of Lemma \ref{lemmaPS-criticalpower+superscaledpower} up to \eqref{PSeq1-1critical} and reach
\begin{equation*}
\int_{\mathbb R^N}|\nabla(u_n-u)|^2dx+\int_{\mathbb R^N}\frac{|u_n-u|^q}{|x|^b}dx=-\int_{\mathbb R^N}\frac{|u_n-u|^{r_1}}{|x|^{\eta_1}}dx+o(1).
\end{equation*}
This gives that $(u_n)$ converges strongly to $u$ both in $E^q_b$ and $L^{r_1}_{\eta_1}(\mathbb R^N)$.
Since ${\Phi}$ is bounded from below in $E^q_b$, as in the proof of the previous theorem, the levels $c_k$ in \eqref{ck sub near 0} are bounded from below. We just need to check that $c_k<0$ for all $k\geq 1$. Notice that now we have, for any $u\in \mathcal M$,
\begin{eqnarray*}
   \Phi(u_t)&=&t^{\,\ell} - \frac{t^{\,\ell(\eta,r)}}{r}\int_{\mathbb{R}^N}\frac{|u|^{r}}{|x|^\eta}dx - \frac{\lambda t^{\,\ell}}{p}\int_{\mathbb{R}^N}\frac{|u|^{p}}{|x|^a}dx+\frac{t^{\,\ell(\eta_1,r_1)}}{r_1}\int_{\mathbb{R}^N}\frac{|u|^{r_1}}{|x|^{\eta_1}}dx
\end{eqnarray*}
and since $\ell(\eta,r)<\ell$ and $\ell(\eta,r)<\ell(\eta_1,r_1)$, this implies that
\begin{eqnarray*}
   \Phi(u_t)&=&-t^{\,\ell(\eta,r)}\left(\frac{1}{r}\int_{\mathbb{R}^N}\frac{|u|^{r}}{|x|^\eta}dx+o(1)\right)
\end{eqnarray*}
with $o(1)\to 0$ as $t\to 0^+$ uniformly in  $u\in\mathcal{M}$. This is exactly the same identity as in \eqref{Phi <0 near origin}, and therefore the remainder of the argument coincides with the proof of Theorem~\ref{Th-single-subscaled}.
\end{proof-th}

\begin{proof-th}
    \textbf{\textit{of Theorem \ref{Th-subscased+scaled}.}}\addcontentsline{toc}{subsubsection}{Proof of Theorem~\ref{Th-subscased+scaled}}
    The associated functional, defined in \eqref{functionals}, is in this setting given by
\[
\Phi(u)=\Phi_\lambda(u)- \frac{1}{r}\int_{\mathbb{R}^N}\frac{|u|^{r}}{|x|^\eta}dx.
\]
Here $f(x,t)=\lambda|x|^{-a}|t|^{p-2}t+h(x,t)$, for $h(x,t)=|x|^{-\eta}|t|^{r-2}t$, and  the associated potential operator
$\tilde f=\lambda\mathcal{B}+\tilde{h}$ is compact (due to Corollary \ref{corollary_emb} and \ref{subscaled compact}) with 
  $$
  \tilde h(u_t)v_t=\int_{\mathbb R^N}\frac{|u_t|^{r-2}u_tv_t}{{|x|^\eta}}dx=o(t^{\,\ell})\|v\|\quad\mbox{as}\,\,t\to\infty,
  $$
uniformly in $u$ on bounded sets of $E^q_b$, for all $v\in E^q_b$. So, the $(PS)$ condition holds provided $\lambda\notin\sigma(\mathcal{A},\mathcal{B})$. This is proved in a more general framework in \cite[Lemma~2.26]{MePe2025}.  We also note that it suffices to consider the case \(\lambda > 0\), since when \(\lambda \le 0\) the functional is coercive and the argument follows exactly as in the proof of Theorem~\ref{Th-single-subscaled}. Moreover, if we consider $c_k$ as in \eqref{ck sub near 0},
the estimate $c_k<0$  is also exactly the same, since here we have,  for all $u\in\mathcal{M}$,
\begin{eqnarray}\label{Phi <0 near origin 2}
   \Phi(u_t)&=&t^{\,\ell}\left(1-\frac{\lambda}{\tilde{\Psi}(u)}\right) - \frac{t^{\,\ell(\eta,r)}}{r}\int_{\mathbb{R}^N}\frac{|u|^{r}}{|x|^\eta}dx
   = -t^{\,\ell(\eta,r)}\left(\frac{1}{r}\int_{\mathbb{R}^N}\frac{|u|^{r}}{|x|^\eta}dx+o(1)\right)
\end{eqnarray}
with $o(1)\to 0$ as $t\to 0^+$ uniformly in  $u\in\mathcal{M}$. 

The main difference is that we need to estimate $c_k$ for some $k\geq k_0\geq 1$ from below. Coercivity is no longer available since we are assuming $\lambda>0$. Note that, as $\lambda$ is not an eigenvalue, we may have $0<\lambda<\lambda_1,$ or $\lambda_{k-1}<\lambda<\lambda_{k}$ for some $k\geq2$. Set \(k_0:=\min\{k\in\mathbb{N}: \lambda<\lambda_k\}.\)
By the definition of $\mathcal{F}_k$ in Proposition \ref{critical point result},  for each $M\in\mathcal{F}_k$ we have
\begin{equation}\label{M in Fk}
    i(M)\geq k\geq k_0,
\end{equation}
for all $k\geq k_0\geq 1$. Setting $Y=\{u_t : u\in\tilde{\Psi}_{\lambda_{k_0}},\,t\geq0\},$ we claim that 
\begin{equation}\label{M Y}
    M\cap Y\neq\emptyset \quad\mbox{for all}\,\, M\in\mathcal{F}_k\,\,\mbox{and}\,\,k\geq k_0.
\end{equation}
In fact, if this were not true, we would have $M\cap Y=\emptyset$ for some $M\in\mathcal{F}_k$ and $k\geq k_0$. So $\pi_{|_M}:M\to \mathcal{M}\backslash \tilde{\Psi}_{\lambda_{k_0}}$
would be an odd continuous map, and hence we would have
$
i(M)\leq i(\mathcal{M}\backslash \tilde{\Psi}_{\lambda_{k_0}}).
$
On the other hand, if $\lambda<\lambda_1$, then $k_0=1$ and by Theorem \ref{lambdak}\ref{lambdak(i)} we have $\tilde{\Psi}_{\lambda_{1}}=\mathcal{M}$ so that
$
i(\mathcal{M}\backslash \tilde{\Psi}_{\lambda_{1}})=i(\emptyset)=0,
$
whereas if $\lambda_{k_0-1}<\lambda<\lambda_{k_0}$ with $k_0\geq2$, the same Theorem \ref{lambdak}\ref{lambdak(iii)} yields
$
i(\mathcal{M}\backslash \tilde{\Psi}_{\lambda_{k_0}})=k_0-1.
$
In both cases, we would conclude that
$$
i(M)\leq i(\mathcal{M}\backslash \tilde{\Psi}_{\lambda_{k_0}})\leq k_0-1,
$$
which contradicts \eqref{M in Fk}. Thus, \eqref{M Y} holds. Therefore, for each ${M\in \mathcal{F}_k}$
$$
\sup_{u\in M}\Phi(u)\geq \sup_{u\in M\cap Y}\Phi(u)\geq\inf_{u\in M\cap Y}\Phi(u)
\geq \inf_{u\in  Y}\Phi(u). 
$$
Since $\tilde{\Psi}(u)\geq\lambda_{k_0}>\lambda$ for all $u\in\tilde{\Psi}_{\lambda_{k_0}}$ and $\mathcal{M}$ is bounded in $L^r_\eta(\mathbb{R}^N)$ and in $L^{p}_a(\mathbb{R}^N)$, the first equality in \eqref{Phi <0 near origin 2} yields $\inf_{u\in  Y}\Phi(u):=-C>-\infty$. 
Hence
$$ 
c_k=\inf_{M\in \mathcal{F}_k}\sup_{u\in M}\Phi(u)\geq \inf_{u\in  Y}\Phi(u)\geq -C,\quad\forall k\geq k_0,
$$
and by Proposition \ref{critical point result}  we may conclude the proof. \end{proof-th}

\begin{proof-th}
    \textbf{\textit{of Theorem \ref{Th-subscaled+critical}.}}\addcontentsline{toc}{subsubsection}{Proof of Theorem~\ref{Th-subscaled+critical}} 
This proof also applies Proposition \ref{critical point result}. Here, as in \cite{MePe2025}, we use a truncation of the functional $\Phi_\mu$ associated with this problem, which is
 \begin{eqnarray*}
  \Phi_\mu(u)=\frac{1}{2}\int_{\mathbb{R}^N}|\nabla u|^2dx+\frac{1}{q}\int_{\mathbb{R}^N}\frac{|u|^q}{|x|^{b}}dx - \frac{\mu}{r}\int_{\mathbb{R}^N}\frac{|u|^r}{|x|^{\eta}}dx-\frac{1}{2^*_{\eta_1}}\int_{\mathbb{R}^N}\frac{|u|^{2^*_{\eta_1}}}{|x|^{\eta_1}}dx.
\end{eqnarray*}
We verify that for any admissible pair $(\eta,r)$, since $r\geq2^*_{b,q,\eta}$, there exists $C>0$ such that
\begin{equation}\label{Lr by I(u)}
\int_{\mathbb{R}^N}\frac{|u|^r}{|x|^\eta}dx\leq C[I(u)]^\frac{\ell(\eta,r)}{\ell},\quad\forall u\in E^q_b.
\end{equation}
Then
\begin{eqnarray*}
  \Phi_\mu(u)&\geq& I(u)- {C}\mu[I(u)]^\frac{\ell(\eta,r)}{\ell}-{C}_1[I(u)]^\frac{\ell(\eta_1,2^*_{\eta_1})}{\ell}\quad\forall u\in E^q_b.
\end{eqnarray*}
Let us denote $\gamma_\mu(t)=t- {C}\mu t^\frac{\ell(\eta,r)}{\ell}-{C}_1t^\frac{\ell(\eta_1,2^*_{\eta_1})}{\ell}$, for $t\geq0$. So, we can write 
\begin{eqnarray*}
  \Phi_\mu(u)&\geq& \gamma_\mu(I(u)),\quad\forall u\in E^q_b.
\end{eqnarray*}
Once $\ell(\eta_1,2^*_{\eta_1})>\ell>\ell(\eta,r)>0$ 
there is $\mu^*>0$ such that for each $\mu\in(0,\mu^*)$  it is possible to find  $R_1(\mu)$ and $R_2(\mu)>0$ such that
\[
\gamma_\mu(t)<0,\,\,\forall t\in[0,R_1(\mu))\cup(R_2(\mu),\infty)\qquad\mbox{and}\qquad \gamma_\mu(t)\geq0,\,\,\forall t\in[R_1(\mu)),(R_2(\mu)].
\]
For $\mu\in(0,\mu^*)$ it holds that $\gamma_\mu(t)\geq \gamma_{\mu^*}(t)$ for all $t\in[R_1(\mu^*)),(R_2(\mu^*)]$, hence, $R_1(\mu)\leq R_1(\mu^*)$ and $R_2(\mu^*)\leq R_2(\mu)$.
Consider now $\xi_\mu:[0,\infty)\to[0,1]$ smooth and satisfying 
$$
\xi_\mu\equiv1\,\,\mbox{in}\,\, [0,R_1(\mu)]\quad\mbox{and}\quad \xi_\mu\equiv0\,\,\mbox{in}\,\,[R_2(\mu),\infty),
$$
and define the truncated functional $\tilde{\Phi}_\mu:E^q_b\to \mathbb{R}$ by
$$
\tilde{\Phi}_\mu(u)=\xi_\mu(I(u)){\Phi}_\mu(u).
$$
Note that if $\tilde{\Phi}_\mu(u)<0$ then $I(u)\in(0,R_1(\mu)),$ and, by the continuity of $I$, $I(v)\in(0,R_1(\mu))$ for $v$ in some neighborhood of $u$ in $E^q_b$, which imply
\begin{equation}\label{truncation<0}
    \tilde{\Phi}_\mu(u)={\Phi}_\mu(u)<0\quad\mbox{and}\quad \tilde{\Phi}'_\mu(u)={\Phi}'_\mu(u).
\end{equation}
Hence critical points for $\tilde{\Phi}_\mu$ at negative levels are actual critical points for ${\Phi}_\mu$. Thus, to conclude this proof is enough to apply Proposition \ref{critical point result} to find a sequence of critical points for $\tilde{\Phi}_\mu$ whose energy is negative. Our first step is to show that $\tilde{\Phi}_\mu$ satisfies the $(PS)_c$ condition for $c<0$, provided $\mu\in(0,\mu^*)$ for some $\mu^*$ sufficiently small. Let us consider then a sequence $(u_n)$ in $E^q_b$ such that
$$
\tilde{\Phi}_\mu(u_n)\to c<0\quad\text{and}\quad \tilde{\Phi}'_\mu(u_n)\to 0.
$$
Since $\tilde{\Phi}_\mu(u_n)<0$ for large $n$, we have $I(u_n)\in(0,R_1(\mu^*))$, which forces $(u_n)$ to be uniformly bounded for $\mu\in(0,\mu^*)$. Up to a subsequence, we may assume $u_n\rightharpoonup u$ in $E^q_b$, and this implies  $I(u)\leq R_1(\mu^*)$, for $\mu\in(0,\mu^*)$. Then,  
\eqref{Lr by I(u)} ensures the uniform boundedness of $\|u\|_{L^r_\eta}$.
Note that \eqref{truncation<0} also implies that $(u_n)$ is a $(PS)_c$ sequence for ${\Phi}_\mu$. 
Arguing as in Lemma \ref{lemmaPS-criticalpower+superscaledpower}, if  $(u_n)$ has no subsequence convergent in $E^q_b$, we see that
\begin{eqnarray*}
 c & \geq& \frac{2-\eta_1}{2(N-\eta_1)}S_{\eta_1}^{\frac{N-\eta_1}{2-\eta_1}}
   -\mu \tilde{d}\int_{\mathbb R^N}\frac{|u|^{r}}{|x|^{\eta}}dx
\end{eqnarray*}
for some positive constant $\tilde{d}$ independent of $u$ and $\mu$ (see \eqref{PSineq7} and recall that here $\ell(\eta,r)<\ell$).
Hence, the uniform boundedness of $\|u\|_{L^r_\eta}$ for $\mu\in(0,\mu^*)$ and the fact that $c<0$ gives us
\begin{equation*}
 \mu C\tilde{d}\geq \mu \tilde{d}\int_{\mathbb R^N}\frac{|u|^{r}}{|x|^{\eta}}dx \geq 
 \frac{2-\eta_1}{2(N-\eta_1)}S_{\eta_1}^{\frac{N-\eta_1}{2-\eta_1}}>0
\end{equation*}
which is not possible if  $\mu^*>0$ sufficiently small. By considering $\mu^*>0$ smaller than the first choice, if necessary, we finally conclude that $\tilde{\Phi}_\mu$ satisfies the $(PS)_c$ condition for any $c<0$ and $\mu\in(0,\mu^*)$. Now, let us pick $\mu\in(0,\mu^*)$ and estimate the levels $c_k$ defined in \eqref{ck sub near 0}. Since $\tilde{\Phi}_\mu$ is bounded, we have $c_k>-\infty$ for all $k\in\mathbb{N}$. On the other hand, since $I(u_t)=t^{\,\ell}<R_1(\mu)$ for all $u\in\mathcal{M}$, if $t>0$ is small, this gives us $\tilde\Phi_\mu(u_t)=\Phi_\mu(u_t)$. 
Recalling that $\ell(\eta_1,2^*_{\eta_1})>\ell>\ell(\eta,r)>0$, for all $u\in\mathcal{M}$ we get 
\begin{eqnarray*}
   \Phi_\mu(u_t)&=&t^{\ell}- \frac{\mu t^{\,\ell(\eta,r)}}{r}\int_{\mathbb{R}^N}\frac{|u|^{r}}{|x|^{\eta}}dx
   -\frac{t^{\,\ell(\eta_1,2^*_{\eta_1})}}{2^*_{\eta_1}}\int_{\mathbb{R}^N}\frac{|u|^{2^*_{\eta_1}}}{|x|^{\eta_1}}dx
   =-t^{\,\ell(\eta,r)}\left(\frac{\mu}{r}\int_{\mathbb{R}^N}\frac{|u|^{r}}{|x|^{\eta}}dx+o(1)\right),
\end{eqnarray*}
with $o(1)\to 0$ as $t\to 0^+$ uniformly in  $u\in\mathcal{M}$.
As in the proof of Theorem \ref{Th-single-subscaled} (see \eqref{Phi <0 near origin} and below), for 
$
M_t=\{u_t : u\in\tilde{\Psi}^{\lambda_k}\}
$
we see that $M_t\in\mathcal{F}_k$ and, by Lemma \ref{interpolation estimates}, we obtain $\sup_{u\in \tilde{\Psi}^{\lambda_k}}\Phi(u_t)<0$ for $t>0$ sufficiently small. 
Hence  
\begin{equation*}
    c_k=\inf_{M\in \mathcal{F}_k}\sup_{u\in M}\tilde\Phi_\mu(u)\leq \sup_{u\in M_t}\tilde\Phi_\mu(u)=\sup_{u\in \tilde{\Psi}^{\lambda_k}}\Phi_\mu(u_t)<0,
\end{equation*}
and by Proposition \ref{critical point result}, this concludes the proof. 
\end{proof-th}

We now arrive at the proof of the final main theorem of this work.

\begin{proof-th}
    \textbf{\textit{of Theorem \ref{Th-subscaled+scaled+critical}.}}\addcontentsline{toc}{subsubsection}{Proof of Theorem~\ref{Th-subscaled+scaled+critical}} 
    First, notice that the associated functional satisfies the $(PS)_c$ condition for all $c<c^*:=\frac{2-\eta_1}{2(N-\eta_1)}S_{\eta_1}^{\frac{N-\eta_1}{2-\eta_1}}$, as observed in Remark \ref{RemarkPSsubscaled+critical}.
Then, we shall apply Corollary \ref{Corollary 2.34 MP}.
For any $\lambda\in(\lambda_k,\infty)$ we may find $m\in\mathbb{N}$ such that $\lambda_{k+m-1}<\lambda\leq\lambda_{k+m}$. We fix $\tilde{\lambda}\in(\lambda_{k+m-1},\lambda)$ and  $\lambda'\in(\lambda,\infty)$. By Theorem \ref{lambdak}\ref{lambdak(iii)} we have
$$
i(\mathcal{M}\backslash\tilde\Psi_{\tilde\lambda})=k+m-1\geq k.
$$
Pick $A_0\subset \mathcal{M}\backslash\tilde\Psi_{\tilde\lambda}$ a compact symmetric subset of index $k+m-1$ (see the proof of Proposition 3.1 in \cite{De-Lan-2007}). For any $R>\rho>0$ to be chosen later, we set
$$
A=\{u_R:u\in A_0\}\quad\mbox{and}\quad X=\{u_t:u\in A_0,\,t\geq0\}.
$$
For all $u\in\mathcal{M}$ we have
\begin{eqnarray}\label{Phi sub-superscaled critical}
   \Phi(u_t)=t^{\,\ell}\left(1-\frac{\lambda}{\tilde\Psi(u)}\right)+ \frac{\mu\, t^{\,\ell(\eta,r)}}{r}\int_{\mathbb{R}^N}\frac{|u|^{r}}{|x|^{\eta}}dx-\frac{t^{\, \ell(\eta_1,2^*_{\eta_1})}}{2^*_{\eta_1}}\int_{\mathbb{R}^N}\frac{|u|^{2^*_{\eta_1}}}{|x|^{\eta_1}}dx.
\end{eqnarray}
Note that as $\lambda'>\lambda>\lambda_1$, the boundedness of $\|u\|_{L^{2^*_s}}$ for $u\in\mathcal{M}$, Theorem \ref{lambdak}\ref{lambdak(i)} and Lemma \ref{interpolation estimates} give us the estimate
\begin{eqnarray*}
   \Phi(u_t)\geq -t^{\,\ell}\left(\frac{\lambda}{\lambda_{1}}-1\right)+ c_1\mu t^{\,\ell(\eta,r)}-c_2t^{\, \ell(\eta_1,2^*_{\eta_1})}\quad \forall u\in\tilde\Psi^{\lambda'}, \forall t\geq0,
   \end{eqnarray*}
for some $c_1,c_2>0$ which do not depend on $\mu$. Also observe that 
\begin{eqnarray*}
   \Phi(u_t)\geq t^{\,\ell}\left(1-\frac{\lambda}{\lambda'}\right)-c_2t^{\, \ell(\eta_1,2^*_{\eta_1})}\quad \forall u\in\mathcal{M}\backslash \tilde\Psi^{\lambda'}, \forall t\geq0.
   \end{eqnarray*}
Recalling that $\, \ell(\eta_1,2^*_{\eta_1})>\ell>\ell(\eta,r)$ we conclude that
$$
\inf_{u\in\mathcal{M}}\Phi(u_\rho)>0
$$
provided $\rho>0$ is small enough. 
On the other hand, since $A_0\subset \mathcal{M}\backslash\tilde\Psi_{\tilde\lambda}$  we have $\tilde\Psi(u)< \tilde{\lambda}$ for any $u\in A_0$. Moreover, $\inf_{u\in A_0}\|u\|_{L^{2^*_{\eta_1}}_{\eta_1}}>0$, as $A_0$ is compact.
Hence, using the boundedness of $\|u\|_{L^r_{\eta}}$ for $u\in\mathcal{M}$, by \eqref{Phi sub-superscaled critical} we get
\begin{eqnarray*}
   \Phi(u_t)\leq -t^{\,\ell}\left(\frac{\lambda}{\tilde{\lambda}}-1\right)+ c'\mu t^{\,\ell(\eta,r)}-c_3t^{\,\ell(\eta_1,2^*_{\eta_1})}\quad \forall u\in A_0, \forall t\geq0,
   \end{eqnarray*}
for some $c',c_3>0$. Picking now $R>\rho$ sufficiently large, we obtain
$$
\sup_{u\in A}\Phi(u)=\sup_{u\in A_0}\Phi(u_R)\leq0.
$$
Note also that, for $u\in A_0$ and $t\geq0$, we have
\begin{align*}
   \Phi(u_t)\le&\  c'\mu t^{\,\ell(\eta,r)}-c_3t^{\,\ell(\eta_1,2^*_{\eta_1})}\\\leq&\  \big(\ell(\eta_1,2^*_{\eta_1})-\ell(\eta,r)\big)\left[\frac{\ell(\eta,r)}{c_3}\right]^\frac{\ell(\eta,r)}{\ell(\eta_1,2^*_{\eta_1})-\ell(\eta,r)}
   \left[\frac{c'\mu}{\ell(\eta_1,2^*_{\eta_1})}\right]^\frac{\ell(\eta_1,2^*_{\eta_1})}{\ell(\eta_1,2^*_{\eta_1})-\ell(\eta,r)}.
   \end{align*}
As a consequence, for $\mu\in(0,\mu^*)$ with $\mu^*$ sufficiently small, it holds that
$$
\sup_{u\in X}\Phi(u)=\sup_{u\in A_0,t\in[0,R]}\Phi(u_t)<\frac{2-\eta_1}{2(N-\eta_1)}S_{\eta_1}^{\frac{N-\eta_1}{2-\eta_1}},
$$
and therefore, by Corollary \ref{Corollary 2.34 MP} there exist $k+m-1(\geq k)$ pairs of critical points for $\Phi$, at positive energy levels. This concludes the proof.
\end{proof-th}
 
\section{Radial improvements}\label{sectionradial}

This section is devoted to exploring the improvements that arise when working in the 
radial subspace \(E^q_{b,\mathrm{rad}} \subset E^q_b\). When \(\eta > b\ge0\), no  improvement is obtained, and the embedding ranges coincide with those of the nonradial case, presented in Lemma~\ref{Lemma Embedding with weight}. In contrast, as we can see  in Lemma~\ref{Lemma Embedding radial}, the embeddings admit a strictly larger range of exponents in the radial setting when \(0 \le \eta < b\).  
 In particular, for \(N = 2\) we prove that one can work with singular weights \(|x|^{-\eta}\) even when \(\eta < b\), a feature that is not available in the 
nonradial setting.
Moreover, in $E^q_{b,\rm rad}$ we reach compactness of the embeddings in the case $\eta=b=0$ for $N\ge3$ and in the case $\eta=b\ge0$ for $N=2$. Notice that in these cases Lemma~\ref{Lemma Embedding with weight} provides a nontrivial interval for continuous embedding, without compactness.

The main tool leading to these improved ranges is provided by 
\cite[Theorem~18 and Proposition~26]{Mallick-Nguyen-2023}, where refined versions of the Caffarelli–Kohn–Nirenberg inequalities are established for radial smooth functions. After proving the radial embeddings, we list some existence results for radial solutions of \eqref{general problem} and compare them with those obtained in the nonradial case.

To denote this new interval of embeddings,  for $N\geq2$, $0\le\eta< N$ and $0\le b< 2<q$, we consider
\begin{equation*}
2^{\text{rad}}_{b,q,\eta}:=
\begin{cases}
   \vspace{.2cm}
\frac{q(N-\eta)}{N-b}&if \quad b\leq\eta,\\
 \frac{q(2N-2-\eta)+2(b-\eta)}{2N-2-b}&if \quad 0\le\eta<b .  
\end{cases}
\end{equation*}
 Notice that if  \(\eta \ge b\) then \(2^{\mathrm{rad}}_{b,q,\eta}=2^*_{b,q,\eta}\). In particular,  \(2^{\mathrm{rad}}_{b,q,b}=2^*_{b,q,b}=q\).  On the other hand, for \(0\le\eta < b\)  we obtain
    \[
       q< \frac{q(N-\eta)}{N-b}
        \;<\;
        2^{\mathrm{rad}}_{b,q,\eta} \;<\; 2^*_\eta
        \; \;( \; 2^{\mathrm{rad}}_{b,q,\eta}\;<\;
        2^*_{b,q,\eta}\quad\text{if}\quad N\ge3).
    \]
    We are able to prove the following:
\begin{lemma}\label{Lemma Embedding radial}
Suppose $N\ge 2$, let $0\le b<2<q<2^*_b$ and $0\le\eta\le b$.
   Then we have $E^q_{b,rad}\hookrightarrow L^r_\eta(\R^N)$ for 
 $$
   \begin{cases}
                \vspace{.2cm}
   
           r\in \left[2^{\rm{rad}}_{b,q,\eta},\infty\right)&\text{if}\quad N=2,\\
           
           \vspace{.2cm}
           
        r\in \left[2^{\rm{rad}}_{b,q,\eta},2^*_\eta\right]&\text{if}\quad N\ge 3,
    \end{cases}
    $$
     and for $\theta\in \left[0,1\right]$ such that $\frac{1}{r}=\theta\frac{N-2}{2(N-\eta)}+(1-\theta)\frac{N-b}{q(N-\eta)}$ it holds
    \[
\left(\int_{\R^N}\frac{|u|^r}{|x|^\eta} dx\right)^\frac{1}{r}\leq C \|\nabla u\|^\theta_2 \| |\cdot|^\frac{-b}{q} u \|_q^{1-\theta},\quad\forall u\in E^q_{b,rad}(\R^N).
\]
Moreover, all these embeddings are compact except for the endpoints of these intervals.
\end{lemma}
\begin{proof}
We are going to apply CKN inequality again, for $\gamma=-\eta/r$. If $\theta=0$, since $\gamma$ satisfies \eqref{parameters} we get  $\eta=b$ and $r=q\  (=2^{\rm{rad}}_{b,q,b})$ and so the inequality obviously holds. Let us consider $\theta>0$.
From the proof of Lemma \ref{Lemma Embedding with weight},  for
$N\ge3$ we get $\sigma=-\eta/(r\theta)+(1-\theta)b/(q\theta)\leq0$ for all $\theta\in(0,1]$ if $\eta= b$ and for all $\theta\in[\theta_1,1]$ when $0\le\eta<b$. 
Moreover, $\sigma\geq-1$ if $\theta=1$, since $\eta\leq2$. When $N=2$ and $\eta=b$ we get $\sigma=0$  for all $\theta\in(0,1]$.
Then \eqref{CKN ineq} holds for any $u\in C_c^\infty(\R^N)$, in particular for radial functions, for $r$ related to these $\theta$.

On the other hand, when $0\le\eta<b$, for $\theta\in(0,\theta_1)$ if $N\geq3$, or $\theta\in(0,1)$ if $N=2$, we have $\sigma>0$ so that \eqref{sigma's condition} is not satisfied. Nevertheless, for radial functions we can apply \cite[Theorem 18]{Mallick-Nguyen-2023} which guarantees that 
\eqref{CKN ineq} still holds if we assume \eqref{parameters} with $\theta>0$, \eqref{dimensional balance} and consider $\sigma\in(0,N-1]$. We verify that, for $\theta\in(0,1)$, one has
\[
N-1\geq \sigma=-\frac{\eta(N-2)}{2(N-\eta)}-\frac{(1-\theta)N(\eta-b)}{\theta q(N-\eta)} \quad\Longleftrightarrow \quad\theta\geq 
\frac{2(b-\eta)}{2(b-\eta)+q(2N-2-\eta)}=:\theta_2.
\]
Moreover, for $\theta=\theta_2$ we get $r=2^{\text{rad}}_{b,q,\eta}$.

To obtain compact embeddings away from the endpoints, one argues exactly as in the proof of Lemma~\ref{Lemma Embedding with weight}, but using \cite[Proposition~26]{Mallick-Nguyen-2023} instead of Proposition~\ref{propMallickNguyen}, which allows \( \sigma < N - 1\).
This completes the proof.
\end{proof}

All the existence and multiplicity results stated in  Theorems~\ref{Th-superscaled-subcritical}–\ref{Th-subscaled+scaled+critical},
which were formulated under the admissibility condition of Definition~\ref{defscaled} and rely only on the weighted Sobolev embeddings from Lemma~\ref{Lemma Embedding with weight}, remain valid, with exactly the same proofs, when restricted to the radial subspace $E^q_{b,\mathrm{rad}}$. In case $0\le\eta\le b$, some of them can be improved provided that the admissibility requirement on $(\eta,r)$ is replaced by the alternate definition: 

\begin{definition}\label{defscaled-radial}
Assume \eqref{basicparametershypothesis}.
We say that  the pair \((\eta,r)\) is \emph{admissible} if
\[
    0 \le \eta < \frac{N+2}{2}
    \quad \text{and}\quad
    r \in (\,2^{\rm{rad}}_{b,q,\eta},\; 2^*_{\eta}\,),\quad\text{with}\quad r> 1.
\]
 In addition, when \(\eta \le b\) we allow \(r = 2^{\rm{rad}}_{b,q,\eta}\), and if 
\(\eta \le 2\) and \(N \ge 3\) we include the second endpoint \(r = 2^*_\eta\).
  The classification of these admissible pairs as scaled, subscaled and superscaled follows the same one as in Definition \ref{defscaled}.
\end{definition}

We use this new definition on both theorems below. The first one comprises the results of Theorems \ref{Th-superscaled-subcritical} and \ref{Th-superscaled-subcritical-multi}.

\begin{theorem}\label{Th-superscaled-subcritical-radial}
Assume \eqref{basicparametershypothesis}. 
Suppose $f(x,t)=\lambda|x|^{-a}|t|^{p-2}t+|x|^{-\eta}h(t)$, with $\lambda\in\R$, $0\le\eta<N-q(N-2)/2$ and $h=H'$ for $H:\R\to\R$ a $C^1$ function such that
 \begin{align*}
 |h(t)|\leq c_1|t|^{r_1-1}+c_2|t|^{r_2-1}
 \quad\text{and}\quad
 c |t|^r\leq r H(t)\leq h(t)t,\quad\forall t\in\R,
 \end{align*}
for some constants $c,c_1,c_2>0$, and superscaled pairs $(\eta,r_1)$, $(\eta,r_2)$ and $(\eta,r)$, with $2^{\rm{rad}}_{b,q,\eta}<r_1,r_2<2^*_\eta$ and $q<r<2^*_\eta$.
Then:
\begin{description}
    \item[i)] If $\lambda$ is not an eigenvalue of \eqref{NE problem}, then problem \eqref{general problem} has a solution with positive energy.
\item[ii)] If $\lambda\in\mathbb{R}$ then problem \eqref{general problem} has a nontrivial radial solution. 
\item[iii)] If $\lambda\in\mathbb{R}$ and $h$ is an odd function, then  problem \eqref{general problem} has a sequence $(u_k)$ of nontrivial radial solutions satisfying $\Phi(u_k)\nearrow\infty$.
\end{description}
\end{theorem}

The next result describes the improvements that can be obtained in Theorems~\ref{th cri-super-mult} and \ref{th cri any lambda} and Corollary \ref{corollary-th} in the radial setting.

\begin{theorem}\label{th cri-super-mult-radial}
Assume $N\geq3$ and \eqref{basicparametershypothesis}. 
Consider $f$ as in \eqref{f critical r2}, with $\mu \ge 0$, $\eta_1$ satisfying \eqref{eta1 small} and the superscaled pair $(\eta_2,r_2)$  satisfying one of the following conditions:
\begin{align*}
&\eta_2\in[0,\eta_1)\quad \text{and}\quad 2^*_{\eta_1}\leq r_2< 2^*_{\eta_2}\quad\text{or}\\
&\eta_2\in[\eta_1,2)\, \,(\eta_2=\eta_1\,\, \text{only if} \,\,\eta_1>b)\quad  \text{and} \quad r_2<2^*_{\eta_2},
\quad\text{with}\,\,
r_2> q+(2^*_{\eta_1}-q)\frac{(b-\eta_2)}{b-\eta_1}
\quad\text{if}\quad\eta_2< b. 
\end{align*} 
The following results hold:
\begin{itemize}

\item[i)] If for some $k,m\in\mathbb{N}$ it holds that
\[
\lambda_k = \lambda_{k+1} = \dots = \lambda_{k+m-1} < \lambda_{k+m},
\]
then there is a number $\delta_k > 0$ such that problem~\eqref{general problem} possesses at least $m$ distinct pairs of nontrivial radial solutions with positive energy whenever
\(\lambda \in \bigl(\lambda_k - \delta_k,\; \lambda_k\bigr).\) 

\item [ii)] For any arbitrary $\lambda\in\mathbb R$ and  $m \in\mathbb{N}$, there exists  $\mu_m>0$ such that, for all $\mu>\mu_m$, problem~\eqref{general problem} admits at least $m$ distinct pairs of nontrivial solutions with positive energy. 
\end{itemize}
\end{theorem}

Another result that benefits from the radial setting is Theorem~\ref{Th-subscased+scaled+superscaled1}. 
There, the admissible pair $(\eta_1,r_1)$ is required to satisfy $\ell(\eta_1,r_1)>\ell(\eta,r)$ when $\lambda\le 0$, or $\ell(\eta_1,r_1)>\ell(a,p)$ when $\lambda>0$. 
In both cases, the admissible range for $r_1$ enlarges when $0\le\eta<b$, if Definition~\ref{defscaled-radial} is used in place of Definition~\ref{defscaled}. 
The remaining results (namely, Theorems~\ref{Th-single-subscaled}, \ref{Th-subscased+scaled}, \ref{Th-subscaled+critical}, and \ref{Th-subscaled+scaled+critical}) either involve subscaled pairs, which never occur when $\eta\le b$, or require critical exponents, and therefore do not benefit from the radial improvement.

\medskip

\section{Appendix}\label{sectionappendix}

In this section, we prove the density of smooth functions in the functional spaces that make up the framework of this work.

\begin{proposition}\label{Density}
Consider $N\geq2$, $b\in[0,N)$ and $q\geq1$. If $N=2$ assume also that $q\geq b$. Then $C^\infty_c(\R^N)$ is dense in $E^q_b$.
\end{proposition}
\begin{proof}
Consider $\varphi\in C^\infty_c(\R^N)$ such that $0\leq\varphi(x)\leq1$, $\|\nabla\varphi\|_\infty\leq2$, $\varphi\equiv1$ in $B_1$ and $\varphi\equiv0$ in $B_2^c$. Then define the sequence $\varphi_k(x)=\varphi(x/k)$. 
Consider also the sequence of functions $T_k:\R\to\R$ given by
\[
T_k(s)=\begin{cases}
    s&\text{if}\quad |s|<k\\
    k\frac{s}{|s|}&\text{if}\quad |s|\ge k.
\end{cases}
\]
For a fixed $u\in E^q_b$ we denote $u_k:=\varphi_k T_k(u)$. It is easily seen that $u_k\to u$
in $ L^q_b(\R^N)$. Moreover,
\begin{equation}\label{conv L2}
\left\|\frac{\partial u_k}{\partial x_i}-\frac{\partial u}{\partial x_i}\right\|_{L^2}\leq 
\left\|\frac{\partial \varphi_k}{\partial x_i}T_k(u)\right\|_{L^2}+
\left\|\left(\frac{\partial (T_k(u))}{\partial x_i}-\frac{\partial u}{\partial x_i}\right)\varphi_k\right\|_{L^2}+
\left\|\frac{\partial u}{\partial x_i}(1-\varphi_k)\right\|_{L^2}.
\end{equation}
Since $\varphi_k(x)\to1$ and $\frac{\partial (T_k(u))}{\partial x_i}(x)\to \frac{\partial u}{\partial x_i}(x)$ a.e. in $\R^N$ and $\frac{\partial u}{\partial x_i}\in L^2(\R^N)$, these last two norms go to $0$ as $k\to\infty$. Moreover, for $i\in\{1,\cdots,N\}$, it holds
\[
\int_{\R^N}\left|\frac{\partial \varphi_k}{\partial x_i}T_k(u)\right|^2dx=\frac1{k^2}\int_{\R^N}\left|\frac{\partial \varphi}{\partial x_i}(x/k)T_k(u(x))\right|^2dx.
\]
Here we analyze the cases $N=2$ or grater. When $N\ge3$ we have $E^q_b\hookrightarrow L^{2^*}(\R^N)$. Then, using $|T_k(s)|\leq|s|$ and $supp(\frac{\partial \varphi_k}{\partial x_i})\subset B_{2k}\setminus B_k$, and the Holder's inequality we get
\[
\int_{\R^N}\left|\frac{\partial \varphi_k}{\partial x_i}T_k(u)\right|^2dx\leq
\|\nabla\varphi\|_{L^N}^2\left(\int_{B_{2k}\setminus B_k}|u|^{2^*}dx\right)^\frac{2}{2^*}
=o(1),\quad\text{as}\,\,k\to\infty.
\]
Assume now  $N=2$ and $b\in[0,2)$. If $\max\{b,1\}\leq q\leq 2$, since $|T_k(s)|\leq\min\{|s|,k\}$ and $u\in L^q_b(\R^2)$,  we have
\begin{align*}
\int_{\R^2}\left|\frac{\partial \varphi_k}{\partial x_i}T_k(u)\right|^2dx
&\leq\frac{k^{2-q}}{k^2}\|\nabla\varphi\|_\infty^2\int_{B_{2k}\setminus B_k}|u(x)|^qdx
\leq \frac{2^{b+2}}{k^{q-b}}\int_{B_{2k}\setminus B_k}\frac{|u(x)|^q}{|x|^b}dx=o(1).
\end{align*}
In case $q>2$, using Holder's inequality again we get
\begin{align*}
\int_{\R^2}\left|\frac{\partial \varphi_k}{\partial x_i}T_k(u)\right|^2dx
&\leq k^{-\frac{2(2-b)}{q}}\|\nabla\varphi\|_{L^\frac{2q}{q-2}}^2\|u\|_{L^q_b}^2=o(1).
\end{align*}
Thus, for $N=2$, $b\in[0,2)$ and $ q\ge\max\{b,1\}$ or $N\ge3$, $b\in[0,N)$ and $q\geq1$,  from \eqref{conv L2} we obtain $\nabla u_k\to \nabla u$ in $(L^2(\R^N))^N$. Therefore, $u_k\to u$ in $E^q_b$.

Now, let us deal with smooth functions with compact support. Given $\varepsilon>0$, fix $k_0$ such that $u_0:=u_{k_0}$ satisfies
\begin{align}\label{norma geral}
\|u_{0}-u\|<\frac{\varepsilon}{2}.
\end{align}
Since $b\in[0,N)$ we have $w(x):=|x|^{-b}$ a function in $L^1_{loc}(\R^N)$. So, $w\chi_{\{w\leq n^q\}}\to w$ in $L^1(B_{2k_0})$ as $n\to\infty$ and we fix $n_0\in\mathbb{N}$ such that
\begin{align}\label{(i)}
    \|w\chi_{\{w\leq n^q_0\}}- w\|_{L^1(B_{2k_0})}<\frac{\varepsilon}{10k_0}.
\end{align}
Also, choose $\sigma>0$ such that
\begin{align}\label{(ii)}
    \| w\|_{L^1(B_{2k_0+\sigma}\setminus B_{2k_0})}<\frac{\varepsilon}{10k_0}.
\end{align}
At this point, we fix a function $\rho\in C^\infty_c(\R^N)$ such that $\rho\geq0$, $supp(\rho)\subset B_1$ and $\|\rho\|_{L^1}=1$, and then define $\rho_m(x)=m^N\rho(mx)$. 
We know that $\rho_m*v\to v$ in $L^2(\R^N)$ for all $v\in L^2(\R^N)$. In particular, it holds for $v=\frac{\partial u_0}{\partial {x_i}}$. We  also have $\rho_m*u_0\to u_0$ in $L^q(\R^N)$, as $m\to\infty$. Then we fix $m_0\in\mathbb{N}$ sufficiently large so that 
\begin{align}\label{(iii)2}
    \left\|\rho_{m_0}*\frac{\partial u_0}{\partial {x_i}}-\frac{\partial u_0}{\partial {x_i}}\right\|_{L^2}<\frac{\varepsilon}{10N}
\end{align}
for all $i\in\{1,\cdots,N\}$, and,  for $\psi_0:=\rho_{m_0}*u_0$ we have $supp(\psi_0)\subset B_{2k_0+\sigma}$ and
\begin{align}\label{(iii)}
    \|\psi_0-u_0\|_{L^q}<\frac{\varepsilon}{10n_0}.
\end{align}
Notice that, by the definition of $u_0$ and $T_k$, we get
\[
|\psi_0(x)|\leq\int_{\R^N}\rho_{m_0}(x-y)|u_0(y)|dy\leq k_0\int_{\R^N}\rho_{m_0}(x-y)dy=k_0.
\]
Thus, using \eqref{(iii)} we see that
\begin{align}\label{(iv)}
    \|\psi_0\chi_{B_{2k_0}}-u_0\|_{L^q_b(\{w\leq n_0^q\})}=\left(\int_{B_{2k_0}\cap\{w\leq n_0^q\}}|\psi_0-u_0|^qw(x)dx\right)^\frac{1}{q}
    \leq n_0\|\psi_0-u_0\|_{L^q(B_{2k_0})}<\frac{\varepsilon}{10}.
\end{align}
Also, using \eqref{(ii)} we have
\begin{align}\label{(v)}
    \|\psi_0-\psi_0\chi_{B_{2k_0}}\|_{L^q_b}=\left(\int_{\R^N}|\psi_0(1-\chi_{B_{2k_0}})|^qw(x)dx\right)^\frac{1}{q}
    \leq k_0\|w\|_{L^1(B_{2k_0+\sigma}\setminus B_{2k_0})}<\frac{\varepsilon}{10}.
\end{align}
So, due to \eqref{(iv)}, \eqref{(v)} and \eqref{(i)} we obtain
\begin{align}\label{(vi)}
    \|\psi_0-u_0\|_{L^q_b}&\leq \|\psi_0-\psi_0\chi_{B_{2k_0}}\|_{L^q_b}+  \|\psi_0\chi_{B_{2k_0}}-u_0\|_{L^q_b(\{w\leq n_0^q\})}+
    \|(\psi_0\chi_{B_{2k_0}}-u_0)(1-\chi_{\{w\leq n_0^q\}})\|_{L^q_b}\nonumber\\
    &< \frac{\varepsilon}{10}+\frac{\varepsilon}{10}+ 2k_0 \|(1-\chi_{\{w\leq n_0^q\}})w\|_{L^1(B_{2k_0})} <\frac{2\varepsilon}5.
\end{align}
Then, due to \eqref{(iii)2} and \eqref{(vi)} it holds
\[ 
 \|\psi_0-u_0\|_{E^q_b}\leq  \|\psi_0-u_0\|_{L^q_b}+\sum_{i=1}^N\left\|\rho_{m_0}*\frac{\partial u_0}{\partial x_i}-\frac{\partial u_0}{\partial x_i}\right\|_{L^2}< \frac{\varepsilon}{2}.
\]
Finally, using \eqref{norma geral} we obtain
\[ 
 \|\psi_0-u\|_{E^q_b}\leq  \|\psi_0-u_0\|_{E^q_b}+ \|u_0-u\|_{E^q_b}< {\varepsilon},
\]
and conclude this proof.
\end{proof}

\section*{Acknowledgements}
Elisandra Gloss and Bruno Ribeiro would like to thank Conselho Nacional de Desenvolvimento Científico e Tecnol\'ogico (CNPq) for the support; in particular, for the grants 201454/2024-6 and 443594/2023-6 (E.G.),  grants 314111/2023-9, 443594/2023-6, 201452/2024-3 (B.R.).

\section*{Conflict of interest statement}

The authors have no relevant financial or non-financial interests to disclose.



\end{document}